\documentclass[11pt, a4paper]{article}
\usepackage{fullpage}
\usepackage{listings}
\usepackage{latexsym, amsfonts, amssymb, amsmath, amsthm, mathrsfs, mathtools, setspace, graphics, graphicx, bbm, float, bigints}
\usepackage{enumerate, array}
\usepackage{caption}
\usepackage{indentfirst}
\usepackage{framed}
\usepackage{yhmath}
\usepackage[nodayofweek,level]{datetime}
\usepackage{overpic}
\usepackage{subcaption} 

\usepackage{lipsum}        

\usepackage{stmaryrd}

\RequirePackage[ruled,vlined]{algorithm2e}
\RequirePackage{longtable}

\newcolumntype{L}{>{\raggedright\arraybackslash}p{6cm}} 
\newcolumntype{R}{>{\raggedleft\arraybackslash}p{8cm}}  
\newcolumntype{B}{|>{\bfseries\arraybackslash}p{3cm}|}  

\RequirePackage{tikz}
\usetikzlibrary{calc,trees,positioning,arrows,chains,shapes.geometric,%
    decorations.pathreplacing,decorations.pathmorphing,shapes,%
    matrix,shapes.symbols}
\allowdisplaybreaks    

\usepackage{geometry}
\usepackage{multirow}



\usepackage{lscape}

\title{Particle method for the numerical simulation of the path-dependent McKean-Vlasov equation}

\usepackage[hidelinks,urlcolor=black]{hyperref}
\newcommand{\footremember}[2]{
   \footnote{#2}
    \newcounter{#1}
    \setcounter{#1}{\value{footnote}}
}
 
\makeatletter
\newcommand{\leqnomode}{\tagsleft@true}
\newcommand{\reqnomode}{\tagsleft@false}
\makeatother

\author{%
Armand Bernou\footremember{a}{\small University Lyon 1, ISFA, LSAF (EA 2429), Lyon, France. E-mail address: armand.bernou@univ-lyon1.fr} 
  \and Yating Liu \footremember{b}{\small CEREMADE, CNRS, UMR 7534, Universit\'e  Paris-Dauphine, PSL University, 75016 Paris, France. E-mail address: liu@ceremade.dauphine.fr}%
}
  
\numberwithin{equation}{section}
\newtheorem{thm}{Theorem}[section]
\newtheorem{lem}[thm]{Lemma}
\newtheorem{prop}[thm]{Proposition}
\newtheorem{cor}[thm]{Corollary}
\newtheorem{defn}[thm]{Definition}
\newtheorem{manualtheoreminner}{Assumption}
\newenvironment{manualtheorem}[1]{%
  \renewcommand\themanualtheoreminner{#1}%
  \manualtheoreminner
}{\endmanualtheoreminner}

\theoremstyle{remark}
\newtheorem{rem}[thm]{Remark}

\newcommand{\widesim}[2][1.5]{
  \mathrel{\overset{#2}{\scalebox{#1}[1]{$\sim$}}}
}

\newcommand{\vertiii}[1]{{\left\vert\kern-0.25ex\left\vert\kern-0.25ex\left\vert #1 
    \right\vert\kern-0.25ex\right\vert\kern-0.25ex\right\vert}}
\newcommand{\vertii}[1]{\left\Vert #1\right\Vert}

\newcommand\independent{\protect\mathpalette{\protect\independenT}{\perp}}
\def\independenT#1#2{\mathrel{\rlap{$#1#2$}\mkern2mu{#1#2}}}

\newcommand{\PPC}{\mathcal{P}_{p}\big(\mathcal{C}([0, T], \mathbb{R}^{d})\big)}
\newcommand{\CPP}{\mathcal{C}\big([0, T], \mathcal{P}_{p}(\mathbb{R}^{d})\big)}
\newcommand{\CRD}{\mathcal{C}\big([0, T], \mathbb{R}^{d}\big)}
\newcommand{\PPRD}{\mathcal{P}_{p}(\mathbb{R}^{d})}
\newcommand{\R}{\mathbb{R}}
\newcommand{\PP}{\mathbb{P}}
\newcommand{\RD}{\mathbb{R}^{d}}
\newcommand{\RR}{\mathbb{R}}

\newcommand{\EE}{\mathbb{E}\,}

\newcommand{\calP}{\mathcal{P}}
\newcommand{\calW}{\mathcal{W}}

\renewcommand{\sp}{{\footnotesize \hbox{sup}}}
\newcommand{\ddr}{\mathrm{d}}
\newcommand{\cws}{{\cdot \wedge s}}

\renewcommand{\d}{\mathrm{d}}

\newcommand{\dd}{\color{black}}

\newcommand{\iid}{\widesim{\,\mathrm{i.i.d.}\,}}

\newcommand{\llN}{\llbracket N \rrbracket}
\newcommand{\llM}{\llbracket M \rrbracket}

\usepackage[normalem]{ulem}
\usepackage[textwidth=1.8cm, textsize=tiny]{todonotes}

\begin{document}
\maketitle 

\begin{abstract}
We present the particle method for simulating the solution to the path-dependent McKean-Vlasov equation, in which both the drift and the diffusion coefficients depend on the whole trajectory of the process up to the current time $t$, as well as on the corresponding marginal distributions. Our main contribution is the derivation of explicit convergence rates that capture the interplay between time and space discretization. To control the uniform-in-time convergence of empirical measures in Wasserstein distance on a fixed interval, we develop two approaches: one based on Fournier-Guillin estimates, the other extending Horowitz-Karandikar's method to general $p \ge 2$. We then compare the respective regimes of applicability. Numerical simulations of a generalized Ornstein-Uhlenbeck process with memory provide evidence for the accuracy of our bounds. We also apply our method to an extension of the Jansen-Rit mean-field model for neural masses.
\end{abstract}


\setcounter{tocdepth}{1}

\bigskip
\section{Introduction}

We consider a filtered probability space $(\Omega,\mathcal{F}, (\mathcal{F}_t)_{t \ge 0}, \mathbb{P})$ satisfying the usual condition\footnote{The usual condition means that $\mathcal{F}_0$ contains all $\mathbb{P}$-null sets and the filtration is right-continuous, i.e., $\mathcal{F}_t=\mathcal{F}_{t+}\coloneqq \cap_{s>t}\mathcal{F}_s$.} and  an $(\mathcal{F}_t)$-standard Brownian motion $(B_t)_{t\ge 0}$ valued in $\RR^{q}$, $q \in \mathbb{N}^*$. Let $T > 0$ be the fixed  time horizon and let $\mathbb{M}_{d,q}(\R)$ denote the space of matrices of size $d \times q$, $d \in \mathbb{N}^*$, equipped with the operator norm $\vertiii{\cdot}$ defined by $\vertiii{A} := \sup_{z \in \R^q, |z| \le 1} \big|Az\big|$. We write $\mathcal{C}([0,T],S)$ for the set of continuous maps from $[0,T]$ to some Polish space $S$ endowed with the distance $d_S$, and, for $p \ge 1$, we write $\calP_p(S)$ for the set of probability distributions on $S$ admitting a finite moment of order $p$ equipped with the Wasserstein distance (see \eqref{eq:def_Wasserstein} below). Moreover, 
for $\alpha=(\alpha_t)_{t\in[0,T]}\in\CRD$, $(\nu_t)_{t\in[0,T]}\in\CPP$ and for a fixed $t_0\in[0,T]$, we define $\alpha_{\cdot \wedge t_0}=(\alpha_{t \wedge t_0})_{t\in[0,T]}$  and $\nu_{\cdot \wedge t_0}=(\nu_{t \wedge t_0})_{t\in[0,T]}$  by  
\begin{equation}\label{infprocess}
\alpha_{t \wedge t_0}\: \coloneqq \:\begin{cases}\alpha_t\quad \:\,\text{if}\quad t   \in[0,t_0], \\ \alpha_{t_0}\quad \text{if}\quad  t\in(t_0, T],\end{cases}\quad \text{and}\quad \nu_{t \wedge t_0}\: \coloneqq \:\begin{cases}\nu_t\quad \:\,\text{if}\quad t   \in[0,t_0], \\ \nu_{t_0}\quad \text{if}\quad   t\in(t_0, T].\end{cases}
\end{equation}
It is obvious that $\alpha_{\cdot \wedge t_0}\in\CRD$ and $\nu_{\cdot \wedge t_0}\in\CPP$.

In this paper, we consider the following path-dependent McKean-Vlasov equation 
\begin{equation}\label{eq:path-dependent_McKean}
X_t=X_0+\int_{0}^{t}\,b(s, X_{\cdot\wedge s}, \mu_{\cdot\wedge s}) \, \d s+\int_{0}^{t}\,\sigma(s, X_{\cdot\wedge s}, \mu_{\cdot\wedge s})\, \d B_s, \qquad t \in [0,T],
\end{equation}
where $X_0:  (\Omega, \mathcal{F}, \mathbb{P})\rightarrow \big(\RD, \mathcal{B}(\RD)\big)$  is a random variable independent of $(B_t)_{t\in[0,T]}$,  the coefficient functions $b$ and $\sigma$  are measurable functions defined on $[0, T]\times \mathcal{C}\big([0, T], \RD\big)\times \mathcal{C}\big([0, T], \mathcal{P}_p(\RD)\big)$ and respectively valued in $\RD$ and in $\mathbb{M}_{d, q}(\RR)$, and $\mu_{\cdot \wedge t}$ denotes the marginal distributions of the process $X_{\cdot \wedge t}$, that is, for every $s\in[0,T]$, $\mu_{s \wedge t}=\mathbb{P}\circ X_{s \wedge t}^{-1}$ .

In \eqref{eq:path-dependent_McKean}, the arguments $X_{\cdot \wedge t}$ and $\mu_{\cdot \wedge t}$ in the coefficients $b$ and $\sigma$ keep track of the whole trajectory of $X_\cdot$ and its marginal distribution $\mu_\cdot$ between $0$ and $t > 0$, which can be seen as a generalization of the standard McKean-Vlasov equation 
\begin{align}\label{eq:McKean_standard}
X_t = X_0 + \int_0^t b(s, X_s, \mu_s)\d s+ \int_0^t \sigma(s, X_s, \mu_s)\d B_s
\end{align}
first introduced by McKean in \cite{mckean1967propagation} as a stochastic
model naturally associated to a class of non-linear PDEs. See also \cite{Sznitman_1991, Chaintron_2022a, Chaintron_2022b} for a systematic presentation of the  standard McKean-Vlasov equation, including the notion of propagation of chaos. 

This paper aims to study the convergence rate of a numerical method for simulating the solution to \eqref{eq:path-dependent_McKean}. The construction of the numerical scheme comprises two essential components:  a temporal discretization over the interval $[0, T]$ by using an interpolated Euler scheme, and a spatial discretization across $\mathbb{R}^d$ using a discrete particle system. 

\noindent\textit{(a) Temporal discretization by an interpolated Euler scheme.}

In the following definition, $M \in \mathbb{N}^*$ should be thought of as the temporal discretization number, while $h\coloneqq \frac{T}{M}$ is the time step. For every $m = 0,\dots, M$, we set $t_m = m h$.  To simplify the notations, we will write $x_{0:m}\coloneqq(x_0, ..., x_m)$, $\mu_{0:m}\coloneqq(\mu_0, ..., \mu_m)$. Our interpolated Euler scheme uses the following interpolator. 
\begin{defn}[Interpolator]\label{definterpolator}
	\begin{enumerate}[$(a)$]
		\item For every $m=1, \dots, M$, we define a piecewise affine interpolator $i_m$ on $m+1$ points in $\RD$ by 
		\begin{equation}\label{definterp1}
			x_{0:m}\in(\RD)^{m+1}\longmapsto i_m(x_{0:m})=(\bar x_t)_{t\in[0,T]}\in\CRD,
		\end{equation}
		where for every $t\in[0,T]$, $\bar x_t$ is defined by 
		\begin{align}
			&\forall \,k=0, ..., m-1, \quad \forall \,t\in[t_k, t_{k+1}),\quad \bar x_t =\frac{1}{h}(t_{k+1}-t)\,  x_k+\frac{1}{h}(t-t_k) \, x_{k+1},\nonumber\\
			&\forall \,t\in[t_m, T], \quad \bar x_t=x_m.\nonumber
		\end{align}
		By convention, we define, for every $t \in [0,T]$, $i_0(x_0)_t := x_0$. 
	\item Let $p\geq 1$. For every $m=1, ..., M$, we define a piecewise affine interpolator for $m+1$ probability measures in $\PPRD$, still denoted by $i_m$ with a slight abuse of notation, by
\begin{equation}\label{definterp2}
	\mu_{0:m}\in(\PPRD)^{m+1}\longmapsto i_m(\mu_{0:m})=(\bar \mu_t)_{t\in[0,T]}\in\CPP,
\end{equation}
where for every $t\in[0,T]$, $\bar \mu_t$ is defined by 
\begin{align}\label{definterp2bis}
	&\forall \,k=0, ..., m-1, \quad \forall \,t\in[t_k, t_{k+1}),\quad \bar \mu_t =\frac{1}{h}(t_{k+1}-t)\, \mu_k+\frac{1}{h}(t-t_k) \, \mu_{k+1},\nonumber\\
	&\forall \,t\in[t_m, T], \quad \bar \mu_t=\mu_m.
\end{align}
\end{enumerate}
By convention, we define, for every $t \in [0,T]$, $i_0(\mu_0)_t := \mu_0$. 
\end{defn}

\smallskip
The well-posedness of the interpolator $i_m$ is proved in Lemma \ref{combcovprop} below. With this at hand, we define our interpolated Euler scheme in which we use the short-hand notation $Y_{t_0:t_m}$ (\textit{respectively}, $\nu_{t_0:t_m}$) to denote 
$(Y_{t_0}, \dots, Y_{t_m})$ (\textit{resp.} $(\nu_{t_0}, \dots, \nu_{t_{m}})$).

\smallskip
\begin{defn}
	\label{def:discretization_scheme}
	Let $M\in\mathbb{N}^{*}$,  $h=\frac{T}{M}$. For every $m=0,..., M$, we set $t_m=m h$. For the same Brownian motion $(B_t)_{t \in [0,T]}$ and random vector $X_0$ as in \eqref{eq:path-dependent_McKean}, the interpolated scheme $(\widetilde X_{t_m}^h)_{0\leq m\leq M}$ of the path-dependent McKean-Vlasov equation \eqref{eq:path-dependent_McKean} is defined as follows :
\begin{enumerate}
\item[1.] $\widetilde X^h_0 = X_0$; 
\item[2.] for all $m \in \{0, \dots, M-1\}$, 
\begin{equation}\label{eq:discretescheme}
\widetilde X_{t_{m+1}}^h = \widetilde X_{t_m}^h + h \, b_m(t_m, \widetilde X_{t_0:t_m}^h, \widetilde \mu_{t_0:t_m}^h) + \sqrt{h} \, \sigma_m (t_m, \widetilde X_{t_0:t_m}^h, \widetilde \mu_{t_0:t_m}^h) Z_{m+1}, 
\end{equation}
\end{enumerate}
where, for $k \in \{0, \dots, M\}$, $\widetilde \mu^h_{t_k}$ is the probability distribution of $\widetilde X^h_{t_k}$, where, for $m = 0, \dots, M-1$, $Z_{m+1} = \frac{1}{\sqrt{h}} (B_{t_{m + 1}} - B_{t_m})\widesim{\,\mathrm{i.i.d.}\,}\mathcal{N}(0,\mathrm{I}_q)$, and where 
the applications $b_m, \sigma_m$ are defined on $[0,T]\times (\RD)^{m+1}\times \big(\PPRD\big)^{m+1}$ and respectively valued in $\RD$ and $\mathbb{M}_{d,q}(\RR)$, with
\small
\begin{align}\label{defbmsigmam}
&\forall\, t\in [0, T], \;x_{0:m}\in(\RD)^{m+1},\; \mu_{0:m}\in \big(\PPRD\big)^{m+1}, \nonumber\\
&\: b_m(t, x_{0:m}, \mu_{0:m})\!\coloneqq b\big(t, i_m(x_{0:m}), i_m(\mu_{0:m})\big), \quad \sigma_m(t, x_{0:m}, \mu_{0:m})\!\coloneqq \ \sigma\big(t, i_m(x_{0:m}), i_m(\mu_{0:m})\big).
\end{align}\normalsize
Moreover, we also define the continuous extension process $(\widetilde X^{h}_t)_{t \in [0,T]}$ from \eqref{eq:discretescheme} by setting, for all $t \in (t_m, t_{m+1}]$,
\begin{align}\label{eq:def_continuous_2} 
\widetilde X^{h}_t &= \widetilde X^{h}_{t_m} + (t-t_m) \, b_m(t_m, \widetilde X^{h}_{t_0:t_m}, \widetilde \mu^{h}_{t_0:t_m})   +\,  \sigma_m \, (t_m, \widetilde X^{h}_{t_0:t_m}, \widetilde \mu^{h}_{t_0:t_m})(B_t - B_{t_m}). 
\end{align}
\end{defn}

\begin{rem}\label{rem:dis-input}
The applications $b_m$ and $\sigma_m$ defined in \eqref{defbmsigmam} process discrete inputs, often facilitating computations from a numerical perspective. For instance, if
\begin{equation}
	\label{eq:b_integral}
	b\big(t, (X_{s})_{s\in[0,T]}, (\mu_s)_{s\in[0,T]}\big)\coloneqq\int_{0}^{t} \EE [\phi (X_s)] \d s  = \int_0^t \left(\int_{\RD}\phi(x)\mu_s(\d x)\right)\d s 
\end{equation} 
with a bounded function $\phi$, then,  by definition of $b_m$,   
\begin{equation}
	\label{eq:b_m_sum}
	b_m(t_m, \widetilde{X}^h_{t_{0}:t_{m}}, \widetilde{\mu}^h_{t_{0}:t_{m}})=\frac{h}{2}\Big(\EE \big[\phi (\widetilde{X}^h_{t_{0}}) \big]+\EE \big[\phi (\widetilde{X}^h_{t_{m}}) \big]\Big)+h \sum_{k=1}^{m-1}\EE\big[\phi (\widetilde{X}^h_{t_{k}})\big].
\end{equation}
Clearly, the numerical computation of an integral quantity, as in \eqref{eq:b_integral}, is  more demanding than the handling of sums, as in \eqref{eq:b_m_sum}.
\end{rem}
\smallskip
\noindent\textit{(b) Spatial discretization by a particle system.}

The scheme defined in \eqref{eq:discretescheme} is not directly implementable due to the term $\widetilde \mu_{t_0:t_m}^h$ in the coefficient functions. To overcome this limitation, we enhance \eqref{eq:discretescheme} by incorporating a particle system,  in the spirit of \cite{Talay_1996, Bossy_1997, Antonelli_2002, liu2022particle}, thereby transforming it into a numerically implementable scheme.
To simplify the notation, for $N \in\mathbb{N}^{*}$, we use $\llbracket N \rrbracket$ to denote the set $\{0,\dots,N\}$ and $\llbracket N \rrbracket^*$ for the set $\{1,\dots,N\}$. 


\begin{defn}[Particle method]\label{def:particle_method}
Let $N \in \mathbb{N}^*$. Consider $N$ standard independent Brownian motions $(B^1_t, ..., B^N_t)_{t\in[0,T]}$. For every  $n \in \llN^*$  and for every $m \in \llbracket M-1 \rrbracket$, let $Z^n_{m+1}$ be given by $Z_{m+1}^n\coloneqq (B_{t_{m+1}}^n-B_{t_{m}}^n)/\sqrt{h}$. We define a discrete $N$-particle system $(\widetilde{X}_{t_m}^{1, N,h}, ..., \widetilde{X}_{t_m}^{N, N,h})_{\,0\leq m\leq M}$ as follows :  
\begin{align}\label{eq:particlesystem}
&\widetilde{X}_{t_{m+1}}^{n, N, h}  :=\widetilde{X}_{t_{m}}^{n, N,h}\!+h \, b_m\Big(t_m,\, \widetilde{X}_{t_0 : t_m}^{n, N,h}, \,\widetilde{\mu}_{t_0 : t_m}^{N,h}\Big)\!+\!\sqrt{h}\,\sigma_m\Big(t_m, \,\widetilde{X}_{t_0 : t_m}^{n, N,h}, \,\widetilde{\mu}_{t_0 : t_m}^{N,h}\Big)Z_{m+1}^{n}, 
\end{align}
where $\widetilde{X}_{t_0}^{1, N,h}, ..., \widetilde{X}_{t_0}^{N, N,h}\widesim{\mathrm{i.i.d}}X_0$ and 
 for every $m \in \llM$, $\widetilde \mu^{N,h}_{t_m}$ is the associated empirical distribution of the particle system at time $t_m$, i.e. 
 \begin{align}\label{def:empirical_measure}
 \widetilde \mu^{N,h}_{t_m} := \frac1{N} \sum_{n=1}^N \delta_{\widetilde X^{n,N,h}_{t_m}}. \end{align} 
We also define the continuous extension particle system  $(\widetilde X^{1,N,h}_{t}, \dots, \widetilde X^{N,N,h}_{t})_{t \in [0,T]}$ from \eqref{eq:particlesystem} by setting, for all $i \in \llN^*$, $m \in \llbracket M-1 \rrbracket$  and for all $t \in (t_m, t_{m+1}]$, 
\begin{align}\label{eq:contiparticlesystem}
\widetilde X^{i,N,h}_t = \widetilde X^{i,N,h}_{t_m} + (t-t_m) \, b_m\big(t_m, \widetilde X^{i,N,h}_{t_0:t_m}, \widetilde \mu^{N,h}_{t_0:t_m}\big)+ \sigma_m \, \big(t_m, \widetilde X^{i,N,h}_{t_0:t_m}, \widetilde \mu^{N,h}_{t_0:t_m}\big)(B^i_t - B^i_{t_m}). 
\end{align}
\end{defn}


\subsection{Notations, assumptions and main results}

In the whole paper, we use the notation $\mu = \mathbb{P} \circ X^{-1}=\mathcal{L}(X)$ or $X \sim \mu$  to indicate that a random variable $X$ has the distribution $\mu$ and we use $\Vert X\Vert_p$ for the $L^p$-norm of $X$, $p\geq1$.
For a Polish space $(S,d_S)$,  the Wasserstein distance $W_p$ on $\calP_p(S)$ is defined by
\begin{align}
\label{eq:def_Wasserstein}
&W_p(\mu, \nu) := \inf_{\pi \in \Pi(\mu,\nu)} \Big( \int_{S \times S} d_S(x,y)^p \, \pi(\d x, \d y) \Big)^{\tfrac1{p}}  \\
&\quad = \inf \Big\{ \EE \big[d_S(X,Y)^p \big]^{\tfrac1{p}}, \, X,Y: (\Omega, \mathcal{F}, \mathbb{P}) \to (S,\mathcal{S})\: \text{with } \:\mathbb{P} \circ X^{-1} = \mu, \, \mathbb{P} \circ Y^{-1} = \nu \Big\},  \nonumber
\end{align}
where $\Pi(\mu,\nu)$ denotes the set of probability measures on $(S \times S, \mathcal{S}^{\otimes 2})$ with marginals $\mu$ and $\nu$, and $\mathcal{S}$ denotes the Borel $\sigma$-algebra on $S$ generated by the distance $d_S$. We write $\mathcal{W}_p$ for the case $S = \R^d$ and $\mathbb{W}_p$ for the case $S = \mathcal{C}([0,T], \R^d)$ endowed with the supremum norm $\|\alpha\|_{\sp} = \sup_{t \in [0,T]}|\alpha_t|$. We also introduce $\CPP$, the space of probability distributions $(\mu_t)_{t \in [0,T]}$ such that $t\in[0,T] \mapsto \mu_t \in \mathcal{P}_p(\R^d)$ is continuous with respect to the distance $\mathcal{W}_p$.  For $(\mu_t)_{t\in[0,T]}, (\nu_t)_{t\in[0,T]}\in \CPP$, we will repeatedly use $\sup_{t\in[0,T]}\mathcal{W}_p(\mu_t, \nu_t)$ as a distance between these two elements.  
In addition, we define, for $p \ge 1$ and $t \in [0,T]$, the truncated Wasserstein distance $\mathbb{W}_{p,t}$ on $\mathcal{P}_p(\mathcal{C}([0,T], \R^d))$ by 
\begin{align}
\label{eq:def_truncated_Wass}
&\forall\,\mu, \nu\!\in \!\mathcal{P}_p(\mathcal{C}([0,T], \R^d)), \nonumber \\
&\qquad \mathbb{W}_{p,t}(\mu, \nu) \coloneqq \!\!\inf_{\pi \in \Pi(\mu,\nu)} \!\Big[\int_{\mathcal{C}([0,T], \R^d) \times \mathcal{C}([0,T], \R^d)}\sup_{s\in[0,t]}|x_s-y_s|^p \;\pi(\d x, \d y) \Big]^{\tfrac1{p}}.
\end{align}

In this paper, we introduce three assumptions parameterized by an index $\mathfrak{p} \ge 2$.  


\begin{manualtheorem}{(I)}\label{assump:norm_X0}
$\Vert X_0\Vert_\mathfrak{p}<+\infty. $
\end{manualtheorem}

\begin{manualtheorem}{(II)}\label{assump:lipschitz}
The coefficient functions $b, \sigma$ are continuous in $t$, uniformly Lipschitz continuous in $\alpha$ and in $(\mu_{t})_{t\in[0,T]}$  in the following sense : there exists $L>0$ s.t.  
\begin{align*}
&\forall \, t \in[0,T], \; \forall \,\alpha, \beta\in\CRD \text{ and } \forall \, (\mu_t)_{t\in[0,T]}, (\nu_t)_{t\in[0,T]}\in\mathcal{C}\big([0,T], \calP_{\mathfrak{p}}(\RD)\big),\nonumber\\
& \max \Big( \big|b \big(t,\alpha, (\mu_t)_{t\in[0,T]}\big)-b\big(t,\beta, (\nu_t)_{t\in[0,T]}\big)\big| , 
\vertiii{\sigma\big(t,\alpha, (\mu_t)_{t\in[0,T]}\big)-\sigma\big(t,\beta, (\nu_t)_{t\in[0,T]}\big)}  \Big)\\
&\qquad \leq L \Big[ \;\Vert\alpha-\beta\Vert_{\sup}+\sup_{t\in[0,T]}\mathcal{W}_\mathfrak{p}(\mu_t,\nu_t) \Big]. 
\end{align*} 
\end{manualtheorem}
\begin{manualtheorem}{(III)}\label{assump:holder}
The coefficient functions $b, \sigma$ are $\gamma$-H\"older continuous in  the first argument $t\in[0,T]$ for some $0< \gamma\leq 1$, uniformly in $\alpha$ and in $(\mu_{t})_{t\in[0,T]}$, in the following sense : there exists $L>0$ s.t. 
\begin{align}
	&\forall \, t,s\in[0,T], \; \forall \, \alpha\in\CRD \text{ and }\forall \, (\mu_{t})_{t\in[0,T]}\in\mathcal{C}\big([0,T], \calP_{\mathfrak{p}}(\RD)\big),\nonumber\\
	& \max \Big( \big|b\big(t,\alpha, (\mu_t)_{t\in[0,T]}\big)-b\big(s,\alpha, (\mu_t)_{t\in[0,T]}\big)\big|,  \vertiii{\sigma\big(t,\alpha, (\mu_t)_{t\in[0,T]}\big)-\sigma\big(s,\alpha, (\mu_t)_{t\in[0,T]}\big)} \Big)\nonumber\\
	&\qquad \qquad \qquad \qquad   \leq L\Big(1+\Vert \alpha\Vert_{\sup}+ \sup_{t\in[0,T]}\mathcal{W}_{\mathfrak{p}}(\mu_t, \delta_0)\Big)|t-s|^{\gamma},\nonumber
\end{align}
where $\delta_0$ is the Dirac measure at $0$. 
\end{manualtheorem}
Assumptions \ref{assump:norm_X0} and \ref{assump:lipschitz} are sufficient conditions for the existence and strong uniqueness of the solution $(X_t)_{t\in[0,T]}$ to the path-dependent McKean-Vlasov equation \eqref{eq:path-dependent_McKean}. In fact, the following result can be extracted from \cite[Theorem A.3]{Djete_2022}, see also the earlier version of this work \cite[Theorem 1.1]{Bernou_2023_v1}. 
\begin{thm}
\label{thm:well-posedness}
Let $p\geq 2$. Assume that Assumptions \ref{assump:norm_X0} and \ref{assump:lipschitz} hold with $\mathfrak{p}=p$. There exists a unique strong solution $(X_t)_{ t \in [0,T]}$ from $(\Omega, \mathcal{F}, \mathbb{P})$ to $(\mathcal{C}([0,T], \mathbb{R}^d), \| \cdot \|_{\sup})$ of the path-dependent McKean-Vlasov equation \eqref{eq:path-dependent_McKean}. Moreover, this unique solution $(X_t)_{t\in[0,T]}$ satisfies that
\begin{equation}\label{eq:thm_well_posed}
\Big\Vert \sup_{t\in[0, T]}|X_t|\;\Big\Vert_{p} \le \Gamma \Big(1 + \|X_0\|_{p}\Big), 
\end{equation}
where $\Gamma > 0$ is a constant depending on $b, \sigma, L, T, d$ and $p$. 
\end{thm} \dd


Moreover, consider now  the following $N$-particle system $(X^{1,N}_t, \dots, X^{N,N}_t)_{t \in [0,T]}$ defined by
{\small
\begin{align}\label{eq:particle_system_intro}
\begin{cases}
\:X^{i,N}_t= X^{i,N}_0 \!+ \int_0^t b\big(s, X^{i,N}_{\cws}, \mu^N_\cws \big)  \d s  + \int_0^t \sigma \big(s, X^{i,N}_\cws, \mu^N_\cws \big) \d B^{i}_s,   \hspace{.2cm} 1 \le i \le N, \, t \in [0,T], \\
 \mu^N_t := \frac{1}{N} \sum_{i=1}^N \delta_{X^{i,N}_t}, \quad t \in [0,T], 
\end{cases}
\end{align}
}
where $X^{1,N}_0, \dots, X^{N,N}_0$ are i.i.d. random variables  having the same distribution as $X_0$, and $B^i := (B^i_t)_{t \in [0,T]}$, $1 \le i \le N$ are  independent  $\R^q$-valued standard Brownian motions  and independent of $X^{1,N}_0, \dots, X^{N,N}_0$. 
Assumptions \ref{assump:norm_X0} and \ref{assump:lipschitz} imply the following propagation of chaos result, Theorem \ref{thm:global_thm_chaos}. We note that several variants of propagation of chaos property for path-dependent McKean-Vlasov equation were recently addressed, see Section \ref{subsec:previous_results} below for comparison with our setting; as we could not find a result readily applicable to our framework, we provide a proof in Appendix \ref{sec:chaos} that adapts the classical argument using synchronous coupling to the path-dependent setting.  

\begin{thm}[Propagation of chaos]\label{thm:global_thm_chaos}
Let $p\geq 2$. Assume that  Assumptions \ref{assump:norm_X0} and \ref{assump:lipschitz} hold with $\mathfrak{p}=p$.
Let $X$ be the unique solution to \eqref{eq:path-dependent_McKean}  and write $\mu := \mathbb{P} \circ X^{-1}$ and $(\mu_t)_{t \in [0,T]}$ for its marginal distributions.
	Let $(X^{1,N}_t, \dots, X^{N,N}_t)_{t \in [0,T]}$ be the processes defined by the $N$-particle system \eqref{eq:particle_system_intro} and $(Y^1, \dots, Y^N)$ be $N$ i.i.d. copies of $X$. 
\begin{enumerate}
\item[$(a)$] 
there holds, for some constant $C_{d,p,L,T} > 0$, for all $N \ge 1$,
\begin{align}\label{eq:thm_chaos_1}
\Big\| \sup_{t \in [0,T]} \mathcal{W}_p\Big( \mu_t, \frac1{N} \sum_{i=1}^N \delta_{X^{i,N}_t} \Big) \Big\|_p \le C_{d,p,L,T} \Big\| \mathbb{W}_p(\mu, \nu^N) \Big\|_p, 
\end{align}
where $\nu^N := \tfrac1{N} \sum_{i=1}^N \delta_{Y^i}$ is the empirical measure of $(Y^1,\dots,Y^N)$. Moreover, the norm $\| \mathbb{W}_p(\mu, \nu^N) \|_p$ converges to $0$ as $N \to \infty$. Furthermore,  
for a fixed $k \in \mathbb{N}^*$, we have the weak convergence:
\begin{align}
\label{eq:thm_chaos_2}
\Big( X^{1,N}, \dots, X^{k,N} \Big) \Rightarrow \Big(Y^1, \dots, Y^k \Big) \qquad \mathrm{as } \quad N \to \infty.
\end{align}
\item[$(b)$] If, in addition, Assumption \ref{assump:norm_X0} holds with some constant $\mathfrak{p} \ge d+2p+1$, for all $\delta > 0$, there exists some constant $C_{d,p,L,T,\delta}>0$ with $C_{d,p,L,T,\delta} \to \infty$ as $\delta \to 0$ such that for all $N\geq 1$, 
\begin{align}\label{eq:thm_chaos_3}
\Big\| \sup_{t \in [0,T]} \mathcal{W}_p\Big( \mu_t, \frac1{N} \sum_{i=1}^N \delta_{X^{i,N}_t} \Big) \Big\|_p \le C_{d,p,L,T,\delta} N^{-\frac{1}{d+4p} + \delta}. 
\end{align}
\end{enumerate}   
\end{thm}
\dd Our main result draws inspiration from the above propagation of chaos property to provide a convergence rate for the numerical scheme \eqref{eq:particlesystem}:
\begin{thm}[Convergence rate of the particle method]\label{thm:particlemethod} 
Let $p\geq 2$. Assume that Assumption~\ref{assump:norm_X0} holds with moment order $\mathfrak{p}=p+\varepsilon_0$ for some $\varepsilon_0>0$, where $\varepsilon_0\neq p$ if $p\le d/2$ and $\varepsilon_0\neq p^2/(d-p)$ if $p\in(0,d/2)$. Moreover, assume that Assumptions~\ref{assump:lipschitz} and~\ref{assump:holder} hold with index $\mathfrak{p}=p$.
For every $m \in \llM$, let $\widetilde{\mu}_{t_m}^{N,h}$ denote the empirical measures defined by \eqref{def:empirical_measure} and let  $\widetilde{\mu}^h_{t_m}$  be the probability distribution of $\widetilde{X}^h_{t_m}$ defined by the interpolated scheme in  \eqref{eq:discretescheme}. 
\begin{enumerate}
    \item[(a)]Set 
\[\alpha :=\min\left\{\frac{1}{2p}, \,\frac{\varepsilon_0}{p(p+\varepsilon_0)}\right\}\mathbbm{1}_{\{p\geq \frac{d}{2}\}}+\,\min\left\{\frac{1}{d},\, \frac{\varepsilon_0}{p(p+\varepsilon_0)}\right\}\mathbbm{1}_{\{p<\frac{d}{2}\}} \]
and  $M=\max (N^{\alpha p/(1+(\gamma \wedge\frac{1}{2})p)},2T+1).$ 
We have 
\begin{align}\label{mainresult-review}
&\Big\Vert \max_{0\leq \ell \leq M} \mathcal{W}_p\big(\widetilde{\mu}_{t_\ell}^{N,h}, \mu_{t_\ell}\big)\Big\Vert_p\leq  C_{b, \sigma, L, T, d, p, \varepsilon_0, \Vert X_0\Vert_{p+\varepsilon_0}}  \nonumber\\
&\hspace{1cm}\times\begin{cases}
N^{-\frac{\alpha p  \gamma}{1+  \gamma p}}\mathbbm{1}_{\{\gamma\in (0,\frac{1}{2})\}}+N^{-\frac{\alpha p  }{2+   p}}\big(\log(N)\big)^{\frac{1}{2}}\mathbbm{1}_{\{\gamma\in [\frac{1}{2}, 1]\}}, \quad \text{ if } p\neq \frac{d}{2},\\
N^{-\frac{\alpha p (\frac{1}{2}\wedge \gamma)}{1+ (\frac{1}{2}\wedge \gamma)p}}\big(\log(N)\big)^{\frac{1}{p}}, \quad \text{ if } p=\frac{d}{2}.
\end{cases}
\end{align}
    \item[(b)] If, in addition, Assumption~\ref{assump:norm_X0} holds with moment order $\mathfrak{p}= d+2p+1$, and if one of the conditions (i) and (ii) below is satisfied,
    \begin{itemize}
        \item[(i)] $\gamma\in[\frac{1}{2}, 1]$ and  $d\geq 2p^2\geq 8$,
        \item[(ii)] $\gamma\in(0, \frac{1}{2})$ and $\gamma<  \frac{2}{d+2p}\mathbbm{1}_{\{p\geq \frac{d}{2}\}}+\frac{d}{4p^2}\mathbbm{1}_{\{p<\frac{d}{2}\}}$,
    \end{itemize}
     then, for any $\delta>0$ and for $M\geq \max(2T+1, N^{\,\frac{1}{(d+4p)\min(\gamma, 1/2)}})$,  the following improved
convergence rate holds 
\begin{align}\label{mainresult-v2}
&\Big\Vert \max_{0\leq \ell \leq M} \mathcal{W}_p\big(\widetilde{\mu}_{t_\ell}^{N,h}, \mu_{t_\ell}\big)\Big\Vert_p\leq  C_{b, \sigma, L, T, d, q, p,\Vert X_0\Vert_{q},\delta} \,  N^{-\frac1{d+4p}+\delta},
\end{align}
 where the constant
$C_{b,\sigma,L,T,d,q,p,\|X_0\|_q,\delta}$ scales as $1/\delta$.

\end{enumerate}
\end{thm}

\begin{rem}
Notice that the two bounds have a quite different behavior when $\gamma$ is small. On the one hand, the choice of $M$ in~\eqref{mainresult-review} does not explode as $\gamma \to 0$, but the convergence rate becomes more and more trivial. On the other hand, in~\eqref{mainresult-v2}, the constraint on $M$ explodes as $\gamma \to 0$, but the obtained rate is \textit{uniform} in $\gamma$.
\end{rem}

\begin{rem}
The two error bounds \eqref{mainresult-review} and \eqref{mainresult-v2} in Theorem~\ref{thm:particlemethod} rely on different strategies.
Assertion~(a) is derived from the Wasserstein distance estimates of Fournier and Guillin~\cite{fournier2015rate} for quantities of the form
\[\EE\left[\calW_p^p\left(\mathcal{L}(\xi), \frac{1}{N}\sum_{i=1}^N\delta_{\xi_i}\right)\right],\] 
where $\xi, \,\xi_i, 1\leq i \leq N$ are $\RD$-valued random variables with $\xi_i\iid \mu$ and $\xi \sim \mu$. These estimates are combined with a decomposition of the error into a time discretization term and a particle approximation term, and the discretization parameter 
$M$ is chosen to balance these two sources of error. Further details are provided in Section~\ref{subsec:proof-main-thm-(a)}. 

Assertion~(b) exploits the stronger moment assumption on 
$X_0$ and relies on a clarified and generalized version of the estimate of Horowitz and Karandikar~\cite{MR1329874} for quantities of the form 
\[\EE \left[\sup_{t\in[0,1]}\mathcal{W}_{p}^p\left(\mathcal{L}(\xi_t), \frac{1}{N}\sum_{i=1}^N\delta_{\xi^i_t})\right)\right]\] where $(\xi_t)_{t\in[0,1]}, (\xi_t^i)_{t\in[0,1]}, 1 \le i \le N$ are i.i.d. stochastic processes taking values in $\mathcal{C}\big([0,1], \RD\big)$. This approach also relies on a balance between time discretization and particle approximation errors, but through a different scaling. See Section~\ref{subsec:HKmethod} for a detailed discussion.
\end{rem}

The following corollary shows how the results of Theorem~\ref{thm:particlemethod} allows one to control the distance in Wasserstein distance between the particle system and the solution of~\eqref{eq:path-dependent_McKean}. 

\begin{cor}\label{cor} Let $\widetilde{X}^{1, N}$ denote the first process of the particle system defined by \eqref{eq:contiparticlesystem} and let $X$ denote the unique solution to \eqref{eq:path-dependent_McKean}. Assume that both processes are driven by the same Brownian motion. Under the same assumptions as in Theorem~\ref{thm:particlemethod}-{(a)}, and with the same choice of the parameters $\alpha$ and $M$, we have
\begin{align}\label{mainresult-coroll}
& \mathbb{W}_p\Big(\mathcal{L}(\widetilde X^{1,N}),\mathcal{L}(X) \Big) \leq  C_{b, \sigma, L, T, d, p, \varepsilon_0, \Vert X_0\Vert_{p+\varepsilon_0}}  \nonumber\\
&\hspace{1cm}\times\begin{cases}
N^{-\frac{\alpha p  \gamma}{1+  \gamma p}}\mathbbm{1}_{\{\gamma\in (0,\frac{1}{2})\}}+N^{-\frac{\alpha p  }{2+   p}}\big(\log(N)\big)^{\frac{1}{2}}\mathbbm{1}_{\{\gamma\in [\frac{1}{2}, 1]\}}, \quad \text{ if } p\neq \frac{d}{2},\\
N^{-\frac{\alpha p (\frac{1}{2}\wedge \gamma)}{1+ (\frac{1}{2}\wedge \gamma)p}}\big(\log(N)\big)^{\frac{1}{p}}, \quad \text{ if } p=\frac{d}{2}.
\end{cases}
\end{align}
Moreover, under the same assumptions as in Theorem~\ref{thm:particlemethod}-{(b)}, and for any $\delta>0$ and for $M\geq \max(2T+1, N^{\,\frac{1}{(d+4p)\min(\gamma, 1/2)}})$,  the following improved
convergence rate holds 
\begin{align}
&\mathbb{W}_p\Big(\mathcal{L}(\widetilde X^{1,N}),\mathcal{L}(X) \Big)\leq  C_{b, \sigma, L, T, d, q, p,\Vert X_0\Vert_{q},\delta} \, N^{-\frac1{d+4p}+\delta},\nonumber
\end{align}
 where the constant
$C_{b,\sigma,L,T,d,q,p,\|X_0\|_q,\delta}$ scales as $1/\delta$.


\end{cor}

\begin{rem}[Dependency in time]
    A careful examination of the proofs shows that, in both~\eqref{mainresult-review} and~\eqref{mainresult-v2}, the constant behaves like $e^{e^{2T}}$ as $T$ increases, due to two applications of Gronwall's lemma. So far, uniform-in-time propagation of chaos results (which would be a first step to obtain a uniform-in-time convergence rate for the particle method) were only obtained in the case of standard McKean-Vlasov equations with additional convexity assumptions on $b$ and $\sigma$, see \cite{BRTV-98, BGG-13, Carrillo_2003, Delarue_Tse_21}. To our knowledge, the derivation of uniform-in-time propagation of chaos results for path-dependent McKean-Vlasov equations, even under stronger assumptions on the coefficients, remains an open problem.
\end{rem}

\subsection{Literature review and discussion of our results}\label{subsec:previous_results}


The standard McKean-Vlasov equation \eqref{eq:McKean_standard} has widespread applications in diverse fields such as opinion dynamics \cite{Hegselmann_2002}, finance (for instance through the rank-based model, see \cite{Karatzas_2009} and the references therein), plasma physics \cite[Chapter 1]{Bittencourt_2004} and neurosciences \cite{Caceres_2011, Carrillo_2015, Delarue_2015}. \dd It also plays a key role in the theory of mean-field games \cite{Carmona_2018, Carmona_2018b, Cardaliaguet_2013}, with applications in biological models for animal competition, road traffic engineering and dynamic economic models, see Huang-Malham\'e-Caines \cite{Caines_2006} and the references within. In this context, the study of the convergence rate of particle methods for numerical simulations has been initiated by Talay and Bossy-Talay \cite{Talay_1996, Bossy_1997}, and has been an active area of research in the last decades, see \cite{Antonelli_2002, liu2022particle, hoffmann2023statistical} and the references therein.

The generalized McKean-Vlasov equation with path-dependent coefficients has applications in finance, economics, mean-field games, and is adressed e.g. by Cosso et al. \cite{Cosso_2020}, Lacker \cite{Lacker_2018}, Djete et al. \cite{Djete_2022}, and by Baldasso et al. \cite{Baldasso_2022}. We also mention more specific recent approaches focused on the 2d parabolic-parabolic Keller-Segel equation, see Toma{\v{s}}evi{\'{c}} and Fournier-Toma{\v{s}}evi{\'{c}} \cite{Tomasevic_2021, Fournier_2023b} and  on the modelling of volatility  \cite{guyon2023volatility}. In Section 2, we present theoretically and through numerical results a path-dependent model for neural masses in the visual cortex: there, the path-dependency allows one to take into account potentiation effects.  In \cite{Cosso_2020, Lacker_2018, Djete_2022}, the dependence on the measure argument differs from \eqref{eq:path-dependent_McKean}: specifically, the dynamics are expressed as:
\begin{align}
\label{eq:path_dep_2}
\d X_t &= b\left(t, X_{\cdot \wedge t}, \mathcal{L}(X_{\cdot \wedge t})\right) \d t + \sigma\left(t, X_{\cdot \wedge t}, \mathcal{L}(X_{\cdot \wedge t})\right) \d B_t,
\end{align}
where $\mathcal{L}(X_{\cdot \wedge t}) \in \calP_p(\mathcal{C}([0,T], \R^d))$ represents the probability distribution of the entire path $X_{\cdot \wedge t}$. Regarding the propagation of chaos property of \eqref{eq:path_dep_2}, \cite{Lacker_2018} studies the convergence with respect to the total variation distance. His method uses a Girsanov theorem, hence the diffusion coefficient $\sigma$ cannot depend on the measure argument $\mathcal{L}(X_{\cdot \wedge t})$. The large deviation result presented in \cite{Baldasso_2022} implies the propagation of chaos property (see Corollary 4.4 of \cite{Baldasso_2022}), see also \cite{Budhiraja2012}, where Section 7.2 presents a generalization to the path-dependent structure. Certainly, in the path-dependent McKean-Vlasov equation \eqref{eq:path-dependent_McKean}, the measure argument $\mu_{\cdot \wedge t}$ made from the marginal distributions taken in $\mathcal{C}([0,T], \calP_p(\R^d))$ instead of $\calP_p\big( \mathcal{C}([0,T], \R^d) \big)$ can be considered as a special case of the dependency on $\mathcal{L}(X_{\cdot \wedge t})$. Nevertheless, this framework constitutes a trade-off between the theoretical aspects, the numerical perspectives and the applications.  Indeed, our setting can be simulated more easily and with an explicit convergence rate. The potential adaptation of our strategy to \eqref{eq:path_dep_2} remains unclear, in particular due to our use of the interpolator, and is left as an open problem.

Coming back to the study of~\eqref{eq:path-dependent_McKean}, and as those results do not explicitly cover our setting, we first study the propagation of chaos in terms of the Wasserstein distance in Theorem~\ref{thm:global_thm_chaos}. The proof is provided in Appendix~\ref{sec:chaos}. Our approach relies on a synchronous coupling argument, and we ultimately bound the convergence rate by $\mathbb{W}_p(\mu, \nu^N)$, the convergence rate in Wasserstein distance of the empirical measure of i.i.d. random processes. This strategy is grounded in the well-understood convergence of the empirical measures of i.i.d. sequences in finite-dimensional cases (see, e.g., \cite{fournier2015rate}), yielding~\eqref{mainresult-review} and improves on the control based on Gaussian smoothing from~\cite{MR1329874}. The latter corresponds to Equation~\eqref{mainresult-v2}, which we obtain after suitable extension (and clarification) of the Horowitz-Karandikar argument to the $W_p$ case. This provides a better bound only in high dimension $d \ge 8$. Those bounds are then combined with an analysis of the time discretization appearing in the Euler scheme to yield our main result, Theorem \ref{thm:particlemethod}.

The simulation of path-dependent equations naturally displays a more intricate link between the space and time discretization. A natural question raised by our results thus regards the optimality of the rates derived in this paper.  The simulation of a generalized Ornstein-Uhlenbeck process with memory hints towards such optimality. Indeed, in low dimension with the choice $p = 2$, the bound~\eqref{eq:bound_with_M} yielding Theorem~\ref{thm:particlemethod}-(a) does not converge as $N \to \infty$ if $M < N^{\frac14}$. For our generalized Ornstein-Uhlenbeck process (one-dimensional), with the choice $M = 10$ and $N$ ranging from $2^7$ to $2^{18}$, the numerical behavior of the estimation of the $W_2$ distance is as predicted by the bound~\eqref{eq:bound_with_M}: the error converges as long as $N \le 2^{13}$, before stagnation begins. Notice then that $N^{1/4} \simeq 10$ when $N = 2^{13}$: the numerical behavior thus confirms that the bound~\eqref{eq:bound_with_M} from which Theorem~\ref{thm:particlemethod} is deduced captures the right balance between time and space discretization errors, which corresponds to the one derived for the standard McKean-Vlasov equation~\cite{liu2022particle}, up to the prefactor $M^{1/p}$ multiplying the space discretization error.

\subsection{Plan of the paper}

This article is organized as follows. In Section \ref{sec:Applications} we present our applications and the corresponding numerical results. We focus first in Section \ref{subsec:OU} on a modified Ornstein-Uhlenbeck process with memory effect. This toy model presents the key feature of having an explicit solution, making it an interesting base point to confirm numerically our findings from Theorem \ref{thm:particlemethod}. In Section \ref{subsec:neural}, we introduce an enriched Jansen-Rit model, with memory and delay effects, and present the associated numerical results obtained through the particle method. Section \ref{sec:numericalcvg} first focuses on preliminary results used to prove Theorem \ref{thm:particlemethod} (Section \ref{subsec:prop_im}), then details the derivation of a convergence rate for the Euler schemes \eqref{eq:discretescheme} and \eqref{eq:def_continuous_2} in Section \ref{sec:Euler_scheme} and culminates in the proof of Theorem \ref{thm:particlemethod} (Section \ref{subsec:cvgparticlemethod}). Finally, Appendix \ref{sec:chaos} contains our proof of Theorem \ref{thm:global_thm_chaos}, while Appendix \ref{appB} contains some proofs from Section \ref{sec:Applications} and \ref{sec:numericalcvg}.

\section{Applications and numerical simulations} 
\label{sec:Applications}

We investigate our simulation method on two different examples. The simulation code is
available via \href{https://github.com/ArmdBrn/McKean_PathDep}{Github}, see \href{https://github.com/ArmdBrn/McKean_PathDep}{https://bit.ly/45r7na6}.

\subsection{A linear interaction with delay}
\label{subsec:OU}

In dimension 1, we consider the following path-dependent McKean-Vlasov equation with delay and linear interaction: 
\begin{align}
	\label{eq:linear_interaction}
	\d X_t = 2 \left[\int_0^t \int_{\R} \big(x - X_t\big) \mu_s(\d x) \, \d s \right]\, \d t + \d B_t,  \qquad \mathcal{L}(X_0) = \mathcal{N}(m,1),
\end{align}
for some $m \in \R$,  where $(B_t)_{t \ge 0}$ is a standard Brownian motion independent of $X_0$ and where we write $\mathcal{N}(m,\sigma^2)$ to denote the Gaussian random variable with mean $m \in \R$ and variance $\sigma^2 > 0$. To fit with \eqref{eq:path-dependent_McKean}, for $t \in [0,T]$, $\alpha \in \mathcal{C}([0,T], \mathbb{R})$ and $(\mu_s)_{s\in[0,T]} \in \mathcal{C}([0,T], \mathcal{P}_p(\R))$, our drift writes 
\begin{equation}\label{eq:drift_linear}
b(t, \alpha, (\mu_s)_{s\in[0,T]}): = 2 \int_0^t \Big[ \int_{\mathbb{R}} (x - \alpha_t) \,  \mu_s (\d x) \Big] \d s.
\end{equation}
It is easily checked that Equation \eqref{eq:linear_interaction} writes as \eqref{eq:path-dependent_McKean}, where the drift $b(\cdot, \cdot, \cdot)$ given by~\eqref{eq:drift_linear} and the volatility $\sigma = 1$ satisfy Assumptions \ref{assump:lipschitz} and \ref{assump:holder} with  $p = 2$ and $\gamma = 1$. 
Moreover, our choice of the model \eqref{eq:linear_interaction} is guided by the existence of an explicit solution, as established in Proposition \ref{prop:linear_interaction} below, whose proof is given in Appendix \ref{proofsubsec:appli}.
\begin{prop} 
\label{prop:linear_interaction}
For all $T > 0$ fixed, the equation \eqref{eq:linear_interaction} has a unique strong solution $(X_t)_{t\in[0,T]}$ given by
\[ X_t = (X_0 - m) e^{-t^2} + m + \int_0^t e^{-(t^2 - r^2)} \, \d B_r. \]
In particular, for all $t \in [0,T]$,
\[ X_t \sim \mathcal{N}\Big(m, e^{-2 t^2} \big( 1+ \int_0^t e^{2r^2} \, \d r \big) \Big).\]
\end{prop}

\subsubsection{Numerical results}

We turn to numerical results for the model \eqref{eq:linear_interaction}. We implement the particle method \eqref{eq:particlesystem}. We pick $T = 1$, and several values of both the particle number $N$ and the number of time steps $M$. Each simulation uses a Monte-Carlo approximation with $N_{MC} = 30$ implementations. For all $i \in \llbracket N \rrbracket^*$, $j \in \llbracket N_{MC} \rrbracket^*$, $m \in \llM$ and $t_m=\frac{mT}{M}$, we use $X^{i,N,j}_{t_m}$ to denote the $i$-th particle of the $j$-th Monte Carlo simulation when the number of particles is $N$, taken at time $t_m$. We compute the error in the $L^2$-Wasserstein distance, using that in dimension $1$, the following equality holds for $\mu, \nu \in \calP_2(\R)$:
\begin{align} 
	\label{eq:formula_Wasserstein}
	\mathcal{W}_2(\mu,\nu)^2 = \int_0^1 \Big|F_\mu^{-1}(\xi) - F_\nu^{-1}(\xi) \Big|^2 \, \d \xi, 
	\end{align}
where $F_\mu^{-1}$ and $F_\nu^{-1}$ denote the quantile functions associated with $\mu$ and $\nu$, respectively. Setting for all $(N,j,m) \in \mathbb{N}^* \times \llbracket N_{MC} \rrbracket^* \times \llM$,
\[ \hat \mu^{N,j}_{t_m} := \frac1{N} \sum_{i=1}^N \delta_{X^{i, N,j}_{t_m}}, \]
our error corresponds to an estimation of the $\mathcal{W}_2$ distance at time $T$, and is given, setting $\mu_T = \mathcal{N}(m, e^{-2T^2} (1 + \int_0^T e^{2r^2} \, \d r))$, by
\begin{align}
\label{eq:formula_error}
\hat E_N^2 &= \frac{1}{N_{MC}} \sum_{j=1}^{N_{MC}} \Big\{ \epsilon \, (1-2 \epsilon) \Big[\frac{1}{2} \big|F^{-1}_{\hat \mu^{N,j}_T}(\epsilon) - F^{-1}_{\mu_T}(\epsilon)\big|^2 + \frac12\big|F^{-1}_{\hat \mu^{N,j}_T}(1-\epsilon) - F^{-1}_{\mu_T}(1-\epsilon)\big|^{\frac12} \\
&\hspace{5cm} + \sum_{k=1}^{1/\epsilon - 1} \big|F^{-1}_{\hat \mu^{N,j}_T}(k\epsilon) - F^{-1}_{\mu_T}(k\epsilon) \big|^2 \Big] \Big\}. \nonumber 
\end{align}
The parameter $\epsilon$ is both the precision of the discretization of the integral appearing in \eqref{eq:formula_Wasserstein} and the truncation for this value (since the quantile functions at $0$ and $1$ take infinite values). In the simulation, we choose $\epsilon = 10^{-6}$, as higher choices create non-negligible truncation errors.

In Figure \ref{fig:OU_M2000}, we display in $\log_2-\log_2$ scale results obtained with the choice of a fixed value $M = 200$ and $N$ ranging from $2^7$ to $2^{18}$. A linear regression of the results provides the line $y = -0.923 \, x + 0.992$ (coefficients are rounded to three significant numbers).

Next, we turn to the case where $M$ depends on $N$. 
    Since Assumption \ref{assump:holder}  with $\gamma = 1$ is satisfied by the model \eqref{eq:linear_interaction}, and because of the explicit solution given by Proposition \ref{prop:linear_interaction}, we have a prototypical example to challenge our findings from Theorem~\ref{thm:particlemethod}. 
We thus consider $M = N^{\frac{1}{4} + \epsilon}$ for some $\epsilon > 0$, in which case Theorem~\ref{thm:particlemethod} gives an upper bound on the rate of convergence in the Wasserstein $2$ distance in $O(\log(N) N^{-\frac1{8}})$. In Figure \ref{fig:OU_M14}, we consider the case $M = \lfloor N^{0.26}\rfloor + 1$ with $N = 2^j$, $9 \le j \le 18$. As predicted, we obtain a steady convergence of the error in Wasserstein distance as $N$ increases.
Our bounds appears conservative for this toy example, as it predicts a slope of about $-0.25$ in Figure \ref{fig:OU_M14}, much slower than the observed slope $-0.51$, but the very simple form of the interaction may explain this acceleration of the convergence rate. 

Still for the model \eqref{eq:linear_interaction}, we consider at last the case $M = 10$, see Figure \ref{fig:OU_M10}. Taking a constant $M$ allows us to challenge our choice of $M$  in Theorem~\ref{thm:particlemethod}, corresponding in this context to $M \ge N^{0.25}$ in regard to the final $M$-dependent bound obtained in the proof~\eqref{eq:bound_with_M}, which predicts that any lower value of $M$ will hinder the convergence. For $N = 2^j$ with $7 \le j \le 18$ we observe that the convergence of the error towards zero slows down as $N$ passes the value $2^{13}$. Note that $2^{\frac{13}4} \le 10 \le 2^{\frac{14}4}$: the numerical behavior confirms that the bound~\eqref{eq:bound_with_M} from which Theorem~\ref{thm:particlemethod} captures efficiently the relation between $M$ and $N$. 

\begin{figure}[htbp]
\caption*{Figures 1-3: Numerical results for the path-dependent Ornstein-Uhlenbeck model \eqref{eq:linear_interaction}.}
\vspace{.5cm}
    \centering
    \begin{minipage}[c][0.3\textheight][c]{0.61\textwidth}
        \centering
        \includegraphics[width=\textwidth]{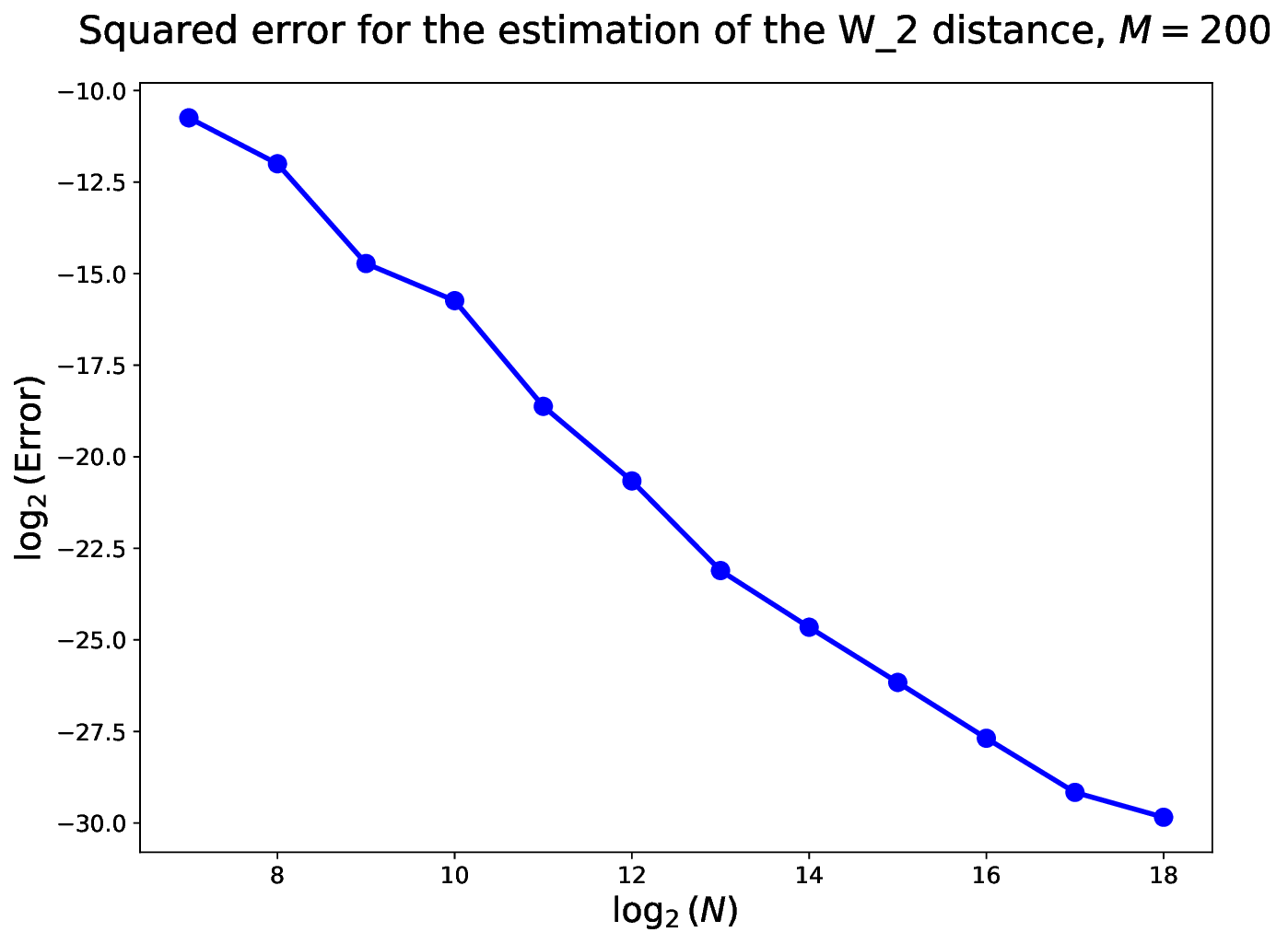}
    \end{minipage}
    \hspace{.2cm}
    \begin{minipage}[c][0.3\textheight][c]{0.35\textwidth}
        \raggedright
        \small
        \refstepcounter{figure}
        \textbf{Figure \thefigure.} $\log_2-\log_2$ scale results obtained with the choice of a fixed value $M = 200$ and $N$ ranging from $2^7$ to $2^{18}$. A linear regression of the data gives  $y = -0.923 \, x + 0.992$.
        \label{fig:OU_M2000}
    \end{minipage}

    \vskip\baselineskip 
    \begin{minipage}[c][0.3\textheight][c]{0.61\textwidth}
        \centering
        \includegraphics[width=\textwidth]{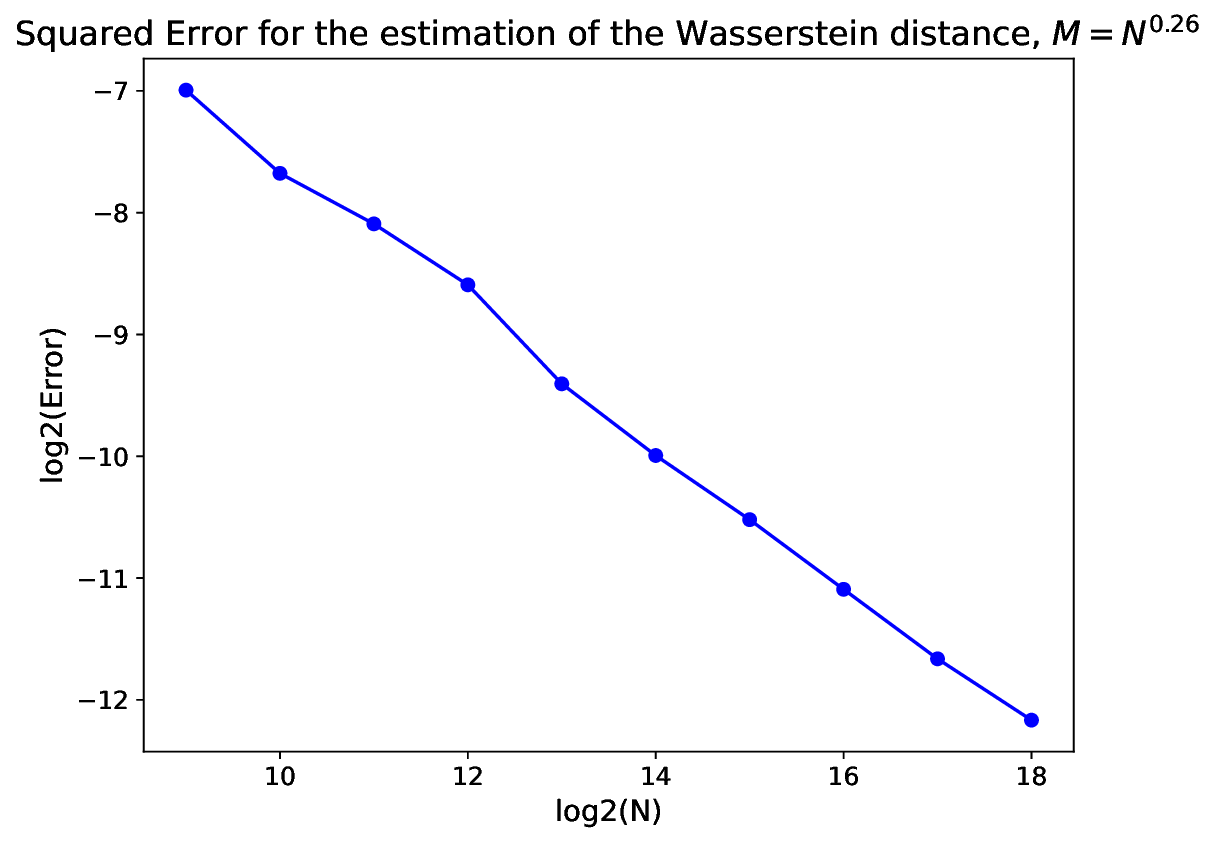}
    \end{minipage}
    \hspace{0.2cm}
    \begin{minipage}[c][0.3\textheight][c]{0.35\textwidth}
        \raggedright
        \small
        \refstepcounter{figure}
        \textbf{Figure \thefigure.}  $\log_2-\log_2$ scale results obtained with the choice $M = \lfloor N^{0.26} \rfloor + 1$ and $N$ ranging from $2^9$ to $2^{18}$. A linear regression of the data gives $y = -0.581 \, x -1.78$.
        \label{fig:OU_M14}
    \end{minipage}

    \vskip\baselineskip 
    \begin{minipage}[c][0.3\textheight][c]{0.61\textwidth}
        \centering
        \includegraphics[width=\textwidth]{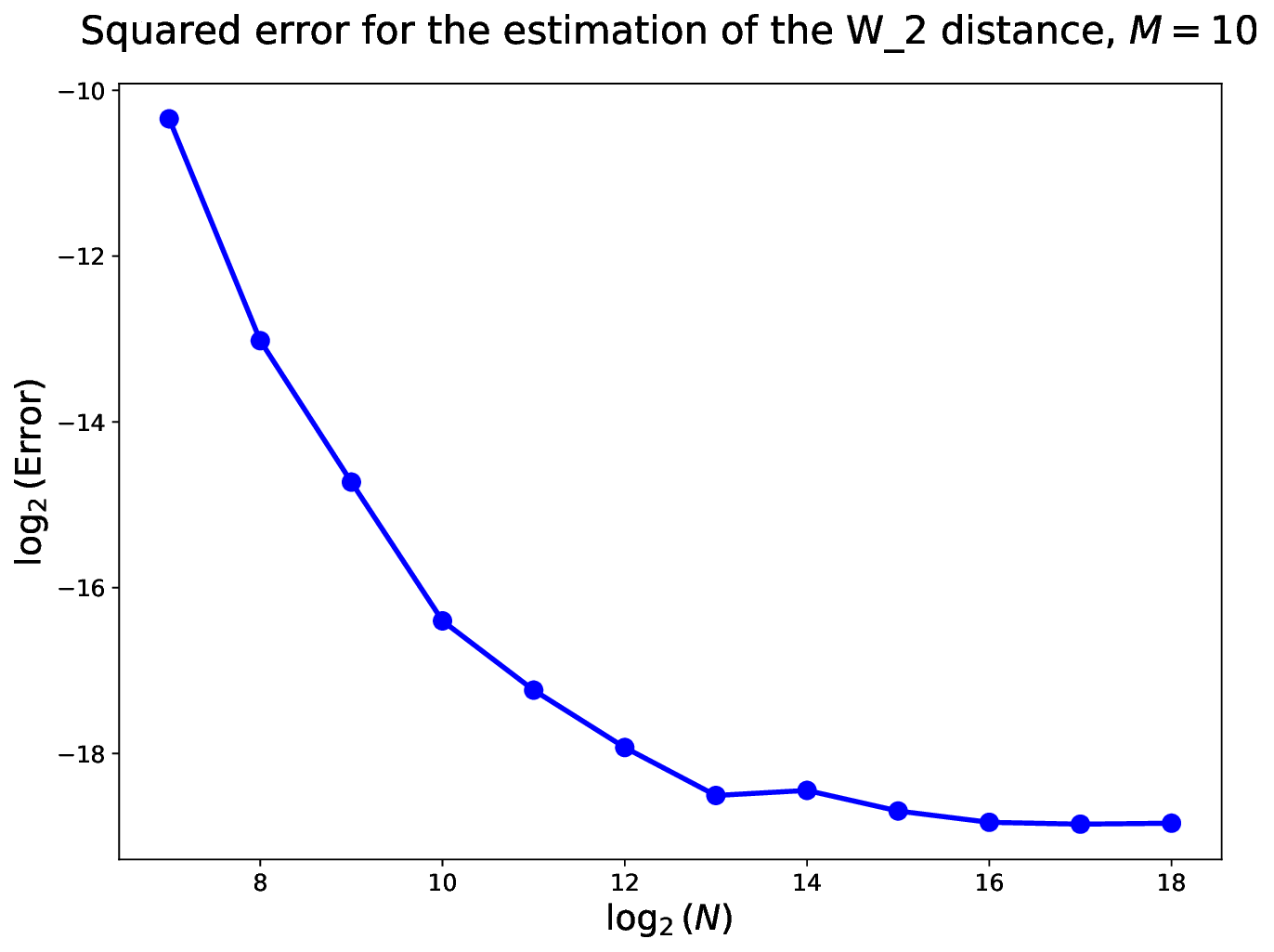}
    \end{minipage}
    \hspace{0.2cm}
    \begin{minipage}[c][0.3\textheight][c]{0.35\textwidth}
        \raggedright
        \small
        \refstepcounter{figure}
        \textbf{Figure \thefigure.} $\log_2-\log_2$ scale results obtained with the choice of a fixed value $M = 10$ and $N$ ranging from $2^7$ to $2^{18}$. As predicted from the bound~\eqref{eq:bound_with_M}, the convergence stops once $M < N^{0.25}$. A linear regression of the data gives the line  $y = -0.332 \, x	-4.25$.
        \label{fig:OU_M10}
    \end{minipage}
\end{figure}

\subsection{Application: A neural mass model with intrinsic excitability}
\label{subsec:neural}

\subsubsection{Motivation and theoretical results}

We introduce an extended version of the microscopic system leading, in the mean-field limit, to Jansen and Rit's model \cite{Jansen1995}, in the form of the equations given by Faugeras-Touboul-Cessac \cite{Faugeras2009}. This neural mass model (NMM) includes three different neuron populations and is used to get a deeper understanding of visual cortical signals, more specifically of the emergence of oscillations in the electrical activity of the brain registered by an electroencephalogram after a stimulation of a sensory pathway. The three populations are organised as follows: the pyramidal population, thereafter numbered $1$, the excitatory feedback population, indexed by $2$, and the inhibitory interneuron population, indexed by $3$. More details on the model can be found in \cite{Faugeras2009}, see in particular their Figure 2 for a graphical representation.

At the level of the particle system, given a time horizon $T > 0$ and a number $N_j \in \mathbb{N}^*$ of neurons in population $j$, the equations for the potential of the neuron $i$ in population $j$ of \cite{Faugeras2009} take the form
\begin{align}
	\label{eq:Jansen_Rit_particle}
	\d V_{j,i}(t) = - \frac1{\tau_j} V_{j,i}(t) \d t + \Big( \sum_{k=1}^3 \sum_{\ell = 1}^{N_k} \bar J_{j,k} S\big(V_{k,\ell}(t) \big) + I_j(t) \Big) \d t + f_j(t) \d W^{j,i}_t, 
\end{align}
for $t \in [0,T]$, where the first drift term corresponds to a modulation of the exchanges with time.

An important effect for the visual cortex is the so-called potentiation due to intrinsic excitability \cite{Cudmore_2004}: depending on its previous behavior, the sensibility of a neuron to incoming signals can vary. When the neuron was previously highly active, it reaches an excitability state in which incoming signals are magnified. To model this feature, we enrich the coefficients $\bar J_{j,k}$, constant in \eqref{eq:Jansen_Rit_particle}, by including a path-dependent function of the trajectory of the neuron at hand. As a second extension, we include a delay in the signal received by the neurons of population $j$ from the neurons of population $k$. To simplify, we consider the same delay $\triangle$ in each population, but our setting could easily adapt to treat a delay depending on the population (up to a straightforward modification of the initial data). 

We thus consider  
\begin{align}
\label{eq:def_j}
\tau_1 = \tau_2 > 0, \quad \tau_3 > 0, \quad \bar J_{i,j} = \frac{J_{i,j}}{N_j} \qquad \hbox{with } J = \begin{pmatrix}
		J_{1,1} & J_{1,2} & J_{1,3} \\
		J_{2,1} & J_{2,2} & 0 \\
		J_{3,1} & 0 & J_{3,3}
	\end{pmatrix},
\end{align}
where $J_{1,2}, J_{1,3}, J_{2,1}, J_{3,1}$ and $(J_{i,i})_{1 \le i \le 3}$ are functions of $[0,T] \times \mathcal{C}([0,T], \R)$ given by
\begin{align} 
	\label{eq:Jij}
	J_{i,j}\Big( t, (\alpha_s)_{s \in [0,T]} \Big) = D_{i,j} \Big(1 + \varepsilon \int_0^t \varphi(\alpha_s) \d s \Big),
\end{align}
where for all $(i,j) \in \{(1,2), (1,3), (2,1), (3,1) \}$, $D_{i,j}$ are fixed constants and $\varepsilon$ is a small parameter modulating the rate-based plasticity \cite[Section 6.6]{Roth_2009}. The function $\varphi$ is assumed to be bounded and Lipschitz continuous from $\R$ to $\R$.

Note that this extends the model of \cite{Faugeras2009} and that our hypotheses allow for any choice of $J_{i,j}$ that is regular enough (see Assumptions \ref{assump:lipschitz} and \ref{assump:holder}). The setting \eqref{eq:Jij} should be thought of as a toy model illustrating our ability to take into account the effect of the potential trajectory on the postsynaptic strengths. We mention that the justification of neural mass models from the microscopic dynamics is a challenging topic in computational neuroscience \cite{Deschle_2021}.

Ultimately the following microscopic system is considered, for $j \in \{1,2,3\}$, $i \in \{1,\dots, N_j\}$ and $t \in [\triangle,T]$
\begin{align}
	\label{eq:extended_Jansen_Rit_particle}
	\d V_{j,i}(t) &= - \frac1{\tau_j} V_{j,i}(t) \d t + \Big( \sum_{k=1}^3 \sum_{\ell = 1}^{N_k} \bar J_{j,k} \Big(t, \big(V_{j,i}(\cdot)\big)_{\cdot \wedge t} \Big) S\big(V_{k,\ell}(t-\triangle) \big) + I_j(t) \Big) \d t \\
    & \qquad + f_j(t) \d W^{j,i}_t. \nonumber
\end{align}
The functions $I_j, f_j$ from $\R_+$ to $\R$ are assumed to be Lipschitz continuous. The Brownian motions $(W^{j,i}_t)_{t \ge 0}$ for $\{(j,i): j \in \{1,2,3\}, i \in \{1,\dots, N_j\} \}$ are assumed to be mutually independent and independent of initial data, and the function $S$ is given on $\R$ by 
\begin{align}
	\label{eq:def_S} S(v) = \frac{v_{m}}{1 + e^{r(v_0 - v)}},
\end{align}
with $r > 0$ and $0 < v_0 < v_{m}$. Note that this function is bounded and Lipschitz continuous with Lipschitz constant $v_m r$. Initial data are given trajectories $(V_{j,i}(s))_{s \in [0,\triangle]}$ for all $1 \le j \le 3$ and $1 \le i \le N_j$. 

\medskip 

As the number of particles in each population grows to infinity, it is natural to expect the system to be described by the following system of three path-dependent McKean-Vlasov equations. Write $\mu^j_t$ for the distribution of the potential of population $j \in \{1,2,3\}$ at time $t$ in $[0,T]$. In the  mean-field limit, we obtain the following system set on $[\triangle,T]$, for $j \in \{1,2,3\}$,
\begin{align}
	\label{eq:Jansen_Rit_mean_field}
	\left\{ 
	\begin{array}{ll}
		&\d \bar V_j(t) \!=\! \Big\{ - \frac1{\tau_j}\bar V_j(t) + \sum_{k=1}^3 D_{j,k} \Big( 1 + \varepsilon \int_0^t \varphi\big(\bar V_j(u)\big) \d u \Big) \int_{\R} S(y) \mu^k_{t-\triangle}(\d y) \Big\} \d t \\
  &\qquad \qquad \qquad + I_j(t) \, \d t +  f_j(t) \, \d W_t^j 
		\\
		&\bar V_j(t) \sim \mu^j_t, \quad t \in [0,T],
	\end{array}
	\right.
\end{align} 
where $(W^1, W^2, W^3)$ are three independent Brownian motions. 
We summarize our assumptions as follow:

\noindent 

\begin{manualtheorem}{(IV)}\label{assump:old-assum-2.1}
\begin{enumerate}
    \item $ T >  \triangle$;
	\item For $s \in [0, \triangle]$, fix $\bar V_j(s) =  \bar V_j(0) \in\R$ (leading to $\mu^j_s = \delta_{V_j(0)}$, $s \in [0, \triangle]$ in \eqref{eq:Jansen_Rit_mean_field}); 
	\item the functions $I_j , f_j: \R_+ \to \R$ are Lipschitz continuous;
	\item the function $S$ is given by \eqref{eq:def_S};
	\item the function $\varphi$ appearing in the definition of $J_{i,j}$ in \eqref{eq:Jij} is bounded, Lipschitz continuous from $\R$ to $\R$.
\end{enumerate} 
\end{manualtheorem}

The following proposition, whose proof can be found in Appendix \ref{proofsubsec:appli}, shows that the model \eqref{eq:Jansen_Rit_mean_field} fits our setting. 

\begin{prop}\label{prop:ex1}
  Under Assumption~\ref{assump:old-assum-2.1}, the system~\eqref{eq:Jansen_Rit_mean_field} satisfies Assumption~\ref{assump:norm_X0} with moment order $2+\varepsilon_0$ for any $\varepsilon_0>0$, and Assumptions~\ref{assump:lipschitz} and~\ref{assump:holder} hold with $p=2$ and $\gamma=1$. 
\end{prop}

This provides a proof of well-posedness on finite time $[0,T]$ for any $T > 0$ for this upgraded version of the model treated in \cite{Faugeras2009}. In addition, Proposition \ref{prop:ex1} induces
a moment propagation result, in the sense that, still letting $\bar V_j(s) = V_j(0)$ for all $s \in [0,\triangle]$, if $\bar V_j(0) \in L^p(\R)$, $p\geq2$ for $j \in \{1,2,3\}$, then $\bar V_j(t) \in L^p(\R)$  at all time $t \in [0,T]$. 
Proposition \ref{prop:ex1} also provides a justification for the derivation of \eqref{eq:Jansen_Rit_mean_field} from the particle system \eqref{eq:Jansen_Rit_particle}. More precisely, letting for all $s \in [0, T]$, $\tilde \mu_s = \mu_s^1 \otimes \mu_s^2 \otimes \mu_s^3$, the following proposition is a direct result of Theorem \ref{thm:well-posedness}, Theorem \ref{thm:global_thm_chaos} and Proposition \ref{prop:ex1}. 

\begin{prop}
	\label{prop:ex1_2}
	Let $N \in \mathbb{N}^*$ and assume that $N_1 = N_2 = N_3 = N$. Assume that for all $i \in \{1,\dots, N\}$ and for all $s \in [-\triangle, 0]$, $(V_{1,i}(s), V_{2,i}(s), V_{3,i}(s)) = 0_{\R^3}$. Under Assumption \ref{assump:old-assum-2.1}, the particle system \eqref{eq:Jansen_Rit_particle} is well-defined. Moreover, defining  $\mu^N_t := \frac1{N} \sum_{i=1}^N \delta_{(V_{1,i}(t), V_{2,i}(t), V_{3,i}(t))}$, we have 
	\begin{align*} 
		\Big\| \sup_{t \in [0,T]} \mathcal{W}_2 \big( \tilde \mu_t, \mu^N_t \big) \Big\|_2 \,  \underset{N \to \infty}{\longrightarrow} 0.
	\end{align*}
\end{prop}

\subsubsection{Numerical results}

We turn to our numerical simulations. Our choices for the parameters and functions appearing in \eqref{eq:Jansen_Rit_mean_field} are displayed in Table \ref{tab:parameters}.

\begin{table}[ht]
\centering
\renewcommand{\arraystretch}{2}
\begin{tabular}{|l|l|l|} 
\hline
$\varphi(x) = e^{-|x|}$ on $\mathbb{R}$ & $f_j \equiv 1$ for all $j$  & $S(v) = \frac{10}{1 + e^{(1-v)}}$ for all $v \in \R$\\  \hline
 $I_j \equiv 0$ & $\varepsilon = 0.1$  &$\tau = (\tau_j)_{1\le j\le 3}$ with $\tau = (1,1,1)$\\ \hline
$T = 1$ & $D_{1,2} = D_{1,3} = 1$, &$D_{i,i} = 1$ for all $i \in \{1,2,3\}$ \\ \hline
$D_{2,1} = 5$ &$D_{3,1} = -1$  & $D_{2,3} = D_{3,2} = 0$\\ \hline
\end{tabular}
\caption{Parameters values and functions for the simulation of \eqref{eq:Jansen_Rit_mean_field}}
\label{tab:parameters}
\end{table}


We note that we considered $\varepsilon$ to be a fixed positive constant in our simulations, but a population-dependent coefficient $\varepsilon_{i,j}$ (possibly negative) can also be used.
In the setting of Table \ref{tab:parameters}, the model writes
\begin{align*}
\left\{
	\begin{array}{ll}
&\d \bar V_j(t) = \Big\{-\bar V_j(t) + \sum_{k=1}^3 D_{j,k} \Big( 1 + 0.1 \int_0^t e^{-|\bar V_j(u)|} \d u \Big) \int_{\mathbb{R}} \frac{10}{1+e^{1-y}} \mu^k_{t-\triangle}(\d y) \Big\} \d t +  \d W^j_t, \\
		&\bar V_j(t) \sim \mu^j_t, \quad t \in [0,T].
		\end{array}
		\right.
  \end{align*}
We take $M=100$ and $N = 2^j$ for $j \in \{7, \ldots,16\}$. The choice of  $M$ satisfies $M > N^{\frac14 + 0.1}$ for all choices of $N$ considered, which, according to our results from Theorem~\ref{thm:particlemethod}, guarantees that the rate of convergence is only limited by $N$. To challenge our numerical results, we face two main obstacles:
\begin{enumerate}
    \item We do not have access to the true distribution of our model, to which we may compare our approximation. To overcome this issue, we use the simulation with $2^{16}$ particles as a proxy of this true distribution. 
    \item This model is set in dimension 3, rendering the formula \eqref{eq:formula_Wasserstein} inapplicable when considering all coordinates together. 
\end{enumerate}
To solve the second issue mentioned above, we consider two numerical results. In the first one, we compute an approximation of the square of the  Wasserstein distance with respect to the simulation with $2^{16}$ particles \textit{coordinates by coordinates}, using a similar estimate to \eqref{eq:formula_error}, as all of those are one-dimensional. The results are displayed on Figure \ref{fig:Wasserstein_Neural}. As the Wasserstein approximation involves a discretization error, see \eqref{eq:formula_error}, we consider a truncation parameter $\epsilon = 10^{-6}$, and display also the error for $N = 2^{16}$, which quantifies the truncation error. We observe a consistent decay with the rise of $N$. Slopes of each line presented here between $2^{7}$ and $2^{15}$ range between $-3.16$ (coordinate 2) and $-2.19$ (coordinate 1). 

\begin{figure}[htbp]
\caption*{Figures 4-5: Numerical results for the model \eqref{eq:Jansen_Rit_mean_field}.}
\vspace{1cm}
    \centering
    \begin{minipage}[c][0.3\textheight][c]{0.67\textwidth}
        \centering
        \includegraphics[width=\textwidth]{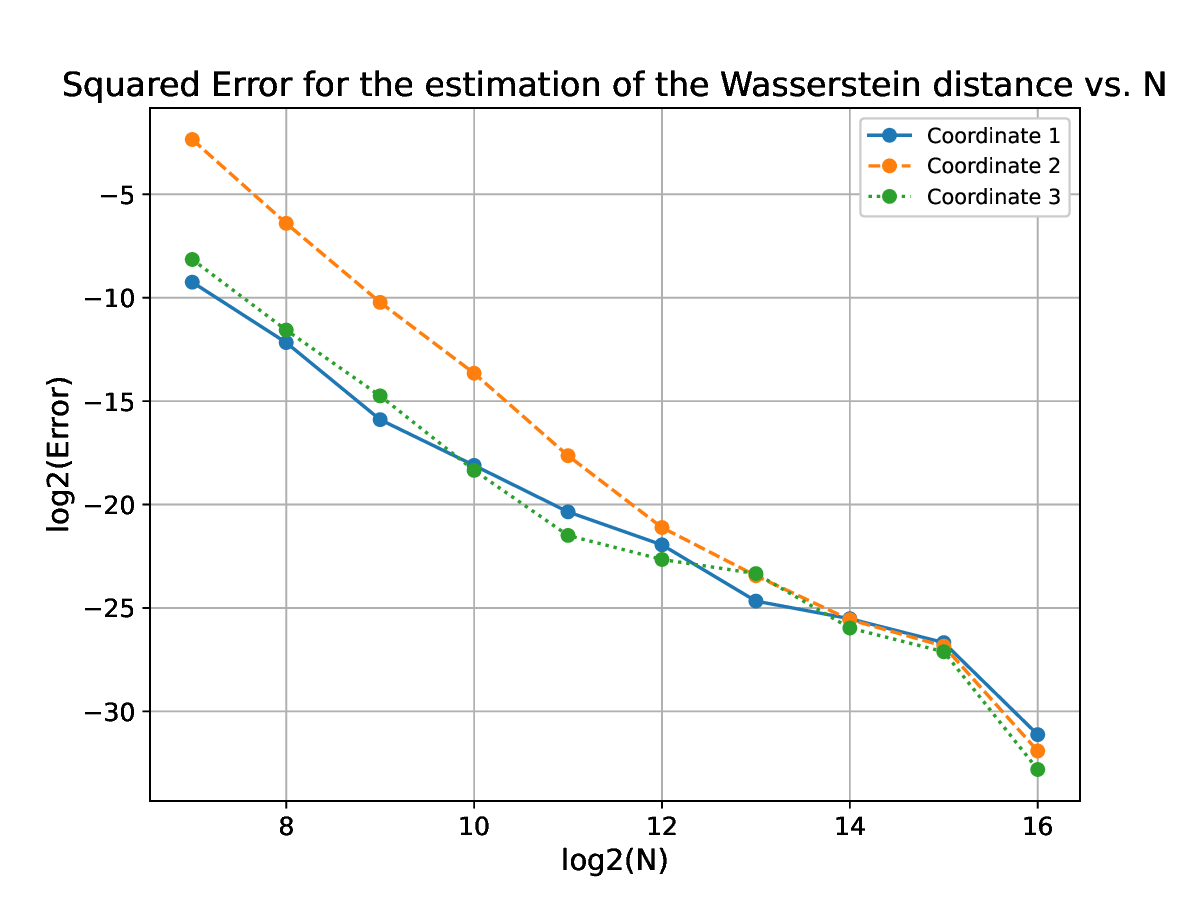}
    \end{minipage}
    \hfill
    \begin{minipage}[c][0.3\textheight][c]{0.32\textwidth}
        \raggedright
        \small
        \refstepcounter{figure}
        \textbf{Figure \thefigure.} $\log_2-\log_2$ scale estimation of the Wasserstein distance for each coordinate. Simulations were performed with a fixed value $M = 100$ and $N$ ranging from $2^7$ to $2^{16}$. Linear regression of the data gives $y = -2.19 \, x + 4.64$ (coordinate $1$), $y = -3.16 \, x + 18.4$ (coordinate $2$) and $y = -2.34 \, x + 6.50$ (coordinate 3).
        \label{fig:Wasserstein_Neural}
    \end{minipage}
    \vskip\baselineskip 
    \vspace{2cm} 
    \begin{minipage}[c][0.3\textheight][c]{0.67\textwidth}
        \centering
        \includegraphics[width=\textwidth]{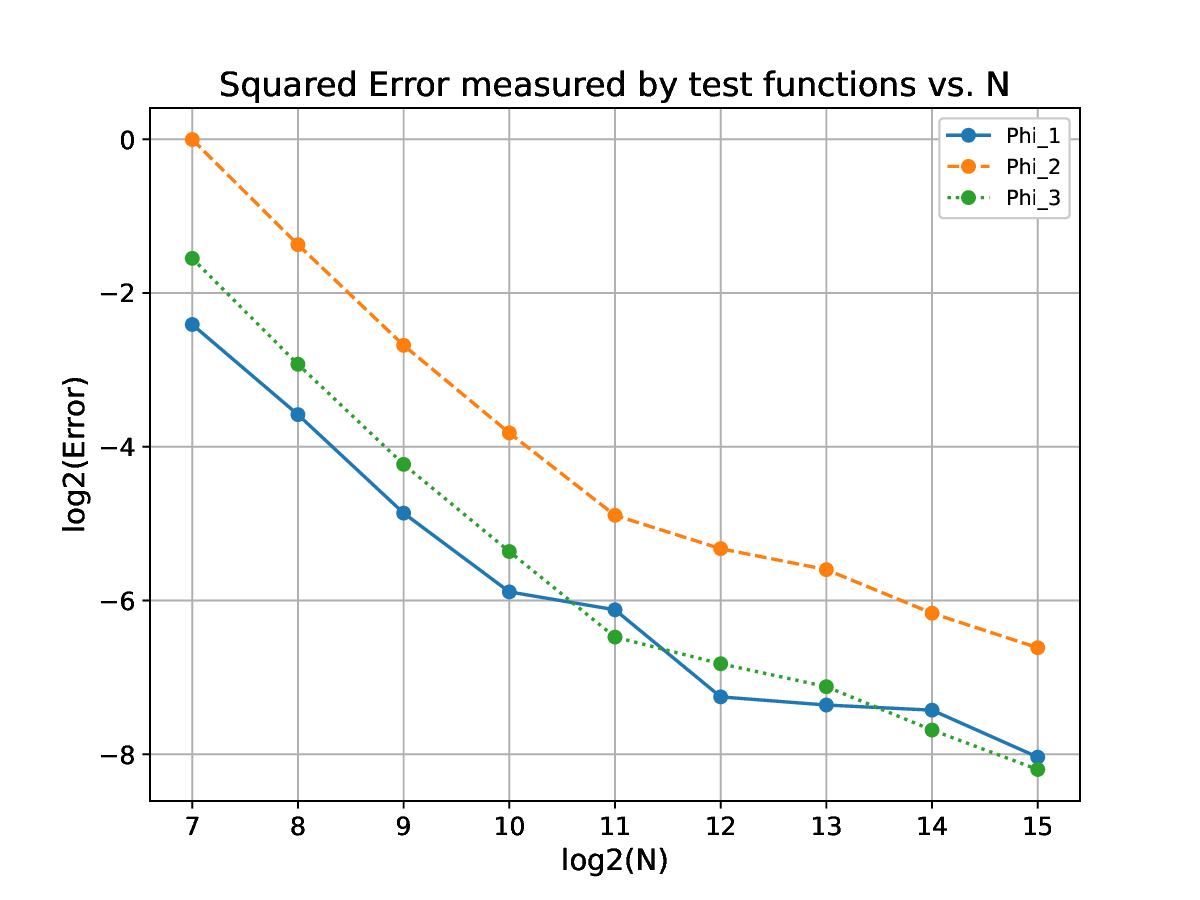}
    \end{minipage}
    \hfill
    \begin{minipage}[c][0.3\textheight][c]{0.32\textwidth}
        \raggedright
        \small
        \refstepcounter{figure}
        \textbf{Figure \thefigure.} $\log_2-\log_2$ scale estimation via three test functions. Simulations were performed with a fixed value $M = 100$ and $N$ ranging from $2^7$ to $2^{16}$. Linear regression of the data gives $y = -0.673 \, x + 1.52$ ($\Phi_1$), $y = -0.802 \, x + 4.78$ ($\Phi_2$) and $y = -0.802\, x + 	3.22$ ($\Phi_3$).
        \label{fig:Neural_test_function}
    \end{minipage}
\end{figure}

To estimate the global Wasserstein distance (that is, involving all three coordinates at once), we use a test function method. More precisely, we rely on the fact that the $L^2$-Wasserstein distance defines, in any finite dimension $d$, the same topology as the $2$-Zolotarev distance $d_Z$ \cite[Proposition 1]{Belili_2000}, where $d_Z$ is defined as the distance between measures $\mu, \nu$ such that
\[ d_Z(\mu,\nu) := \sup \Big\{ \int_{\R^d} g(x) (\mu - \nu)(\d x) : g \in C^2_b(\R^d), g'(0) = 0, \|g''\|_{L^\infty(\R^d)} = 1 \Big\} .\]
Moreover, for two measures $\mu, \nu$ on $\R^d$, we recall from \cite[Theorem 2]{Belili_2000} the inequality
\[ \mathcal{W}_2^2(\mu,\nu) \le 8 d_Z(\mu, \nu). \]
We introduce now the test functions defined, for all $x \in \R^d$, by
\[ \Phi_1(x) = e^{-\frac1{1-|x|^2}}, \, \qquad \Phi_2(x) = e^{-\frac{|x|^2}{2}}, \, \qquad \Phi_3(x) = \frac1{1+e^{-|x|}}. \] 
For $(\mu_t)_{t \in [0,T]}$ the true distribution of the model and $\hat \mu^N_t = \frac1N \sum_{i=1}^N \delta_{X^i_t}$, $t \in [0,T]$, the empirical distribution obtained through our simulations, for all $\ell \in \{1,2,3\}$, we thus have
\[ d_Z(\hat \mu^N_t, \mu_t) \ge \frac1{N} \sum_{i=1}^N \Phi_\ell(X_t^i) - \int_{\R^d} \Phi_\ell(x) \mu_t(\d x). \]
Hence, we introduce, for $N_1 \in \{2^7, \ldots, 2^{15}\}$, $\ell \in \{1,2,3\}, N_f = 2^{16}$, $N_{MC} \in \mathbb{N}^*$, the estimator $\hat E_N$ given by
\begin{align*}
    \hat E_N(N_1, \ell, N_{MC}) = 8 \frac{1}{N_{MC}} \sum_{j=1}^{N_{MC}} \Big| \frac1{N_1} \sum_{k=1}^{N_1} \Phi_\ell(X^{N_1}_T) - \frac1{N_f} \sum_{k=1}^{N_f} \Phi_\ell(X^{N_f}_T) \Big|.
\end{align*}
The results are displayed in Figure \ref{fig:Neural_test_function}.  We observe, for all test functions considered, a steady convergence for smaller values of $N$, although the rate of convergence seems slower for values larger than $2^{12}$, while the error itself is still important, being of order $2^{-8}$ instead of $2^{-26}$ for the previous estimation based on coordinates. Most likely, this is due to the fact that our choices of test functions do not approximate well the Zolotarev distance $d_Z$.


\section{Proof for the convergence rate of the particle method}\label{sec:numericalcvg}


In Section \ref{subsec:prop_im}, we gather several preliminary results that will be used for the proof of Theorem \ref{thm:particlemethod}.  Next, in Section \ref{sec:Euler_scheme}, we study the convergence of the interpolated Euler scheme \eqref{eq:discretescheme} and of its continuous counterpart \eqref{eq:def_continuous_2}. Finally, Section \ref{subsec:cvgparticlemethod} is devoted to the proof of Theorem \ref{thm:particlemethod}.

\subsection{Preliminary results}\label{subsec:prop_im}

In this subsection, we introduce the properties of the interpolator $i_m$ and several preliminary results essential for establishing Theorem \ref{thm:particlemethod}. The detailed proofs of the lemmas presented here can be found in Appendix \ref{proofsubsec:prop_im}.  For any $t\in[0, T]$, we define $\pi_{t}: \CRD\rightarrow\mathbb{R}^{d}$ by 
\begin{equation}\label{eq:pi_coordinate}
\alpha\mapsto\pi_{t}(\alpha)=\alpha_{t}.
\end{equation} The following lemma, and its proof, can be found in \cite[Lemmas 5.1.2 and 5.1.3]{Liu_PhD}.

\begin{lem}\label{injectionmeasure}
The application $\iota: \PPC \rightarrow \CPP$ defined by \[\mu\mapsto \iota(\mu)=(\mu\circ\pi_{t}^{-1})_{t\in[0, T]}=(\mu_{t})_{t\in[0, T]}\] is well-defined and $1$-Lipschitz continuous. 
\end{lem}

For two probability measures $\mu, \nu\in\PPRD$ and for $\lambda\in[0,1]$, we define $\lambda \mu+(1-\lambda)\nu$ by 
\[\forall \,B\in\mathcal{B}\big(\RD\big), \quad \big(\lambda \mu+(1-\lambda)\nu\big)(B)\coloneqq\lambda \mu(B)+(1-\lambda)\nu(B).\]
It is easy to check that $\lambda \mu+(1-\lambda)\nu\in\mathcal{P}_p(\RD)$.

\begin{lem}\label{combcovprop}Let $\mu,\,\nu\in\PPRD$ with $p\geq1$. We define the application $\tau$ by \[\tau: \lambda \in [0,1]\mapsto \tau(\lambda) =  \lambda \mu+(1-\lambda)\nu\in\PPRD.\]
\begin{enumerate}[$(a)$]
\item The application $\tau$ is $\frac{1}{p}$-H\"older continuous with respect to the Wasserstein distance $\mathcal{W}_p$ i.e. 
		\[\forall \lambda_1, \lambda_2\in[0,1], \quad \mathcal{W}_{p}\big( \tau(\lambda_1), \tau(\lambda_2) \big)\leq |\lambda_1-\lambda_2|^{\frac{1}{p}}\mathcal{W}_p(\mu, \nu).\]  
		\item Let $\delta_0$ denote the Dirac measure at $0\in\RD$. Then
		\[\sup_{\lambda\in[0,1]}\,\mathcal{W}_p\big( \tau(\lambda), \delta_0\big)\leq \mathcal{W}_p( \mu, \delta_0)\vee \mathcal{W}_p( \nu, \delta_0).\]
	\end{enumerate}
\end{lem}

Remark that Lemma \ref{combcovprop} implies that the interpolator $i_m$ from \eqref{definterp2} and \eqref{definterp2bis} is well- defined. The following two results describe further its properties.

\begin{lem}[Properties of the interpolator $i_m$]\label{interpolatorprop} Let $m \in \mathbb{N}^*$. 
	\begin{enumerate}[$(a)$]
		\item For every $x_{0:m}\in (\RD)^{m+1}$, $\Vert i_m(x_{0:m})\Vert_{\sup}=\sup_{0\leq k\leq m}|x_k|$. 
		\item For every $\mu_{0:m}\in \big(\PPRD\big)^{m+1}$, $\sup_{t\in[0,T]}\mathcal{W}_p\big( i_m(\mu_{0:m})_t, \delta_0\big)=\sup_{0\leq k\leq m}\mathcal{W}_p(\mu_k, \delta_0)$.
	\end{enumerate}
\end{lem}

	

\begin{lem}\label{lem:inter}
\begin{enumerate}[$(a)$]
\item For every  $ x_1, x_2, y_1, y_2\in \RD$ and for every $ \lambda\in[0,1]$, let $x_\lambda\coloneqq\lambda x_1+(1-\lambda)x_2$ and $y_\lambda\coloneqq\lambda y_1+(1-\lambda)y_2,$ 
we have $|x_\lambda-y_\lambda|\leq \max (|x_1-y_1|, |x_2-y_2|)$.
\item For every $\mu_1, \mu_2, \nu_1, \nu_2\in \mathcal{P}_p(\RD)$ and for every $\lambda\in[0,1]$, let $\mu_\lambda=\lambda \mu_1+(1-\lambda)\mu_2$ and $\nu_\lambda=\lambda \nu_1+(1-\lambda)\nu_2$, we have $\mathcal{W}_p(\mu_\lambda, \nu_\lambda)\leq \max\big(\mathcal{W}_p(\mu_1, \nu_1), \mathcal{W}_p(\mu_2, \nu_2)\big)$.
\item For every $x_{0:m}, y_{0:m}\in (\RD)^{m+1}$,  we have $\Vert i_m (x_{0:m})-i_m(y_{0:m})\Vert_{\sup}\leq  \max_{0\leq \ell\leq m}|x_\ell-y_\ell|$. 
\item For every $\mu_{0:m}, \nu_{0:m}\in \big(\mathcal{P}_p(\RD)\big)^{m+1}$, we have \[ \sup_{t\in[0,T]} \mathcal{W}_p\big(i_m (\mu_{0:m})_t, i_m(\nu_{0:m})_t\big) \leq  \max_{0\leq \ell\leq m}\mathcal{W}_p(\mu_\ell, \nu_\ell).\]
\end{enumerate}
\end{lem}

 The next result is a direct consequence of Assumptions \ref{assump:norm_X0} and \ref{assump:lipschitz}, that we shall use several times in our proof of Theorem \ref{thm:particlemethod}.  

\begin{lem}\label{lineargrowth} Let $p\geq 2$. Assume that Assumptions \ref{assump:norm_X0} and \ref{assump:lipschitz} hold with $\mathfrak{p}=p$. Then the coefficient functions $b$ and $\sigma$ have a linear growth in $\alpha$ and in $(\mu_t)_{t\in[0,T]}$ in the sense that there exists a constant $C_{b,\sigma, L, T} > 0$ s.t. for every $\,t\in[0,T]$, $\alpha\in\CRD$, $(\mu_{t})_{t\in[0,T]}\in\CPP$,
\begin{align}\label{lineargr}
& \big| b(t, \alpha, (\mu_{t})_{t\in[0,T]})\big| \vee \vertiii{\sigma(t, \alpha, (\mu_{t})_{t\in[0,T]})}  \leq C_{b,\sigma, L, T}\Big(1+\Vert \alpha\Vert_{\sup}+\sup_{t\in[0,T]}\mathcal{W}_{p}(\mu_t, \delta_0)\Big).
\end{align}
\end{lem}

In the last part of this section, we present important technical tools from the literature.
We begin with the generalized Minkowski's Inequality and the Burk\"older-Davis-Gundy Inequality. For the proof of these two inequalities, we refer to \cite[Section 7.8]{pages2018numerical} among other references. 
\begin{lem}[The Generalized Minkowski's Inequality]\label{gemin}
For any (bi-measurable) process $X=(X_{t})_{t\geq0}$, for every $p\in[1,\infty)$ and for every $ T\in[0, +\infty],$
\[\left\Vert \int_{0}^{T}X_{t}\, \d t\right\Vert_{p}\leq\int_{0}^{T}\left\Vert X_{t}\right\Vert_{p}\d t.\]
\end{lem}
\begin{lem}[Burk\"older-Davis-Gundy Inequality (continuous time)]\label{BDGin}
For every $p\in(0, +\infty)$, there exist two real constants $c_{p}^{BDG}>0$ and $C_{p}^{BDG}>0$ such that, for every continuous local martingale $(X_{t})_{t\in[0, T]}$ null at 0, denoting $(\langle X \rangle_t)_{t \in [0,T]}$ its total variation process, \dd 
\[c_{p}^{BDG}\left\Vert \sqrt{\langle X \rangle_{T}}\right\Vert_{p}\leq\left\Vert \sup_{t\in[0,T]}\left|X_{t}\right|\right\Vert_{p}\leq C_{p}^{BDG}\left\Vert \sqrt{\langle X\rangle_{T}}\right\Vert_{p}.\]
\end{lem}

Note that under Assumptions \ref{assump:norm_X0} and \ref{assump:lipschitz}, $t \mapsto \sigma(t,X_{\cdot \wedge t}, \mu_{\cdot \wedge t})$ is adapted and continuous, hence progressively measurable. 
Recall also that $p \ge 2$.  A direct application of those two inequalities provides the following lemma.

\begin{lem}
\label{lem:BDG_for_sigma}
	Let $(B_t)_{t \in [0,T]}$ be an 
	$(\mathcal{F}_t)_{t \in [0,T]}$-standard Brownian motion, and $(H_t)_{t \in [0,T]}$ be an $(\mathcal{F}_t)_{t \in [0,T]}$ progressively measurable process having values in $\mathbb{M}_{d,q}(\mathbb{R})$ such that $\int_0^T \vertiii{H_t}^2 \d t <\! \infty$, $\mathbb{P}$-a.s.. Then, for all $t \in [0,T]$,
\begin{align*}
\left\Vert \sup_{s \in [0,t]} \Big| \int_0^s H_u \d B_u \Big| \right\Vert_p \le C_{d,p}^{BDG} \Big[ \int_0^t \Big\| \vertiii{H_u} \Big\|_p^2 \d u \Big]^{\tfrac12}.
\end{align*}
\end{lem}

 We will also make use of the following version on Gr\"onwall's lemma, whose proof is given in \cite[Lemma 7.3]{pages2018numerical}, and 
 Theorem \ref{FG} from \cite{fournier2015rate}, which provides a non-asymptotic upper bound of the convergence rate in the Wasserstein distance of the empirical measures of i.i.d. random vectors.
\begin{lem}[``\`A la Gronwall" Lemma]\label{Gronwall}
	Let $f : [0, T]\rightarrow\mathbb{R}_{+}$ be a Borel, locally bounded and non-decreasing function and let $\psi: [0, T]\rightarrow\mathbb{R}_{+}$ be a non-negative non-decreasing function satisfying 
	\[\forall t\in[0, T],\,\, f(t)\leq A\int_{0}^{t}f(s)\d s+B\left(\int_{0}^{t}f^{2}(s)\d s\right)^{\frac{1}{2}}+\psi(t),\]
	where $A, B$ are two positive real constants. Then, for any $t\in[0, T],$ \[ f(t)\leq 2e^{(2A+B^{2})t}\psi(t).\]
\end{lem}

\begin{thm}(\cite[Theorem 1]{fournier2015rate})\label{FG}
	Let $p>0$ and let $\mu\in\mathcal{P}_{q}(\mathbb{R}^{d})$ for some $q>p$. Let $n \ge 1$ and $U^{1}, \dots, U^{n},  \dots$ be i.i.d random variables with distribution $\mu$. Let $\mu_{n}$ denote the empirical measure of $\mu$ defined by
	\[\mu_{n}\coloneqq \frac{1}{n}\sum_{i=1}^{n}\delta_{U^{i}}.\]
	Then, there exists a real constant $C > 0$ only depending on $p,d,q$ such that, for all $n\geq 1$,
	\[\mbox{\small $\mathbb{E}\Big(\mathcal{W}_{p}^{p}(\mu_{n}^{\omega}, \mu)\Big)\leq CM_{q}^{p/q}(\mu)\times\begin{cases}
			n^{-1/2}+n^{-(q-p)/q} & \:\mathrm{if}\:p>d/2\:\mathrm{and}\:q\neq2p, \\
			n^{-1/2}\log(1+n)+n^{-(q-p)/q} & \:\mathrm{if}\:p=d/2\:\mathrm{and}\:q\neq2p,\\
			n^{-p/d}+n^{-(q-p)/q} & \:\mathrm{if}\:p\in(0, d/2)\:\mathrm{and}\:q\neq dp/(d-p),
	\end{cases}$}\]
	where $M_{q}(\mu)\coloneqq \int_{\mathbb{R}^{d}}\left|\xi\right|^{q}\mu(\d \xi)$.
\end{thm}


 We mention that Fournier \cite{Fournier_2023} recently obtained an explicit value of $C$ for the previous theorem. 

Our last preliminary result adjusts \cite[Lemma 7.4]{pages2018numerical} to our context. For $H \in \mathbb{M}_{d,q}(\R)$, we write $\|H\|_p$ for $\| \vertiii{H} \|_p$. 

\begin{lem}\label{lem:lemma7.4-pages2018numerical}
Let $p\geq 2$ and let $(Y_t)_{t\in[0,T]}$ be an $\RD-$valued It\^o process defined on $(\Omega, \mathcal{A}, \mathbb{P})$ by 
\[Y_t=Y_0+\int_{0}^tG_s \, \d s +\int_0^t H_s \, \d B_s, \quad t\in[0,T],\]
where $G$ and $H$ are $(\mathcal{F}_t)$-progressively measurable having values in $\RD$ and $\mathbb{M}_{d, q}(\R)$ respectively and satisfying $\sup_{t\in[0,T]}\big[\Vert G_t\Vert_p\vee \Vert H_t\Vert_p\big]<+\infty$. Then for every $s,t\in [0,T]$, 
\[\Vert Y_t-Y_s\Vert_{p}\leq \Big(C_{d,p}\sup_{t\in[0,T]}\Vert H_t\Vert_p+\sqrt{T}\sup_{t\in[0,T]}\Vert G_t\Vert_p\Big)|t-s|^{\frac{1}{2}}.\]
\end{lem}

\subsection{The convergence rate of the interpolated Euler scheme}
\label{sec:Euler_scheme}

Let $M \in \mathbb{N}^*$. According to the definition of $b_m$ and $\sigma_m$ in \eqref{defbmsigmam}, 
the continuous Euler scheme \eqref{eq:def_continuous_2} writes, for $m \in \{0, \dots, M-1\}$ and $t \in (t_m, t_{m+1}]$, 
\begin{align*} \widetilde X^{h}_t &= \widetilde X^{h}_{t_m} + (t-t_m) \,b\,\Big(t_m, i_m\big(\widetilde X^{h}_{t_0:t_m} \big), i_m \big( \widetilde \mu^{h}_{t_0:t_m} \big) \Big)  + \sigma_m \Big(t_m, i_m\big( \widetilde X^{h}_{t_0:t_m} \big), i_m \big(\widetilde \mu^{h}_{t_0:t_m} \big) \Big) \big(B_t - B_{t_m} \big). 
	\end{align*}
In order to compare this with equation \eqref{eq:path-dependent_McKean}, we write, for all $t \in [0,T]$, $\widetilde \mu^{h}_t$ for the distribution of $\widetilde X^{h}_t$, and for all $m \in \{0,\dots, M-1\}$, we set
\begin{equation}
	\label{eq:t_bar}
	\underline t:= t_m, \qquad [\underline t] := m \qquad \hbox{if} \quad t \in [t_m, t_{m+1}). 
\end{equation}
With this at hand, the process $(\widetilde X^{h}_t)_{t \in [0,T]}$ defined by \eqref{eq:def_continuous_2} satisfies 
\begin{align}
\label{eq:rmk_tilde}
\widetilde X^{h}_t &= \widetilde X^{h}_0 + \int_0^t b \Big(\underline s, i_{[\underline s]}\big(\widetilde X^{h}_{t_0:t_{[\underline s]}} \big), i_{[\underline s]}\big(\widetilde \mu^{h}_{t_0:t_{[\underline s]}} \big) \Big)\, \d s + \int_0^t \sigma \Big(\underline s, i_{[\underline s]}\big(\widetilde X^{h}_{t_0:t_{[\underline s]}} \big), i_{[\underline s]}\big(\widetilde \mu^{h}_{t_0:t_{[\underline s]}} \big) \Big)\, \d B_s.
\end{align}

The goal of this section is to prove a convergence result for the interpolated Euler scheme, Proposition \ref{prop:cvg_cont_Euler}, and the associated Corollary \ref{cor2}, both given below. More precisely, we start in Section \ref{subsec:prop_x_tilde} by deriving a key preliminary result, Proposition \ref{propXtilde} and proceed to the proof of Proposition \ref{prop:cvg_cont_Euler} and Corollary \ref{cor2} in Section \ref{subsec:3proofs}.  


\begin{prop}[Convergence rate of the interpolated Euler scheme]\label{prop:cvg_cont_Euler}
Let $p\geq 2$. Assume that Assumption \ref{assump:norm_X0} holds with $\mathfrak{p}=p+\varepsilon_0$ for some $\varepsilon_0>0$ and that Assumptions \ref{assump:lipschitz} and \ref{assump:holder} hold with $\mathfrak{p}=p$.  
Let $(X_t)_{t \in [0,T]}$ be the unique strong solution to \eqref{eq:path-dependent_McKean} and let $(\widetilde X_t^h)_{t \in [0,T]}$ be the process defined by \eqref{eq:def_continuous_2}. If $M \ge 2T+1$, one has
\begin{equation}
\vertii{\sup_{t\in[0,T]}\left|X_{t}-\widetilde{X}_{t}^{h}\right|}_{p}\leq\tilde{C} \Big( h^{\gamma} + \big(h |\log(h) \big|  \big)^{\tfrac12} \Big),
\end{equation}
where $\tilde{C} > 0$ is a constant depending on $L, p, \varepsilon_0, d, \vertii{X_{0}}_{p+\varepsilon_0}, T$ and $\gamma$.
\end{prop}

From Definition \ref{def:discretization_scheme}, we can introduce a continuous extension of $(\widetilde{X}^{h}_{t_m})_{0\leq m\leq M}$, denoted by $\widehat{X}^{h}=(\widehat{X}_{t}^{h})_{t\in[0,T]}$ and defined by $\widehat{X}^{h}\coloneqq i_M(\widetilde{X}^{h}_{t_0 : t_M})$. We then have the following convergence. 
\begin{cor}\label{cor2}
Under the same assumptions as in Proposition~\ref{prop:cvg_cont_Euler}, one has
	\begin{equation}
		\vertii{\sup_{t\in[0,T]}\left|X_{t}-\widehat{X}_{t}^{h}\right|}_{p}\leq\tilde{C}\Big( h^{\gamma} + \big(h |\log(h) \big|  \big)^{\tfrac12} \Big),
	\end{equation}
	where $\tilde{C} > 0$ is a constant depending on $L, p, \varepsilon_0, d, \vertii{X_{0}}_{p+\varepsilon_0}, T$ and $\gamma$. 
\end{cor}

\subsubsection{Properties of the  continuous extension process $(\widetilde X^h_t)_{t \in[0,T]}$}\label{subsec:prop_x_tilde}

We gather here several properties of the process $(\widetilde X^h_t)_{t \in[0,T]}$ defined by \eqref{eq:def_continuous_2}. 


\begin{prop}\label{propXtilde} For all $M \in \mathbb{N}^*$, write $(\widetilde{X}_{t}^h)_{t\in[0,T]}$ for the process defined by \eqref{eq:def_continuous_2} with parameter $M$. Let $p\geq 2$.  Then 
\begin{enumerate}[$(a)$]
\item  Assume that Assumptions \ref{assump:norm_X0} and \ref{assump:lipschitz} hold with $ \mathfrak{p}=p$. Then for every $M\in\mathbb{N}^{*}$, we have \[\Big\Vert \sup_{t\in[0,T]}\big|\widetilde{X}_t^h\big|\Big\Vert_{p}\leq \Gamma \big(1+\Vert X_0\Vert_{p} \big),\] where $\Gamma$ is a constant depending on  $p,d,b,\sigma,L,T$. 
\item Assume that Assumption \ref{assump:norm_X0} holds with $\mathfrak{p}=p+\varepsilon_0$ for some $\varepsilon_0>0$, and that Assumptions \ref{assump:lipschitz} and \ref{assump:holder} hold with $ \mathfrak{p}=p$.  Let $M \ge 2T+1$. Then,  
there exists a constant $\kappa > 0$ depending on $L, b, \sigma, \left\Vert X_{0}\right\Vert_{p+\varepsilon_0}, p, \varepsilon_0, d, T$ such that there holds
\[\left\Vert\; \sup_{0\leq m\leq M-1}\;\sup_{v\in[t_m, t_{m+1}]}\Big| \widetilde{X}_v^h-\widetilde X_{t_{m}}^h\Big|\;\right\Vert_p \leq \kappa \, \big(h \big|\log (h) \big| \big)^{\frac{1}{2}}.\]
\item Assume that Assumptions \ref{assump:norm_X0} and \ref{assump:lipschitz} hold with $ \mathfrak{p}=p$.   Let $M \ge 2T+1$. Then, there exists a constant $\tilde\kappa > 0$ depending on $L, b, \sigma, \left\Vert X_{0}\right\Vert_{p}, p,  d, T$ such that for every $s,t\in[0,T]$, 
\begin{equation}\label{eq:X-tilde-lp-holder}
\left\Vert\widetilde{X}_t^h-\widetilde X_{s}^h\right\Vert_p\leq \tilde\kappa|t-s|^{\frac{1}{2}}.
\end{equation}
\end{enumerate}
\end{prop}
Proposition \ref{propXtilde} directly implies the following result.
\begin{cor}\label{propXtilde-bis} Assume that Assumption \ref{assump:norm_X0} holds with $\mathfrak{p}=p+\varepsilon_0$ for some $\varepsilon_0>0$, and that Assumptions \ref{assump:lipschitz} and \ref{assump:holder} hold with $ \mathfrak{p}=p$.  Let $M \ge 2T+1$. Then,  
\begin{align}
&\left\Vert \;\Big\Vert \widetilde{X}^h-i_M\big(\widetilde{X}^h_{t_0: t_M}\big)\Big\Vert_{\sup}\;\right\Vert_p\leq 2\kappa \, \big(h\big|\log(h)\big| \big)^{\frac{1}{2}}\: \: \\
&\text{ and }\:\: \sup_{t\in[0,T]}\mathcal{W}_p\Big( \widetilde{\mu}^h_{t}, i_{M}\big(\widetilde{\mu}^h_{t_{0}:t_{M}}\big)_t\Big)\leq 3\kappa \, \big(h\big|\log(h)\big| \big)^{\frac{1}{2}}.\nonumber
\end{align}
\end{cor}

\begin{proof}[Proof of Corollary \ref{propXtilde-bis}] Let $M$ be a fixed integer with $M \ge 2T+1$. We drop the superscript $h$ in $\widetilde{X}^h$ for simplicity. Clearly
\begin{align}
\Big\Vert \widetilde{X}-i_M\big(\widetilde{X}_{t_0: t_M}\big)\Big\Vert_{\sup} &\leq \sup_{0\leq m\leq M-1}\sup_{t\in[t_m, t_{m+1}]} \left[\Big|\widetilde{X}_t-\widetilde{X}_{t_m}\Big|+\Big|i_M\big(\widetilde{X}_{t_0: t_M}\big)_t-\widetilde{X}_{t_m}\Big|\right]\nonumber\\
& \leq \sup_{0\leq m\leq M-1}\sup_{t\in[t_m, t_{m+1}]} \left[\Big|\widetilde{X}_t-\widetilde{X}_{t_m}\Big|+\Big|\widetilde{X}_{t_{m+1}}-\widetilde{X}_{t_m}\Big|\right]\nonumber\\
& \leq 2 \sup_{0\leq m\leq M-1}\sup_{t\in[t_m, t_{m+1}]}\Big|\widetilde{X}_t-\widetilde{X}_{t_m}\Big|.\nonumber
\end{align}
The first inequality then follows by Proposition \ref{propXtilde}-(b).

	Consider now random variables $(U_{m})_{\,0\leq m\leq M}$ i.i.d. having uniform distribution on $[0,1]$ and independent of the process $(\widetilde{X}_{t})_{t\in[0,T]}$. For every $m\in\{0, ..., M-1\}$ and for every $t\in[t_m, t_{m+1}]$, 
	
	\[ \mathbbm{1}_{\left\{U_m> \frac{t-t_m}{h}\right\}}\widetilde{X}_{t_{m}}+\mathbbm{1}_{\left\{U_m\leq \frac{t-t_m}{h}\right\}}\widetilde{X}_{t_{m+1}} \sim  i_{M}(\widetilde{X}_{t_{0}:t_{M}})_{t}. \]
	This entails
	\begin{align}
		\sup_{t\in[0,T]}\mathcal{W}_{p}\Big(\widetilde{\mu}_{t}, i_{M}\big(\widetilde{\mu}_{t_{0}:t_{M}}\big)_{t}\Big) &\leq \sup_{0\leq m\leq M-1}\sup_{t\in[t_m,t_{m+1}]}\Big\Vert  \widetilde{X}_{t}-\mathbbm{1}_{\left\{U_m> \frac{t-t_m}{h}\right\}}\widetilde{X}_{t_{m}}-\mathbbm{1}_{\left\{U_m\leq \frac{t-t_m}{h}\right\}}\widetilde{X}_{t_{m+1}}\Big\Vert_p\nonumber\\
		&\leq \sup_{0\leq m\leq M-1}\sup_{t\in[t_m,t_{m+1}]} \Big( \Big\Vert  (\widetilde{X}_{t}-\widetilde{X}_{t_{m}})\mathbbm{1}_{\left\{U_m> \frac{t-t_m}{h}\right\}}\Big\Vert_p \nonumber \\
		&\hspace{3.5cm}  +\Big\Vert(\widetilde{X}_{t}-\widetilde{X}_{t_{m+1}})\mathbbm{1}_{\left\{U_m\leq \frac{t-t_m}{h}\right\}}\Big\Vert_p \Big) \nonumber\\
		&\leq 3 \sup_{0\leq m\leq M-1}\sup_{t\in[t_m,t_{m+1}]}\Big\Vert  \widetilde{X}_{t}-\widetilde{X}_{t_{m}}\Big\Vert_p\leq 3\kappa \, \big(h\big|\log(h)\big| \big)^{\frac{1}{2}},\nonumber
	\end{align}
	where the last inequality comes from Proposition \ref{propXtilde}-(b) and the constant 3 in the last two inequalities follows from the fact that for every fixed $t\in[t_m, t_{m+1}]$, the term $\Big\Vert\widetilde{X}_{t}-\widetilde{X}_{t_{m+1}}\Big\Vert_p $ can be upper bounded using the following inequality 
    \[\left\Vert\widetilde{X}_{t}-\widetilde{X}_{t_{m+1}}\right\Vert_p\leq \left\Vert\widetilde{X}_{t}-\widetilde{X}_{t_{m}}\right\Vert_p+\left\Vert\widetilde{X}_{t_m}-\widetilde{X}_{t_{m+1}}\right\Vert_p\leq 2\sup_{t\in[t_m, t_{m+1}]}\left\Vert\widetilde{X}_{t}-\widetilde{X}_{t_{m}}\right\Vert_p.\hfill\qedhere\] 
\end{proof}
We turn to the proof of Proposition~\ref{propXtilde}.
\begin{proof}[Proof of Proposition \ref{propXtilde}] We drop the superscript $h$ in $\widetilde{X}^h$ and in $\widetilde{\mu}^h$ for simplicity.

\noindent$(a)$ {\bf Step 1.} We prove that for every fixed $M\in\mathbb{N}^{*}$,
\begin{equation}\label{eq:discretetilde}
\Big\Vert \;\sup_{0\leq k \leq M}\big|\,\widetilde{X}_{t_k}\,\big|\,\Big\Vert_{p}<+\infty
\end{equation}
by induction. 
First, $\Vert \widetilde{X}_{t_{0}}\Vert_{p}=\Vert X_0\Vert_{p} <+\infty$ by Assumption \ref{assump:norm_X0}. Now assume that, for some $l \in \llbracket M-1\rrbracket$, $\Big\Vert \sup_{0\leq k \leq l} |\widetilde{X}_{t_{k}}| \Big\Vert_{p}<+\infty$. It follows, using also Minkowski's inequality, that
\begin{align*}
&\Big\Vert \sup_{0\leq k \leq l+1} \big|\widetilde{X}_{t_{k}}\big| \Big\Vert_{p} \leq \Big\Vert \sup_{0\leq k\leq l}\big|\widetilde{X}_{t_{k}}\big|\Big\Vert_{p}+\Big\Vert\Big( \big|\widetilde{X}_{t_{l+1}}\big|-\sup_{0\leq k\leq l}\big|\widetilde{X}_{t_{k}}\big|\Big)_{+}\Big\Vert_{p}  \\
& \leq \Big\Vert \sup_{0\leq k\leq l}\big|\widetilde{X}_{t_{k}}\big|\Big\Vert_{p}+\Big\Vert\Big| \;\big|\widetilde{X}_{t_{l+1}}\big|-\big|\widetilde{X}_{t_{l}}\big|\Big|\;\Big\Vert_{p} \leq \Big\Vert \sup_{0\leq k\leq l}\big|\widetilde{X}_{t_{k}}\big|\Big\Vert_{p}+\Big\Vert\widetilde{X}_{t_{l+1}}-\widetilde{X}_{t_{l}}\;\Big\Vert_{p}.
\end{align*}
Moreover, 
\begin{align*}
&\Big\Vert \widetilde{X}_{t_{l+1}}-\widetilde{X}_{t_{l}}\Big\Vert_{p}=\Big\Vert  h\,  b_l (t_l, \widetilde{X}_{t_{0} : t_{l}}, \widetilde{\mu}_{t_{0} : t_{l}} )+\sqrt{h}\, \sigma_l (t_l, \widetilde{X}_{t_{0} : t_{l}}, \widetilde{\mu}_{t_{0} : t_{l}} ) Z_{l+1}\Big\Vert_{p}\nonumber\\
&\quad \leq h \Big\Vert b \Big(t_l, i_{l}\big(\widetilde{X}_{t_{0} : t_{l}}), i_{l}\big(\widetilde{\mu}_{t_{0} : t_{l}} \big)\Big)\Big\Vert_{p} + \sqrt{h}\, \Big\Vert \;\vertiii{\sigma \Big(t_l, i_{l}\big(\widetilde{X}_{t_{0} : t_{l}}), i_{l}\big(\widetilde{\mu}_{t_{0} : t_{l}} \big)\Big)}\;\Big\Vert_{p} \, \Big\Vert Z_{l+1}\Big\Vert_{p}\nonumber\\
&\quad \leq  \Big( h + \sqrt{h} C_{p,q} \Big) \Big\Vert C_{b,\sigma, L, T}\Big(1+\big\Vert i_{l}\big(\widetilde{X}_{t_{0} : t_{l}}) \big\Vert_{\sup}+\sup_{t\in[0,T]}\mathcal{W}_{p}\big(i_{l}\big(\widetilde{\mu}_{t_{0} : t_{l}} \big)_t, \delta_0\big)\Big)\Big\Vert_{p}, \nonumber
\end{align*}
where we used Lemma \ref{lineargrowth}, and where $C_{p,q} = \Vert Z_{l+1}\Vert_{p}<+\infty$ is a constant depending only on $p$ and $q$, as $Z_{l+1} \sim \mathcal{N}(0, \mathbf{I}_q)$. Combined with Lemma \ref{interpolatorprop}, this yields
\begin{align}
\Big\Vert \widetilde{X}_{t_{l+1}}-\widetilde{X}_{t_{l}}\Big\Vert_{p} &\le \Big(h + \sqrt{h} C_{p,q} \Big)  \times \Big\Vert C_{b,\sigma, L, T}\Big(1 +\sup_{0\leq k\leq l}\big |\widetilde{X}_{t_{k}}\big |+\sup_{0\leq k\leq l}\mathcal{W}_{p}\big(\widetilde{\mu}_{t_{k}}, \delta_0\big)\Big)\Big\Vert_{p} \nonumber \\ 
& \leq  C_{b,\sigma, L, T} \Big( h + \sqrt{h} C_{p,q} \Big) \Big(1+ 2\Big\Vert\, \sup_{0\leq k\leq l}\big |\widetilde{X}_{t_{k}}\big |\,\Big\Vert_{p} \Big) <+\infty\nonumber
\end{align}
where we used the induction hypothesis to obtain the last inequality. The claim \eqref{eq:discretetilde} follows by induction. 
	
	\smallskip
\noindent	{\bf Step 2.} We prove that  $\big\Vert \sup_{t\in[0,T]} \big|\widetilde{X}_{t}\big| \big\Vert_p<+\infty$. First, from \eqref{eq:rmk_tilde},  we get for every $t\in[0,T]$, 
	\begin{align}\label{tildebound1}
		\Big\Vert \;\sup_{u\in [0,t]}|\widetilde{X}_u|\;\Big\Vert_p 
		& \leq \Vert X_0\Vert_p+\Big\Vert \int_{0}^{t}\,\Big| \,b\Big(\underline{s}, i_{[\underline{s}]}\big(\widetilde{X}_{t_0 : t_{[\underline{s}]}}\big), i_{[\underline{s}]}\big(\widetilde{\mu}_{t_0 : t_{[\underline{s}]}}\big)\,\Big) \Big| \,\d s\Big\Vert_p\nonumber\\
		& \quad\quad\quad\quad\quad +\Big\Vert\;\sup_{u\in [0,t]}\Big|\int_{0}^{u}\,\sigma\Big(\underline{s}, i_{[\underline{s}]}\big(\widetilde{X}_{t_0 : t_{[\underline{s}]}}\big), i_{[\underline{s}]}\big(\widetilde{\mu}_{t_0 : t_{[\underline{s}]}}\big)\,\Big) \d B_s\Big|\;\Big\Vert_p,
	\end{align}
	where we used Minkowski's inequality.
	The second term in \eqref{tildebound1} can be upper bounded as follows: 
\begin{align} \label{tildebound2}
&\Big\Vert \int_{0}^{t}\,\Big| \,b\Big(\underline{s}, i_{[\underline{s}]}\big(\widetilde{X}_{t_0 : t_{[\underline{s}]}}\big), i_{[\underline{s}]}\big(\widetilde{\mu}_{t_0 : t_{[\underline{s}]}}\big)\,\Big) \Big| \,\d s\Big\Vert_p \\
&\quad \leq\int_{0}^{t}\Big\Vert C_{b,\sigma, L, T}\Big(1+\big\Vert i_{[\underline{s}]}\big(\widetilde{X}_{t_0 : t_{[\underline{s}]}}\big)\big\Vert_{\sup}+\sup_{u\in[0,T]}\mathcal{W}_{p}\big(i_{[\underline{s}]}\big(\widetilde{\mu}_{t_0 : t_{[\underline{s}]}}\big)_u, \delta_0\big)\Big)\Big\Vert_{p}\d s\nonumber\\
&\quad \le \int_{0}^{t}C_{b,\sigma, L, T}\Big(1+2 \, \Big\Vert\,\sup_{0\leq k\leq [\underline{s}]}\big|\widetilde{X}_{t_{k}}\big|\,\Big\Vert_{p}\Big)\d s \nonumber \\
&\quad \leq \, T \, C_{b,\sigma, L, T}+2\,C_{b,\sigma, L, T}\int_{0}^{t}\Big\Vert\,\sup_{0\leq k\leq [\underline{s}]}\big|\widetilde{X}_{t_{k}}\big|\,\Big\Vert_{p}\d s 
<+\infty
\end{align}
where we used Lemmas \ref{gemin}, \ref{lineargrowth}  and \ref{interpolatorprop} and Step 1.
	
Moreover, combining Lemmas \ref{interpolatorprop}, \ref{lineargrowth} and \ref{lem:BDG_for_sigma},  the third term in \eqref{tildebound1} can be upper bounded as follows 
\begin{align}\label{tildebound3}
&\Big\Vert\;\sup_{u\in [0,t]}\Big|\int_{0}^{u}\,\sigma\Big(\underline{s}, i_{[\underline{s}]}\big(\widetilde{X}_{t_0 : t_{[\underline{s}]}}\big), i_{[\underline{s}]}\big(\widetilde{\mu}_{t_0 : t_{[\underline{s}]}}\big)\,\Big)\d B_s\Big|\;\Big\Vert_p\nonumber\\
&\qquad\leq C_{d, p}^{BDG}\Big\{\int_{0}^{t}\Big\Vert C_{b,\sigma, L, T}\Big(1+\big\Vert i_{[\underline{s}]}\big(\widetilde{X}_{t_0 : t_{[\underline{s}]}}\big)\big\Vert_{\sup}   +\sup_{u\in[0,T]}\mathcal{W}_{p}\big(i_{[\underline{s}]}\big(\widetilde{\mu}_{t_0 : t_{[\underline{s}]}}\big)_u, \delta_0\big)\Big)\Big\Vert_{p}^{2}\d s\Big\}^{\frac{1}{2}}\;\;\nonumber\\
&\qquad\leq \sqrt{2T} \, C_{d, p}^{BDG}\, C_{b, \sigma, L, T} + 2 C_{d, p}^{BDG} \, C_{b, \sigma, L, T}\Big\{\int_{0}^{t}\Big\Vert \,\sup_{0\leq k\leq [\underline{s}]}\big|\widetilde{X}_{t_{k}}\big|\,\Big\Vert_{p}^{2}\d s\Big\}^{\frac{1}{2}}
\end{align}
which is again finite by \eqref{eq:discretetilde}. We conclude that $\Big\Vert \sup_{t\in[0,T]} \big|\widetilde{X}_{t}\big| \Big\Vert_p<+\infty$. 
	
\noindent {\bf Step 3.} We conclude the proof of (a). Using that 
\[ \Big\| \sup_{0 \le k \le [\underline{s}]} \big| \widetilde X_{t_k} \big| \Big\|_p^2 \le \Big\| \sup_{u \in [0,s]} \big| \widetilde X_u \big| \Big\|_p^2 \]
by the definition of $[\underline{s}]$, see \eqref{eq:t_bar}, the inequalities \eqref{tildebound1}, \eqref{tildebound2} and \eqref{tildebound3} in the previous step imply that for every $t\in[0,T]$
\begin{align}
\Big\Vert \;\sup_{u\in [0,t]}|\widetilde{X}_u|\;\Big\Vert_p \nonumber & \leq \Vert X_0\Vert_p+T \, C_{b,\sigma, L, T}+2 \, C_{b,\sigma, L, T}\int_{0}^{t}\Big\Vert\,\sup_{u\in [0,s]}|\widetilde{X}_u|\,\Big\Vert_{p}\d s \nonumber\\
&\qquad +\sqrt{2T} \, C_{d, p}^{BDG} \, C_{b, \sigma, L, T} + 2 C_{d, p}^{BDG} \, C_{b, \sigma, L, T}\Big\{\int_{0}^{t}\Big\Vert \,\sup_{u\in [0,s]}|\widetilde{X}_u|\,\Big\Vert_{p}^{2}\d s\Big\}^{\frac{1}{2}}. \nonumber
\end{align}
Applying Lemma \ref{Gronwall} with $\displaystyle f(t)\coloneqq \Big\Vert \;\sup_{u\in [0,t]}|\widetilde{X}_u|\;\Big\Vert_p$, we deduce 
\begin{align}
&\Big\Vert \;\sup_{u\in [0,t]}|\widetilde{X}_u|\;\Big\Vert_p \leq C_{p,d,b,\sigma,L,T}  e^{C_{p,d,b,\sigma,L,T} t}(1+\left\Vert X_{0}\right\Vert_{p}),\nonumber
\end{align}
where the constant $C_{p,d,b,\sigma,L,T} > 0$ is given by 
\begin{align}
C_{p,d,b,\sigma,L,T} = \big( 4C_{b,\sigma, L, T}+ 8 (C_{d, p}^{BDG} \, C_{b, \sigma, L, T})^{2}\big)\vee 2\big(1\vee C_{b,\sigma, L, T}  T+\sqrt{2T} \, C_{d, p}^{BDG} \, C_{b, \sigma, L, T}\big). \nonumber
\end{align} 	
We conclude by applying the previous inequality with $t = T$, choosing $\Gamma = C_{p,d,b,\sigma,L,T} \, e^{C_{p,d,b,\sigma,L,T} T}$. 
	
\smallskip
\noindent $(b)$
By hypothesis, $M$ is such that $h = \frac{T}{M} \le \tfrac12$. We have
\begin{align*}
		&\left\Vert \sup_{0\leq m\leq M-1}\;\sup_{v\in[t_m, t_{m+1}]}\Big| \widetilde{X}_v-\widetilde X_{t_{m}}\Big|\;\right\Vert_p \nonumber\\
		& \leq \Big\Vert \sup_{0\leq m\leq M-1} \sup_{v\in[t_m, t_{m+1}]} \Big[ \Big| (v-t_m) \, b_m (t_m, \widetilde{X}_{t_{0} : t_{m}}, \widetilde{\mu}_{t_{0} : t_{m}} ) \Big| \!+\! \Big| \sigma_m (t_m, \widetilde{X}_{t_{0} : t_{m}}, \widetilde{\mu}_{t_{0} : t_{m}} )  (B_v-B_{t_m})\Big| \;\Big]\Big\Vert_p\nonumber\\
		& \leq h \left\Vert \sup_{0\leq m\leq M-1} \,\Big|b_m (t_m, \widetilde{X}_{t_{0} : t_{m}}, \widetilde{\mu}_{t_{0} : t_{m}} )\Big|\,\right\Vert_p\nonumber\\
		& \quad\quad\quad+\left\Vert \sup_{0\leq m\leq M-1}\,\left[\vertiii{\sigma_m (t_m, \widetilde{X}_{t_{0} : t_{m}}, \widetilde{\mu}_{t_{0} : t_{m}} )}\, \sup_{v\in[t_m, t_{m+1}]}\Big|B_v-B_{t_m}\Big|\right]\right\Vert_p
	\end{align*}
	where we used that $|t_{m+1} - t_m| = h$ and Minkowski's inequality.   We now set $p'=p+\frac{\varepsilon_0}{2}$ and $p_0 = \frac{p(2p+\varepsilon_0)}{\varepsilon_0}$ such that $\frac{1}{p}=\frac{1}{p_0}+\frac{1}{p'}$. 
    Using $(\sum_{i=1}^q a_i)^{p_0/2} \le q^{\frac{p_0}2-1} \sum_{i=1}^q a_i^{p_0/2}$ for all $a_1, \ldots, a_q$ real positive numbers and decomposing the supremum on Brownian motions as follows
    \begin{align*}
        \sup_{\substack{s, t \in [0,T] \\ s \le t}} |B_s - B_t|^{p_0} = \sup_{\substack{s, t \in [0,T] \\ s \le t}} \Big(\sum_{i=1}^q |B^i_s - B^i_t|^2 \Big)^{\frac{p_0}2} \le q^{\frac{p_0}2 - 1} \sum_{i=1}^q \sup_{\substack{s, t \in [0,T] \\ s \le t}} |B^i_s - B^i_t|^{p_0},
    \end{align*}
     we get from H\"older's inequality,
    \begin{align*}
        &\left\Vert \sup_{0\leq m\leq M-1}\,\left[\vertiii{\sigma_m (t_m, \widetilde{X}_{t_{0} : t_{m}}, \widetilde{\mu}_{t_{0} : t_{m}} )}\, \sup_{v\in[t_m, t_{m+1}]}\Big|B_v-B_{t_m}\Big|\right]\right\Vert_p \\
        &\le \Big\| \sup_{\substack{s, t \in [0,T] \\ |s - t| \le h}} \big|B_s - B_t\big| \Big\|_{p_0} \, \,  \Big\| \sup_{0 \le m \le M-1} \vertiii{\sigma_m (t_m, \widetilde{X}_{t_{0} : t_{m}}, \widetilde{\mu}_{t_{0} : t_{m}} )} \Big\|_{p'} \\
        & \le \sqrt{q} \, \EE \Big[ \sup_{\substack{s, t \in [0,T] \\ |s - t| \le h}} \big|Z_s - Z_t\big|^{p_0} \Big]^{\frac{1}{p_0}} \, \, \Big\| \sup_{0 \le m \le M-1} \vertiii{\sigma_m (t_m, \widetilde{X}_{t_{0} : t_{m}}, \widetilde{\mu}_{t_{0} : t_{m}} )} \Big\|_{p'}, 
    \end{align*}
    where $(Z_t)_{t \in [0,T]}$ is a standard one-dimensional Brownian motion. We bound this expectation on $Z$ using moment estimates on the modulus of continuity of uni-dimensional Brownian motion, see \cite[Lemma 3]{Fischer_2009} and find 
     \begin{align*}
        &\left\Vert \sup_{0\leq m\leq M-1}\,\left[\vertiii{\sigma_m (t_m, \widetilde{X}_{t_{0} : t_{m}}, \widetilde{\mu}_{t_{0} : t_{m}} )}\, \sup_{v\in[t_m, t_{m+1}]}\Big|B_v-B_{t_m}\Big|\right]\right\Vert_p \\
        & \le C_{p_0,q} (h \log(\tfrac{2T}h))^{\frac12} \, \Big\| \sup_{0 \le m \le M-1} \vertiii{\sigma_m (t_m, \widetilde{X}_{t_{0} : t_{m}}, \widetilde{\mu}_{t_{0} : t_{m}} )} \Big\|_{p'},
        \end{align*} 
        with $C_{p_0,q} = \frac{6\sqrt{q}}{\sqrt{\log(2)}} \, \Big(\frac{5}{\sqrt{\pi}} \Gamma \big(\tfrac{p_0 + 1}{2}\big) \Big)^{\frac1{p_0}}$.
Hence  
	\begin{align*}
		\Big\Vert \sup_{0\leq m\leq M-1}\;&\sup_{v\in[t_m, t_{m+1}]}\Big| \widetilde{X}_v-\widetilde X_{t_{m}}\Big|\;\Big\Vert_p \leq  h \left\Vert \sup_{0\leq m\leq M-1} \,\Big|b_m (t_m, \widetilde{X}_{t_{0} : t_{m}}, \widetilde{\mu}_{t_{0} : t_{m}} )\Big|\,\right\Vert_p \! \\
		&\quad + C_{p_0, q, T} \big(h \big|\log(h)\big| \big)^{\frac{1}{2}} \left\Vert \sup_{0\leq m\leq M-1}\,\vertiii{\;\sigma_m (t_m, \widetilde{X}_{t_{0} : t_{m}}, \widetilde{\mu}_{t_{0} : t_{m}} )} \;\right\Vert_{p'}\!.
		\end{align*}  
	We now treat the two terms on the right-hand-side of this inequality. First, by definition of $b_m$, we get
	\begin{align}\label{eq:coef-bm-lp-bound}
		&\left\Vert \sup_{0\leq m\leq M-1} \,\Big|b_m (t_m, \widetilde{X}_{t_{0} : t_{m}}, \widetilde{\mu}_{t_{0} : t_{m}} )\Big|\,\right\Vert_p \nonumber \\
		&\quad =\left\Vert \sup_{0\leq m\leq M-1} \,\Big|b \Big(t_m, i_m\big(\widetilde{X}_{t_{0} : t_{m}}\big), i_m\big(\widetilde{\mu}_{t_{0} : t_{m}}\big)\Big )\Big|\,\right\Vert_p\nonumber\\
		&\quad  \leq C_{b,\sigma,L,T}\Big(1+2\,\Big\Vert\sup_{0\leq k \leq M}|\widetilde{X}_{k}|\Big\Vert_p\Big)\leq C_{b,\sigma,L,T}\Big(1+2\, \Gamma(1+\Vert X_0\Vert_p)\Big)<+\infty,
	\end{align}
where we used Lemma \ref{lineargrowth} to obtain the first inequality, and Lemma \ref{interpolatorprop} to get the second one.
	Let $C_{\star}(p)\coloneqq C_{b,\sigma,L,T}\Big(1+2\,\Gamma(1+\Vert X_0\Vert_p)\Big)$, where we recall that $\Gamma$ is given by item $(a)$. By a similar computation, using that Assumption \ref{assump:norm_X0}  is also satisfied with $p+\frac{\varepsilon_0}2$, \dd we obtain 
\begin{equation}\label{eq:coef-sigmam-lp-bound}
 \left\Vert \sup_{0\leq m\leq M-1}\,\vertiii{\;\sigma_m (t_m, \widetilde{X}_{t_{0} : t_{m}}, \widetilde{\mu}_{t_{0} : t_{m}} )} \;\right\Vert_{p'} \leq C_{\star}\left(p+\frac{\varepsilon_0}2\right).
 \end{equation}
Then, using that for $h \in [0,\tfrac12]$, $h \le (h |\log(h)|)^{\tfrac12}$,  
\begin{align}
&\left\Vert \sup_{0\leq m\leq M-1}\;\sup_{v\in[t_m, t_{m+1}]}\Big| \widetilde{X}_v-\widetilde X_{t_{m}}\Big|\;\right\Vert_p \leq  \left(C_\star(p) + C_{p_0, q, T} C_\star\Big(p+\frac{\varepsilon_0}2\Big)\right) \big(h \big|\log(h)\big| \big)^{\frac{1}{2}} \nonumber
\end{align}
and we can conclude by letting $\kappa\coloneqq \big(C_\star(p) + C_{p_0, q, T} C_\star(p+ \frac{\varepsilon_0}2)\big)$.

\smallskip
\noindent $(c)$ By Lemma \ref{lem:lemma7.4-pages2018numerical}, to establish \eqref{eq:X-tilde-lp-holder}, it suffices to prove the inequality:
\begin{equation}
\sup_{s\in[0,T]}\left\Vert b \Big(\underline s, i_{[\underline s]}\big(\widetilde X^{h}_{t_0:t_{[\underline s]}} \big), i_{[\underline s]}\big(\widetilde \mu^{h}_{t_0:t_{[\underline s]}} \big) \Big)\right\Vert_p\vee \left \Vert \sigma \Big(\underline s, i_{[\underline s]}\big(\widetilde X^{h}_{t_0:t_{[\underline s]}} \big), i_{[\underline s]}\big(\widetilde \mu^{h}_{t_0:t_{[\underline s]}} \big) \Big) \right\Vert_p<+\infty, \nonumber
\end{equation}
which follows directly from \eqref{eq:coef-bm-lp-bound} and an analogous argument for $\sigma$. This yields the result with $\tilde{\kappa}=C_{d,p}C_{b,\sigma,L,T}\big(1+\sqrt{T}\big)\Big(1+2\,\Gamma(1+\Vert X_0\Vert_p)\Big)$.
\end{proof}

\subsubsection{Proof of Proposition \ref{prop:cvg_cont_Euler} and Corollary \ref{cor2}} \label{subsec:3proofs}



\begin{proof}[Proof of Proposition \ref{prop:cvg_cont_Euler}] We drop the superscript $h$ in $\widetilde{X}^h$ and in $\widetilde{\mu}^h$ for simplicity.
For every $s\in[0, T]$,  we have 
	\begin{align}
		&X_{s}-\widetilde{X}_{s}=\int_{0}^{s}\left[ b(u, X_{\cdot \wedge u},\mu_{\cdot \wedge u})-b \Big(\underline{u}, i_{[\underline{u}]}\big(\widetilde{X}_{t_0 : t_{[\underline{u}]}}\big), i_{[\underline{u}]}\big(\widetilde{\mu}_{t_0 : t_{[\underline{u}]}}\big)\,\Big)\right]\d u
		\nonumber\\
		&\qquad \qquad\qquad+\int_{0}^{s}\left[  \sigma(u, X_{\cdot \wedge u}, \mu_{\cdot \wedge u})-\sigma \Big(\underline{u}, i_{[\underline{u}]}\big(\widetilde{X}_{t_0 : t_{[\underline{u}]}}\big), i_{[\underline{u}]}\big(\widetilde{\mu}_{t_0 : t_{[\underline{u}]}}\big)\,\Big)\right] \d B_{u},\nonumber
	\end{align}
	and we set  \[f(t)\coloneqq\left\Vert\sup_{s\in[0,t]}\left|X_{s}-\widetilde{X}_{s}\right|\right\Vert_{p}.\]
	
	It follows from Proposition \ref{propXtilde}-$(a)$ that $\widetilde{X}=(\widetilde{X}_{t})_{t\in[0, T]}\in L_{\mathcal{C}([0, T], \mathbb{R}^{d})}^{p}(\Omega, \mathcal{F}, \mathbb{P})$. Consequently, $\widetilde{\mu}\in\PPC$ and $\iota(\widetilde{\mu})=(\widetilde{\mu}_{t})_{t\in[0, T]}\in\CPP$ by Lemma \ref{injectionmeasure}.
	Hence, 
	\begin{align}\label{cvgtilde1}
		f(t)&=\left\Vert\sup_{s\in[0,t]}\left|X_{s}-\widetilde{X}_{s}\right|\right\Vert_{p}\nonumber\\
		&\leq \int_{0}^{t}\vertii{b(s, X_{\cdot \wedge s}, \mu_{\cdot \wedge s})-b \Big(\underline{s}, i_{[\underline{s}]}\big(\widetilde{X}_{t_0 : t_{[\underline{s}]}}\big), i_{[\underline{s}]}\big(\widetilde{\mu}_{t_0 : t_{[\underline{s}]}}\big)\,\Big)}_{p}\d s\nonumber\\
		&\quad+C_{d,p}^{BDG}\Big{[}\int_{0}^{t}\Big{\Vert}\vertiii{\sigma(s, X_{\cdot \wedge s}, \mu_{\cdot \wedge s})-\sigma \Big(\underline{u}, i_{[\underline{u}]}\big(\widetilde{X}_{t_0 : t_{[\underline{u}]}}\big), i_{[\underline{u}]}\big(\widetilde{\mu}_{t_0 : t_{[\underline{u}]}}\big)\,\Big)}\Big{\Vert}_{p}^{2}\d s\Big{]}^{\frac{1}{2}}
	\end{align}
	using Lemma \ref{lem:BDG_for_sigma}. 
	The first term in \eqref{cvgtilde1} can be controlled by
	\begin{align}\label{cvgtilde11}
		&\int_{0}^{t}\vertii{b(s, X_{\cdot \wedge s}, \mu_{\cdot \wedge s})-b \Big(\underline{s}, i_{[\underline{s}]}\big(\widetilde{X}_{t_0 : t_{[\underline{s}]}}\big), i_{[\underline{s}]}\big(\widetilde{\mu}_{t_0 : t_{[\underline{s}]}}\big)\,\Big)}_{p}\d s\nonumber\\
		&\quad\leq \int_{0}^{t}\vertii{b(s, X_{\cdot \wedge s}, \mu_{\cdot \wedge s})-b(\underline{s}, X_{\cdot \wedge {s}}, \mu_{\cdot \wedge {s}})}_{p}\d s\nonumber\\
		& \qquad +\int_{0}^{t}\vertii{b(\underline{s}, X_{\cdot \wedge {s}}, \mu_{\cdot \wedge {s}})-b \Big(\underline{s}, i_{[\underline{s}]}\big(\widetilde{X}_{t_0 : t_{[\underline{s}]}}\big), i_{[\underline{s}]}\big(\widetilde{\mu}_{t_0 : t_{[\underline{s}]}}\big)\,\Big)}_{p}\d s. 
	\end{align}
	For the first term in \eqref{cvgtilde11}, we use Assumptions \ref{assump:lipschitz} and \ref{assump:holder} to obtain
	\begin{align}
		\label{eq:control_b_thm}
		&\int_{0}^{t}\vertii{b(s, X_{\cdot \wedge s}, \mu_{\cdot \wedge s})-b(\underline{s}, X_{\cdot \wedge {s}}, \mu_{\cdot \wedge {s}})}_{p}\d s \nonumber  \\
		&\le \int_0^t L \Big\| 1 + \|X_{\cdot \wedge s}\|_{\sup} + \sup_{u \in [0,T]} \mathcal{W}_p(\mu_{u \wedge s}, \delta_0) \Big\|_p |s - \underline{s}|^{\gamma} \d s \nonumber  \\ 
		&\le \Big(LT + 2LT \big \Vert\sup_{t \in [0,T]} |X_t| \big \Vert_{p} \Big) h^{\gamma}\le 2 \, h^{\gamma}   \, L \, T \, \Gamma \, (1 + \|X_0\|_p),  
	\end{align}
	where we used \eqref{eq:thm_well_posed} to obtain the last inequality. 
	For the second term of \eqref{cvgtilde11}, we have
	\begin{align}\label{tildecvg2}
		&\int_{0}^{t}\vertii{b(\underline{s}, X_{\cdot \wedge s}, \mu_{\cdot \wedge s})-b \Big(\underline{s}, i_{[\underline{s}]}\big(\widetilde{X}_{t_0 : t_{[\underline{s}]}}\big), i_{[\underline{s}]}\big(\widetilde{\mu}_{t_0 : t_{[\underline{s}]}}\big)\,\Big)}_{p}\d s\nonumber\\
		&\quad \leq \int_{0}^{t}\Big\Vert \; L \,\Big[ \big\Vert X_{\cdot \wedge s}-i_{[\underline{s}]}\big(\widetilde{X}_{t_0 : t_{[\underline{s}]}}\big)\big\Vert_{\sup} +  \sup_{v\in[0,T]}\mathcal{W}_p\Big( \mu_{v\wedge s}, i_{[\underline{s}]}\big(\widetilde{\mu}_{t_0 : t_{[\underline{s}]}}\big)_v\Big)\Big]\Big\Vert_{p}\d s\nonumber\\
		&\quad \leq L \int_{0}^{t}\Big\Vert \;\big\Vert X_{\cdot \wedge s}- \widetilde{X}_{\cdot \wedge s}\big\Vert_{\sup}\;\Big\Vert_p\;\d s+L \int_{0}^{t}\Big\Vert \;\big\Vert \widetilde{X}_{\cdot \wedge s}-i_{[\underline{s}]}\big(\widetilde{X}_{t_0 : t_{[\underline{s}]}}\big)\big\Vert_{\sup}\;\Big\Vert_p\;\d s\nonumber\\
		&\qquad  +L\int_{0}^{t} \sup_{v\in[0,T]}\mathcal{W}_p( \mu_{v\wedge s}, \widetilde{\mu}_{v\wedge s}) \d s  +L\int_{0}^{t}\sup_{v\in[0,T]}\mathcal{W}_p\Big( \widetilde{\mu}_{v\wedge s}\,, i_{[\underline{s}]}\big(\widetilde{\mu}_{t_0 : t_{[\underline{s}]}}\big)_v\Big)\d s\nonumber\\
		&\quad \leq L\int_{0}^{t}f(s) \, \d s+LT \, 5\kappa \, \big( h \big|\log(h)\big| \big)^{\tfrac12}+L \int_{0}^{t}\sup_{v\in[0,s]}\big\Vert X_v-\widetilde{X}_v\big\Vert_p \d s\nonumber\\
		&\quad \leq 2L\int_{0}^{t}f(s) \, \d s+5LT \kappa \big( h \big|\log(h)\big| \big)^{\tfrac12},
	\end{align}
	where we used Corollary \ref{propXtilde-bis} to obtain the third inequality. 
	Now we consider the second term of \eqref{cvgtilde1}. It follows by applying Lemma \ref{lem:BDG_for_sigma} and norm inequalities that 
	\begin{align}\label{aaaa}
		& C_{d,p}^{BDG}\Big{[}\int_{0}^{t}\Big{\Vert}\vertiii{\sigma(s, X_{\cdot \wedge s}, \mu_{\cdot \wedge s})-\sigma \Big(\underline{s}, i_{[\underline{s}]}\big(\widetilde{X}_{t_0 : t_{[\underline{u}]}}\big), i_{[\underline{s}]}\big(\widetilde{\mu}_{t_0 : t_{[\underline{u}]}}\big)\,\Big)}\Big{\Vert}_{p}^{2}\d s\Big{]}^{\frac{1}{2}}\nonumber\\
		& \leq \sqrt{2} C_{d,p}^{BDG}\Big{[}\int_{0}^{t}\Big{\Vert}\vertiii{\sigma(\underline{s}, X_{\cdot \wedge s}, \mu_{\cdot \wedge s})- \sigma \Big(\underline{s}, i_{[\underline{s}]}\big(\widetilde{X}_{t_0 : t_{[\underline{s}]}}\big), i_{[\underline{s}]}\big(\widetilde{\mu}_{t_0 : t_{[\underline{s}]}}\big)\,\Big) }\Big{\Vert}_{p}^{2}\d s\Big{]}^{\frac{1}{2}} \nonumber \\
		&\quad + \sqrt{2} C_{d,p}^{BDG}\Big{[}\int_{0}^{t}\Big{\Vert}\vertiii{\sigma(s, X_{\cdot \wedge s}, \mu_{\cdot \wedge s})-\sigma(\underline{s}, X_{\cdot \wedge s}, \mu_{\cdot \wedge s})}\Big{\Vert}_{p}^{2}\d s\Big{]}^{\frac{1}{2}}.
	\end{align}
	For the first term in \eqref{aaaa}, we use the same argument as the one giving \eqref{eq:control_b_thm} to get 
	\begin{align}
		\label{eq:control_sigma}
		\Big{[}\int_{0}^{t}\Big{\Vert}\vertiii{\sigma(s, X_{\cdot \wedge s}, \mu_{\cdot \wedge s})-\sigma(\underline{s}, X_{\cdot \wedge s}, \mu_{\cdot \wedge s})}\Big{\Vert}_{p}^{2}\d s\Big{]}^{\frac{1}{2}} 
		\le h^{\gamma} \Big( \sqrt{2T} + 2 \sqrt{T} \Gamma_2 (1 + \|X_0\|_p) \Big)
	\end{align}
	for some constant $\Gamma_2 > 0$ depending explicitly on $\Gamma$ from \eqref{eq:thm_well_posed} and the constants of Assumptions \ref{assump:norm_X0}, \ref{assump:lipschitz} and \ref{assump:holder}.  
	The second term of \eqref{aaaa} can be upper bounded as follows
	\begin{align}
		&\sqrt{2} C_{d,p}^{BDG}\Big{[}\int_{0}^{t}\Big{\Vert}\vertiii{\sigma(\underline{s}, X_{\cdot \wedge s}, \mu_{\cdot \wedge s})- \sigma \Big(\underline{s}, i_{[\underline{s}]}\big(\widetilde{X}_{t_0 : t_{[\underline{s}]}}\big), i_{[\underline{s}]}\big(\widetilde{\mu}_{t_0 : t_{[\underline{s}]}}\big)\,\Big) }\Big{\Vert}_{p}^{2}\d s\Big{]}^{\frac{1}{2}}\nonumber\\
		&\quad \leq 2 L C_{d,p}^{BDG}\Big{[}\int_{0}^{t}\Big{\Vert}    \big\Vert X_{\cdot \wedge s}-i_{[\underline{s}]}\big(\widetilde{X}_{t_0 : t_{[\underline{s}]}}\big)\big\Vert_{\sup} \Big{\Vert}_{p}^{2}\d s\Big{]}^{\frac{1}{2}} \nonumber\\
		&\qquad\qquad\qquad  + 2 L C_{d,p}^{BDG}\Big{[}\int_{0}^{t}   \sup_{v\in[0,T]}\mathcal{W}_p\Big( \mu_{v\wedge s}, i_{[\underline{s}]}\big(\widetilde{\mu}_{t_0 : t_{[\underline{s}]}}\big)_v\Big)^{2}\d s  \Big{]}^{\frac{1}{2}} \nonumber\\
		&\quad \leq 2\sqrt{2} L C_{d,p}^{BDG} \Big{[}\int_{0}^{t}\Big\Vert \;\big\Vert X_{\cdot \wedge s}- \widetilde{X}_{\cdot \wedge s}\big\Vert_{\sup}\;\Big\Vert_p^2\;\d s\Big{]}^{\frac{1}{2}}\nonumber\\
		&\qquad +2\sqrt{2} L C_{d,p}^{BDG} \Big{[}\int_{0}^{t}\Big\Vert \;\big\Vert \widetilde{X}_{\cdot \wedge s}-i_{[\underline{s}]}\big(\widetilde{X}_{t_0 : t_{[\underline{s}]}}\big)\big\Vert_{\sup}\;\Big\Vert_p^2\;\d s\Big{]}^{\frac{1}{2}}\nonumber\\
		&\qquad  +2\sqrt{2} L C_{d,p}^{BDG}\Big[\int_{0}^{t} \sup_{v\in[0,T]}\mathcal{W}_p\big( \mu_{v\wedge s}, \widetilde{\mu}_{v\wedge s}\big)^2  \d s\Big{]}^{\frac{1}{2}}\nonumber\\
		&\qquad +2\sqrt{2} L C_{d,p}^{BDG}\Big[\int_{0}^{t} \sup_{v\in[0,T]}\mathcal{W}_p\Big( \widetilde{\mu}_{v\wedge s}, i_{[\underline{s}]}\big(\widetilde{\mu}_{t_0 : t_{[\underline{s}]}}\big)_v\Big)^{2} \d s\Big{]}^{\frac{1}{2}}\nonumber\\
		&\quad \leq 4\sqrt{2} L C_{d,p}^{BDG}\Big[\int_{0}^{t}f(s)^{2}\d s\Big]^{\frac{1}{2}}+ 2\sqrt{2} L C_{d,p}^{BDG}\sqrt{T}\, 5\kappa \big( h \big|\log(h)\big| \big)^{\tfrac12}
	\end{align}
	by a similar reasoning as the one leading to \eqref{tildecvg2}. Bringing those inequalities together, we find
	\begin{align}
		f(t)&\leq   L \, 2 h^{\gamma}\,T \Gamma \, \big(1 + \|X_0\|_p\big)  + 2L\int_{0}^{t}f(s)\d s+5LT \kappa \big( h \big|\log(h)\big| \big)^{\tfrac12} \nonumber \\
		& \qquad + h^{\gamma} \Big( \sqrt{2T} + 2 \sqrt{T} \, \Gamma_2 (1 + \|X_0\|_p) \Big)
		+4\sqrt{2} \, L \, C_{d,p}^{BDG}\Big[\int_{0}^{t}f(s)^{2}\d s\Big]^{\frac{1}{2}} \nonumber \\
		&\qquad + 10\sqrt{2} \,  L C_{d,p}^{BDG}\, \sqrt{T} \, \kappa \big( h \big|\log(h)\big| \big)^{\tfrac12}.
	\end{align}
	The conclusion follows by applying Lemma \ref{Gronwall}. 
\end{proof}

\begin{proof}[Proof of Corollary \ref{cor2}]
Corollary \ref{propXtilde-bis} implies that \[\Big\Vert \, \big\Vert \widehat{X}-\widetilde{X} \big\Vert_{\sup} \,\Big\Vert_p \leq 2\kappa \big(h\big|\log(h)\big| \big)^{\frac{1}{2}}. \]  Then the result is a direct application of Proposition \ref{prop:cvg_cont_Euler}.
\end{proof}


\subsection{Convergence of the particle method}\label{subsec:cvgparticlemethod}

This section is devoted to the proof of Theorem \ref{thm:particlemethod}.
Using the notations from \eqref{eq:t_bar}, we obtain the following equation equivalent to \eqref{eq:contiparticlesystem} and readily comparable to \eqref{eq:rmk_tilde} : for every $n=\llN^*$, 
\begin{align}\label{eq:contiparticlesystem2}
&\widetilde{X}_{t}^{n, N,h}=\widetilde{X}_{0}^{n, N,h}+\int_{0}^t b\Big(\underline{s}, i_{[\underline{s}]}\big(\widetilde{X}_{t_0}^{n, N,h}, ..., \widetilde{X}_{t_{[\underline{s}]}}^{n, N,h}\big), i_{[\underline{s}]}\Big(\widetilde \mu^{N,h}_{t_0}, ...,\widetilde \mu^{N,h}_{t_{[\underline{s}]}}\Big)\Big)\d s\nonumber\\
&\qquad \qquad \quad+\int_0^t\sigma\Big(\underline{s}, i_{[\underline{s}]}\big(\widetilde{X}_{t_0}^{n, N,h}, ..., \widetilde{X}_{t_{[\underline{s}]}}^{n, N,h}\big), i_{[\underline{s}]}\Big(\widetilde \mu^{N,h}_{t_0}, ...,\widetilde \mu^{N,h}_{t_{[\underline{s}]}}\Big)\Big)\d B_s^n.
\end{align}

We also introduce an intermediate system, made of i.i.d. particles. 

\begin{defn}
	\label{def:intermediate_system}
	Given $N \ge 1$ $M \ge 2T+1$, $(\widetilde X_{t_0:t_M}^h)$ with the associated probability distributions $(\widetilde \mu_{t_0:t_M}^h)$ from Definition \ref{def:discretization_scheme} and $(\widetilde X^{1,N,h}_{t_0:t_M}, \dots, \widetilde X^{N,N,h}_{t_0:t_M})$ from Definition \ref{def:particle_method}, we define the continuous intermediate particle system without interaction $(Y^{1,h}_t, \dots, Y^{N,h}_t)_{t \in [0,T]}$ as follows: 
\[(Y_0^{1,h}, \dots, Y_0^{N,h})\coloneqq(\widetilde{X}_{0}^{1, N}, \dots, \widetilde{X}_{0}^{N, N}),\]
and for every $m=0,..., M-1$ and  for every $t\in(t_m, t_{m+1}],$ $n \in \llbracket N \rrbracket^*$,
\[Y^{n,h}_t =  Y^{n,h}_{t_m} + (t-t_m) \,b\,\Big(t_m, i_m\big( Y^{n,h}_{t_0:t_m} \big), i_m \big( \widetilde \mu^{h}_{t_0:t_m} \big) \Big) + \sigma\Big(t_m, i_m\big(  Y^{n,h}_{t_0:t_m} \big), i_m \big(\widetilde \mu^{h}_{t_0:t_m} \big) \Big) \big(B^n_t - B^n_{t_m} \big).\]
	
\end{defn}



\begin{rem}\label{rem:law_Y}
It is clear from this definition that $(Y^{1,h},\dots,Y^{N,h})$ are i.i.d. copies of $\widetilde X^h$, and thus, for all $t \in [0,T]$ and $i\in\{1,..., N\}$, $\mathcal{L}(Y^{i,h}_t) = \widetilde \mu^h_t$.  
\end{rem}

For $k \in \{0, \dots, M\}$, recall the notation $\widetilde \mu^{N,h}_{t_k}$ and $\widetilde \mu^h_{t_k}$ from, respectively, \eqref{def:empirical_measure} and Definition \ref{def:discretization_scheme}. To lighten the presentation, we set the following 
\begin{align}
&\overline{b}_{(s, \widetilde{X})}^{\,n}\coloneqq b\Big(\underline{s}, i_{[\underline{s}]}\big(\widetilde{X}_{t_0}^{n, N,h}, ..., \widetilde{X}_{t_{[\underline{s}]}}^{n, N,h}\big), i_{[\underline{s}]}\Big(\widetilde \mu^{N,h}_{t_0}, \dots, \widetilde \mu^{N,h}_{t_{[\underline s]}} \Big)  \Big) \nonumber\\
&\overline{b}_{(s, Y)}^{\,n}\coloneqq b\,\Big(\underline{s}, i_{[\underline{s}]}\big( Y^{n,h}_{t_0}, ...,  Y^{n,h}_{t_{[\underline{s}]}}\big), i_{[\underline{s}]} \big( \widetilde \mu^{h}_{t_0}, \dots, \widetilde \mu^{h}_{t_{[\underline{s}]}}\big) \Big)\nonumber\\
&\overline{\sigma}_{(s, \widetilde{X})}^{\,n}\coloneqq \sigma\Big(\underline{s}, i_{[\underline{s}]}\big(\widetilde{X}_{t_0}^{n, N,h}, ..., \widetilde{X}_{t_{[\underline{s}]}}^{n, N,h}\big), i_{[\underline{s}]}\Big(\widetilde \mu^{N,h}_{t_0}, \dots, \widetilde \mu^{N,h}_{t_{[\underline s]}} \Big)  \Big) \nonumber \\
&\overline{\sigma}_{(s, Y)}^{\,n}\coloneqq \sigma\,\Big(\underline{s}, i_{[\underline{s}]}\big( Y^{n,h}_{t_0}, ...,  Y^{n,h}_{t_{[\underline{s}]}}\big), i_{[\underline{s}]} \big( \widetilde \mu^{h}_{t_0}, \dots, \widetilde \mu^{h}_{t_{[\underline{s}]}}\big) \Big).\nonumber
\end{align}

\subsubsection{Proof of Theorem \ref{thm:particlemethod}-$(a)$}\label{subsec:proof-main-thm-(a)}
\begin{proof}[Proof of Theorem \ref{thm:particlemethod}-$(a)$] 
We present the proof of~\eqref{mainresult-review}. 
For every $n=1,..., N$ and for every $t\in[0,T]$, we have 
\[\big|\widetilde{X}_{t}^{n, N, h}-Y_t^{n,h}\big|=\Big|\,\int_0^t \big[\overline{b}_{(s, \widetilde{X})}^{\,n}-\overline{b}_{(s, Y)}^{\,n}\big]\,\d s+\int_0^t \big[\overline{\sigma}_{(s, \widetilde{X})}^{\,n}-\overline{\sigma}_{(s, Y)}^{\,n}\big]\,\d B^n_s\Big|. \]
Then, using Lemmas \ref{lem:BDG_for_sigma} and \ref{gemin},
\begin{align}\label{ineq:first}
&\Big\Vert \sup_{s\in[0,t]} \big|\widetilde{X}_{s}^{n, N,h}-Y_s^{n,h}\big|\Big\Vert_p\nonumber\\
&\quad \leq \int_0^t \big\Vert\overline{b}_{(s, \widetilde{X})}^{\,n}-\overline{b}_{(s, Y)}^{\,n}\big\Vert_p\,\d s+C_{d,p}^{BDG} \Big[ \int_0^t \Big\| \vertiii{\overline{\sigma}_{(s, \widetilde{X})}^{\,n}-\overline{\sigma}_{(s, Y)}^{\,n}} \Big\|_p^2 \d s \Big]^{\tfrac12}.
\end{align}
For the first term in \eqref{ineq:first},  Assumptions \ref{assump:norm_X0} and \ref{assump:lipschitz} imply that 
\begin{align}\label{ineq1}
\big\Vert\overline{b}_{(s, \widetilde{X})}^{\,n}-\overline{b}_{(s, Y)}^{\,n}\big\Vert_p&\leq L\Big\Vert  \big\Vert i_{[\underline{s}]}\big(\widetilde{X}_{t_0}^{n, N,h}, ..., \widetilde{X}_{t_{[\underline{s}]}}^{n, N,h}\big)- i_{[\underline{s}]}\big( Y^{n,h}_{t_0}, ..., Y^{n,h}_{t_{[\underline{s}]}}\big)\big\Vert_{\sup}\Big\Vert_p\nonumber\\
&\quad +L \, \Big\| \sup_{t\in[0,T]}\mathcal{W}_p\Big( i_{[\underline{s}]}\big(\widetilde \mu^{N,h}_{t_0},\dots, \widetilde \mu^{N,h}_{t_{[\underline s]}}\big)_t\,,  i_{[\underline{s}]} \big( \widetilde \mu^{h}_{t_0}, ..., \widetilde \mu^{h}_{t_{[\underline{s}]}}\big)_t \Big)  \Big\|_p \nonumber\\
&\leq L \, \Big\Vert \sup_{u\in[0,s]}\big| \widetilde{X}_{u}^{n, N,h}- Y^{n,h}_{u}\big|\Big\Vert_p + L \,  \Big\Vert\max_{0\leq \ell\leq  [\underline{s}]} \mathcal{W}_p\big(\widetilde \mu^{N,h}_{t_\ell}, \widetilde \mu^{h}_{t_\ell}\big) \Big\Vert_p
\end{align}
where the second inequality above follows from Lemma \ref{lem:inter}. Similarly, for the second term in \eqref{ineq:first},  we have 
\begin{align}\label{ineq2}
\Big\Vert \;\vertiii{\overline{\sigma}_{(s, \widetilde{X})}^{\,n}-\overline{\sigma}_{(s, Y)}^{\,n}} \;\Big\Vert_p^2  
&\leq 2L^2 \, \Big\Vert \sup_{u\in[0,s]}\big| \widetilde{X}_{u}^{n, N,h}- Y^{n,h}_{u}\big|\Big\Vert_p^{ 2 } +  2L^2  \,  \Big\Vert\max_{0\leq \ell\leq  [\underline{s}]} \mathcal{W}_p\big(\widetilde \mu^{N,h}_{t_\ell}, \widetilde \mu^{h}_{t_\ell}\big) \Big\Vert_p^2.
\end{align}
Inserting \eqref{ineq1} and \eqref{ineq2} into \eqref{ineq:first}, one gets
\begin{align}\label{ineq:second}
&\Big\Vert \sup_{s\in[0,t]} \big|\widetilde{X}_{s}^{n, N,h}-Y_s^{n,h}\big|\Big\Vert_p\nonumber\\
&\leq L \int_0^t \Big\Vert  \sup_{u\in[0,s]} \big|\widetilde{X}_{u}^{n, N,h}-Y_u^{n,h}\big|\Big\Vert_p\d s +  C_{d,p, L} \Big[\int_0^t \Big\Vert  \sup_{u\in[0,s]} \big|\widetilde{X}_{u}^{n, N,h}-Y_u^{n,h}\big|\Big\Vert_p^2\d s \Big]^{\tfrac{1}{2}}+\psi(t)
\end{align}
with  $C_{d,p,L} = C_{d,p}^{BDG} \, \sqrt{2} \, L > 0$ depends only on $d, p$ and $L$,   and with
\begin{align}
\psi(t)&= L \int_{0}^t \Big\Vert\max_{0\leq \ell\leq  [\underline{s}]}  \mathcal{W}_p\big(\widetilde \mu^{N,h}_{t_{\ell}}, \widetilde \mu^{h}_{t_\ell}\big) \Big\Vert_p \d s + \displaystyle C_{d,p, L} \Big[ \int_{0}^t \Big\Vert \max_{0\leq \ell\leq  [\underline{s}]} \mathcal{W}_p\big(\widetilde \mu^{N,h}_{t_\ell}, \widetilde \mu^{h}_{t_l}\big)\Big\Vert_p^2 \d s\Big]^{\tfrac{1}{2}}.\label{eq:defpsi}
\end{align}
By letting $f(t)\coloneqq \Big\Vert \sup_{s\in[0,t]} \big|\widetilde{X}_{s}^{n, N,h}-Y_s^{n,h}\big|\Big\Vert_p$, Lemma \ref{Gronwall} implies that 
\begin{equation}\label{strongerror1}\Big\Vert \sup_{s\in[0,t]} \big|\widetilde{X}_{s}^{n, N,h}-Y_s^{n,h}\big|\Big\Vert_p\leq 2e^{(2L +C_{d,p, L}^2)\,t} \psi(t).
\end{equation}
Moreover, the empirical measure $\frac{1}{N}\sum_{n=1}^{N}\delta_{(\widetilde{X}^{n, N,h}, Y^{n,h})}$ is a coupling of the two random measures ${\widetilde \mu}^{N,h}=\frac{1}{N}\sum_{n=1}^{N}\delta_{\widetilde{X}^{n, N,h}}$ and $\nu^{N,h}=\frac{1}{N}\sum_{n=1}^{N}\delta_{Y^{n,h}}$. Thus, for $t \in [0,T]$, recalling the notation $\mathbb{W}_{p,t}$ from \eqref{eq:def_truncated_Wass}, 
\begin{align*}
\mathbb{E}\,\mathbb{W}_{p,t}^{p} (\widetilde{\mu}^{N,h}, \nu^{N,h})&=\mathbb{E}\Big[\inf_{\pi\in\Pi(\widetilde{\mu}^{N,h}, \nu^{N,h})}\int_{\mathcal{C}([0, T], \mathbb{R}^{d})\times\mathcal{C}([0, T], \mathbb{R}^{d})}\sup_{s\in[0,t]}\left|x_{s}-y_{s}\right|^{p}\pi(\d x, \d y)\Big]\\
&\leq \mathbb{E}\Big[\int_{\mathcal{C}([0, T], \mathbb{R}^{d})\times\mathcal{C}([0, T], \mathbb{R}^{d})}\sup_{s\in[0,t]}\left|x_{s}-y_{s}\right|^{p}\frac{1}{N}\sum_{n=1}^{N}\delta_{(\widetilde{X}^{\,n, N,h}, Y^{n,h})}(\d x, \d y)\Big]\\
&=\frac{1}{N}\sum_{n=1}^{N}\Big\Vert\sup_{s\in[0,t]}\left| \widetilde{X}_{s}^{n, N,h}- Y_{s}^{n,h}\right|\Big\Vert_{p}^{p} \leq\Big[2\,e^{(2L+C_{d, p, L}^{2})\,T}\psi(t)\Big]^{p}.
\end{align*}

 As $\sup_{s\in[0,t]}\mathcal{W}_{p}^{p}({\widetilde \mu}_{s}^{N,h}, \nu_{s}^{N,h})\leq \mathbb{W}_{p,t}^{p} \big(\widetilde \mu^{N,h}, \nu^{N,h} \big)$  a.s.  (see \cite[Lemma 3.4]{liu2022particle}), 
we obtain
\begin{equation}\label{supw}
\Big\Vert \sup_{s\in[0,t]}\mathcal{W}_{p}(\widetilde{\mu}_{s}^{N,h}, \nu_{s}^{N,h})\Big\Vert_{p}\leq C_{d,p,L, T} \psi(t)
\end{equation}
with $C_{d,p,L, T}=2\,e^{(2L+C_{d, p, L}^{2})\,T}$. 
Setting 
\begin{equation}\label{eq:defnuN}
\nu^{N,h}_t=\frac{1}{N}\sum_{n=1}^{N}\delta_{Y^{n,h}_t},
\end{equation} 
the definition of $\psi(t)$ and \eqref{supw} entail \dd 
\begin{multline*}
\Big\Vert \max_{0\leq \ell \leq [\underline{t}]} \mathcal{W}_p\big({\widetilde \mu}_{t_\ell}^{N,h}, \widetilde{\mu}^h_{t_\ell}\big)\Big\Vert_p
\leq \Big\Vert \sup_{s\in[0,t]}\mathcal{W}_p\big({\widetilde \mu}_{s}^{N,h}, \nu_{s}^{N,h}\big)\Big\Vert_p+\Big\Vert \max_{0\leq \ell \leq [\underline{t}]} \mathcal{W}_p\big(\widetilde{\mu}^h_{t_\ell}, \nu_{t_\ell}^{N,h}\big)\Big\Vert_p\\
\leq C_{d,p,L, T} \int_{0}^t \Big\Vert\max_{0\leq \ell\leq  [\underline{s}]}  \mathcal{W}_p\big({\widetilde \mu}_{t_\ell}^{N,h}, \widetilde \mu^h_{t_\ell}\big) \Big\Vert_p \d s  + \displaystyle  C_{d,p,L, T} \Big[ \int_{0}^t \Big\Vert \max_{0\leq \ell\leq  [\underline{s}]} \mathcal{W}_p\big({\widetilde \mu}_{t_\ell}^{N,h}, \widetilde \mu^h_{t_\ell}\big)\Big\Vert_p^2 \d s\Big]^{\tfrac{1}{2}} 
\\
\qquad + \Big\Vert \max_{0\leq \ell \leq [\underline{t}]} \mathcal{W}_p\big(\widetilde{\mu}^h_{t_\ell}, \nu_{t_\ell}^{N,h}\big)\Big\Vert_p.
\end{multline*}
Using again Lemma \ref{Gronwall}, we get
\begin{align}
\Big\Vert \max_{0\leq \ell \leq [\underline{t}]} \mathcal{W}_p\big({\widetilde \mu}_{t_\ell}^{N,h}, \widetilde{\mu}^h_{t_\ell}\big)\Big\Vert_p\leq 2\exp((2C_{d, p, L, T}+C_{d, p, L, T}^2)t)\Big\Vert \max_{0\leq \ell \leq [\underline{t}]} \mathcal{W}_p\big(\widetilde{\mu}^h_{t_\ell}, \nu_{t_\ell}^{N,h}\big)\Big\Vert_p.\nonumber
\end{align}
Hence, by taking $C'_{d,p,L,T}\coloneqq 2\exp((2C_{d, p, L, T}+C_{d, p, L, T}^2)T)$, we get
\begin{align}\label{eq:max-lessthan-max}
\Big\Vert \max_{0\leq \ell \leq M} \mathcal{W}_p\big({\widetilde \mu}_{t_\ell}^{N,h}, \widetilde{\mu}^h_{t_\ell}\big)\Big\Vert_p &\leq C'_{d,p,L,T}\Big\Vert \max_{0\leq \ell \leq M} \mathcal{W}_p\big(\widetilde{\mu}^h_{t_\ell}, \nu_{t_\ell}^{N,h}\big)\Big\Vert_p.
\end{align}
As Assumption \ref{assump:norm_X0} holds with $p+\varepsilon_0$ for some $\varepsilon_0>0$. Thus, for every $t\in[0,T]$, $\widetilde{\mu}_t^h\in\mathcal{P}_{p+\varepsilon_0}(\RD)$ by Proposition \ref{propXtilde}-(a). Using Definition \ref{def:intermediate_system}, we obtain 
\begin{align}
\label{eq:W_mu_tilde_h_nu_h}
&\Big\Vert \max_{0\leq \ell \leq M} \mathcal{W}_p\big(\widetilde{\mu}^h_{t_\ell}, \nu_{t_\ell}^{N,h}\big)\Big\Vert_p \nonumber \leq (M+1)^{\frac1{p}} \max_{0 \le \ell \le M} \Big\Vert  \mathcal{W}_p\big(\widetilde{\mu}^h_{t_\ell}, \nu_{t_\ell}^{N,h}\big)\Big\Vert_p\nonumber\\
&\leq  C_{b, \sigma, L, T, d, p, q,\varepsilon_0} \big( 1+\Vert X_0\Vert_{p+\varepsilon_0}\big)(M+1)^{\frac1{p}} \\
&\qquad \times\begin{cases}
N^{-1/2p}+N^{-\frac{\varepsilon_0}{p(p+\varepsilon_0)}} & \:\mathrm{if}\:p>d/2\:\mathrm{and}\:\varepsilon_0\neq p\\
N^{-1/2p}\big(\log(1+N)\big)^{1/p}+N^{-\frac{\varepsilon_0}{p(p+\varepsilon_0)}} & \:\mathrm{if}\:p=d/2\:\mathrm{and}\:\varepsilon_0\neq p\\
N^{-1/d}+N^{-\frac{\varepsilon_0}{p(p+\varepsilon_0)}} & \:\mathrm{if}\:p\in(0, d/2)\:\mathrm{and}\:\varepsilon_0\neq d/(d-p)-p
\end{cases},\nonumber
\end{align}
where the second inequality follows from Theorem~\ref{FG} since for every $\ell\in\{0, ..., M\}$, $\nu_{t_\ell}^{N,h}$ is, by definition, an empirical measure of ${\widetilde{\mu}^h_{t_\ell}}$. 
A direct combination with Propositions~\ref{prop:cvg_cont_Euler} and~\ref{propXtilde} gives
\begin{align}
\label{eq:bound_with_M}
    &\Big\| \max_{0 \le \ell \le M} \calW_p \big(\widetilde \mu_{t_\ell}^{N,h}, \mu_{t_ \ell} \big) \Big\|_p \le C_{b,\sigma, L, T, d, p, \varepsilon_0, \|X_0\|_{p + \varepsilon_0}} \\
    &\times \begin{cases}
M^{-\gamma} + M^{-\frac12}  \log(M)^{\frac{1}{2}} + M^{\frac1{p}} \big(N^{-\frac{1}{2p}}+N^{-\frac{\varepsilon_0}{p(p+\varepsilon_0)}} \big),  \quad \mathrm{if}\:p>d/2\:\mathrm{and}\:\varepsilon_0\neq p,\\
M^{-\gamma} + M^{-\frac12}  \log(M)^{\frac{1}{2}}  + M^{\frac1{p}} \big(N^{-\frac{1}{2p}}\big(\log(1+N)\big)^{1/p}+N^{-\frac{\varepsilon_0}{p(p+\varepsilon_0)}}\big), \quad\:\mathrm{if}\:p=d/2\:\mathrm{and}\:\varepsilon_0\neq p,\\
M^{-\gamma} + M^{-\frac12}  \log(M)^{\frac{1}{2}} + M^{\frac1{p}} \big(N^{-\frac{1}{d}}+N^{-\frac{\varepsilon_0}{p(p+\varepsilon_0)}}\big) \quad\:\mathrm{if}\:p\in(0, \tfrac{d}2)\:\mathrm{and}\:\varepsilon_0\neq \frac{d}{d-p} - p.
\end{cases}\nonumber 
\end{align} 
From there, choosing $M$ as in the statement yields the result.
\qedhere
\end{proof}

\subsubsection{Clarification and generalization of \cite{MR1329874} and proof of Theorem~\ref{thm:particlemethod}-$(b)$}\label{subsec:HKmethod}



From \eqref{eq:max-lessthan-max}, an alternative way to upper bound the term $\Vert\max_{0\leq \ell \leq M} \mathcal{W}_p(\widetilde{\mu}^h_{t_\ell}, \nu_{t_\ell}^{N,h})\Vert_p$ is to rely on the method introduced by Horowitz and Karandikar~\cite{MR1329874}, see also~\cite[Theorem~10.2.7]{MR1619171}.
We first generalize and clarify the result of~\cite{MR1329874} from the case $p=2$ to arbitrary $p\geq2$  in Theorem \ref{thm:sup-W2-in-MR1329874}, whose proof is postponed to Appendix~\ref{proofsubsec:prop_im}. We then apply this result to complete the proof of Theorem~\ref{thm:particlemethod}--$(b)$.

\begin{thm}\label{thm:sup-W2-in-MR1329874}
Let $n \ge 1$ and let $(\xi_t)_{t\in[0,1]}, (\xi_t^i)_{t\in[0,1]}, 1 \le i \le n$ be i.i.d. stochastic processes taking values in $\mathcal{C}\big([0,1], \RD\big)$. For every $t\in [0,1]$, set $\mu_t^n=\frac{1}{n}\sum_{i=1}^{n}\delta_{\xi_t^i}\dd$ and let $\mu_t$ denote the law of $\xi_t$. Assume that there exist $2 \le p < q$, $\ell \in (0,2]$ and $c<\infty$ such that
\begin{enumerate}[(a)]
    \item $\EE\big[|\xi_t|^{d+2p+1}\big]\leq c,\;0\leq t\leq 1$;
    \item $\EE \big[ |\xi_s-\xi_r|^q|\xi_s-\xi_t|^q\big]\leq c|t-r|^2,\;0\leq r<s<t\leq 1$;
    \item $\EE \big[ |\xi_s-\xi_r|^\ell |\xi_s-\xi_t|^\ell \big]\leq c|t-r|^2,\;0\leq r<s<t\leq 1$; 
    \item  $\EE \big[ |\xi_s-\xi_t|^q\big]\leq c|t-s|$, $0\leq s\leq t\leq 1$;
    \item $\EE \big[ |\xi_s-\xi_t|^2\big]\leq c|t-s|$, $0\leq s\leq t\leq 1$.
\end{enumerate}
Then, for all $\delta > 0$, there exists a constant $C_{q,c,d,p,\delta} > 0$ such that for all $n \ge 1$,
\[ \EE \Big[\sup_{t\in[0,1]}\mathcal{W}_{p}^p(\mu_t^n, \mu_t)\Big]\leq C_{q,c,d,p,\delta}\,n^{-p/(d+4p) + \delta}. \]
\end{thm}
\begin{rem}
The constant \(C_{q,c,d,p,\delta}\) appearing in Theorem~\ref{thm:sup-W2-in-MR1329874} tends to $+\infty$ as \(\delta \to 0\), and behaves proportionally to $1/\delta$, as shown in Step~3 of the proof. 
\end{rem}

\begin{prop}\label{prop:another-cvg-rate}
Assume that Assumption~\ref{assump:norm_X0} holds with \(\mathfrak{p} = d + 2p + 1\), and that Assumptions~\ref{assump:lipschitz} and~\ref{assump:holder} hold with \(\mathfrak{p} = p\). Then for $M\geq \max(2T+1, N^{\,\frac{1}{(d+4p)\min(\gamma, 1/2)}})$ and for every $\delta >0$, we have
\begin{align}\label{eq:cvg-rate-from=HK}
&\Big\Vert \max_{0\leq \ell \leq M} \mathcal{W}_p\big(\widetilde{\mu}_{t_\ell}^{N,h}, \mu_{t_\ell}\big)\Big\Vert_p\leq  C_{b, \sigma, L, T, d, p,\Vert X_0\Vert_{d+2p+1},\delta} N^{-\frac1{d+4p}+\delta}.
\end{align}
\end{prop}
\begin{proof}[Proof of Proposition \ref{prop:another-cvg-rate}]
It suffices to prove that the following estimate holds for every $\delta >0$, 
\begin{equation}
\label{eq:intermediate_HK_rate}
\Big\Vert \max_{0\leq \ell \leq M} \mathcal{W}_p\big({\widetilde \mu}_{t_\ell}^{N,h}, \widetilde{\mu}^h_{t_\ell}\big)\Big\Vert_p
\leq  C_{b, \sigma, L, T, d, p,  \delta} \big( 1+\Vert X_0\Vert_{d+2p+1}\big) N^{-\frac1{d+4p}+\delta},
\end{equation}
Indeed, once~\eqref{eq:intermediate_HK_rate} is established, combining it with Proposition~\ref{prop:cvg_cont_Euler}  yields
\begin{equation}
\label{eq:HK_final}
\Big\Vert \max_{0\leq \ell \leq M} \mathcal{W}_p\big(\widetilde{\mu}_{t_\ell}^{N,h}, \mu_{t_\ell}\big)\Big\Vert_p
\leq  C_{b, \sigma, L, T, d, p, \|X_0\|_{d+2p+1}, \delta} \Big( h^\gamma + (h |\log(h)|)^{\frac12} + N^{-\frac1{d+4p}+\delta}\Big),
\end{equation}
then choosing $M$ such that $M \ge N^{\,\frac{1}{(d+4p)\min(\gamma, 1/2)}}$ yields~\eqref{eq:cvg-rate-from=HK}.

We thus turn to the proof of~\eqref{eq:intermediate_HK_rate}.  
We begin by defining the process $(R_t)_{t \in [0,1]}$ such that $\widetilde X^h_t = R_{\frac{t}{T}}$ for all $t \in [0,T]$, where \((\widetilde X^h_t)_{t \in [0,T]})\) is defined by \eqref{eq:def_continuous_2}. We check that the process \((R_t)_{t \in [0,1]}\) satisfies the five conditions of Theorem \ref{thm:sup-W2-in-MR1329874}. 

For the first condition, remark that Proposition \ref{propXtilde}-(a) implies  $$\sup_{t\in[0,T]}\Vert{\widetilde{X}_t^h}\Vert_{d + 2p+1}\leq   \Gamma(1+\vertii{X_0}_{d + 2p+1}) $$ so that $\sup_{t\in[0,1]}\EE \big[\big|R_t\big|^{d+2p+1}\big]\leq \Gamma^{d+2p+1}(1+\vertii{X_0}_{d + 2p+1})^{d+2p+1}$.

For the second condition, set $\mathfrak{r}=\frac{1}{2}(d + 2p +1)$ and remark that $\mathfrak{r}>2$. Then by Proposition \ref{propXtilde}-(c), we have for all $t,s\in[0,1]$,
\begin{equation}\label{eq:l6norm-Xtilde}
\| R_t - R_s\|_{2\mathfrak{r}} =
\left\Vert\widetilde{X}_{Tt}^h-\widetilde X_{Ts}^h\right\Vert_{2\mathfrak{r}}= \left\Vert\widetilde{X}_{Tt}^h-\widetilde X_{Ts}^h\right\Vert_{d+2p+1} \leq \sqrt{T} \, \tilde \kappa \, |t-s|^{\frac{1}{2}}.  
\end{equation} It follows that for every $0\leq r<s<t\leq 1$, we have
\begin{align*}
    \mathbb{E} \Big[ \left| R_s - R_r \right|^{\mathfrak{r}} & \left| R_s - R_t\right|^{\mathfrak{r}} \Big] \le \Big\{ \EE \left[ \left| R_s - R_r\right|^{2{\mathfrak{r}}} \right] \Big\}^{\frac12} \Big\{ \EE \left[ \left| R_s - R_t\right|^{2{\mathfrak{r}}} \right] \Big\}^{\frac12} \\
    &\le \big(\tilde \kappa \sqrt{T}\big)^{2\mathfrak{r}} (s-r)^{\frac{\mathfrak{r}}{2}} (t-s)^{\frac{\mathfrak{r}}{2}} \le (\tilde \kappa \sqrt{T})^{2\mathfrak{r}} \, (t-r)^\mathfrak{r} \le \big(\tilde \kappa \sqrt{T}\big)^{2\mathfrak{r}} \, (t-r)^2
\end{align*}
where the first inequality follows from the Cauchy–Schwarz inequality, the second uses the bound in~\eqref{eq:l6norm-Xtilde}, and the third and fourth ones use $0 \le \max(t-s, s-r) \le t - r \le 1$ and $\mathfrak{r} > 2$.

For the third condition, noticing that $
\| R_t - R_s\|_{4}\leq\| R_t - R_s\|_{2\mathfrak{r}} \leq \sqrt{T} \, \tilde \kappa \, |t-s|^{\frac{1}{2}}$, a similar computation yields 
\begin{align*}
    \EE \Big[|R_s - R_r|^2 |R_s - R_t|^2\Big] \le \tilde \kappa^4 T^2 (t-r)^2. 
\end{align*} 

For the fourth and fifth condition, the inequality~\eqref{eq:l6norm-Xtilde} directly  implies that for all $0 \le s \le t \le 1$, 
\begin{align*}
   \|R_t - R_s\|_2 \le \|R_t - R_s\|_\mathfrak{r} \le \|R_t - R_s\|_{2\mathfrak{r}} \le \sqrt{T} \tilde \kappa |t-s|^{\frac12}
\end{align*}
so that 
\begin{align*}
    \EE \Big[ \big|R_t - R_s\big|^2 \Big] \le T \, \tilde \kappa^2 |t-s|, \qquad \EE \Big[ \big|R_t - R_s\big|^\mathfrak{r} \Big] \le \big(\sqrt{T} \tilde \kappa \big)^\mathfrak{r} |t-s|^{\mathfrak{r}/2} \le \big(\sqrt{T} \tilde \kappa \big)^\mathfrak{r} |t-s|
\end{align*}
where we used $|t-s| \le 1$ and that $\mathfrak{r} > 2$. 

Therefore, by applying Theorem \ref{thm:sup-W2-in-MR1329874}, we obtain for all $\delta > 0$
\begin{equation}\label{eq:application-thm-3.11}
  \EE \left[ \max_{0\leq \ell \leq M} \mathcal{W}_p^p\big(\widetilde{\mu}^h_{t_\ell}, \nu_{t_\ell}^{N,h}\big)\right] \leq \EE \left[\sup_{t\in[0,T]} \mathcal{W}_p^p\big(\widetilde{\mu}^h_{t}, \nu_{t}^{N,h}\big)\right]\leq C N^{-\frac{p}{d+4p}+\delta}
\end{equation}
with the constant $C$ independent of $N$ and behaving as $\frac1{\delta}$ as $\delta \to 0$. 
Consequently, \eqref{eq:intermediate_HK_rate} follows from \eqref{eq:max-lessthan-max} together with \eqref{eq:application-thm-3.11}.
\end{proof}

\begin{proof}[Proof of Theorem~\ref{thm:particlemethod}-$(b)$]
It suffices to compare the convergence rate obtained in~\eqref{eq:cvg-rate-from=HK} with that of~\eqref{mainresult-review}.

First, note that if $\Vert X\Vert_{d+2p+1}<+\infty$, then by taking $\varepsilon_0=d+p+1$, we obtain 
\[\alpha =\frac{1}{2p}\mathbbm{1}_{\{p\geq \frac{d}{2}\}}+\frac{1}{d}\mathbbm{1}_{\{p<\frac{d}{2}\}} \]
in Theorem \ref{thm:particlemethod}-$(a)$. 

\noindent $(i)$ When $\gamma\in[\frac{1}{2}, 1]$, and  $d>2p^2$, the convergence rate in ~\eqref{eq:cvg-rate-from=HK} is better than the one given in \eqref{mainresult-review}.

\noindent $(ii)$ When $\gamma\in(0,\frac{1}{2})$, the bound in \eqref{mainresult-review} simplifies to 
\begin{align}\label{mainresult-small-gamma}
&\Big\Vert \max_{0\leq \ell \leq M} \mathcal{W}_p\big(\widetilde{\mu}_{t_\ell}^{N,h}, \mu_{t_\ell}\big)\Big\Vert_p\leq  C_{b, \sigma, L, T, d, p, \varepsilon_0, \Vert X_0\Vert_{p+\varepsilon_0}} \cdot\begin{cases}
N^{-\frac{ \gamma}{2(1+  \gamma p)}}\quad \text{ if } p>\frac{d}{2},\\
N^{-\frac{ \gamma}{2(1+  \gamma p)}} \big(\log(N)\big)^{\frac{1}{p}}\quad \text{ if } p=\frac{d}{2},\\
N^{-\frac{ p  \gamma}{d(1+ \gamma p)}} \quad \text{ if } p<\frac{d}{2}.
\end{cases}\nonumber
\end{align}
In this regime, if 
\[\gamma<  \frac{2}{d+2p}\mathbbm{1}_{\{p\geq \frac{d}{2}\}}+\frac{d}{4p^2}\mathbbm{1}_{\{p<\frac{d}{2}\}},\]
then the convergence rate in~\eqref{eq:cvg-rate-from=HK} is better than the one in~\eqref{mainresult-review}.
\end{proof}





\smallskip

\begin{proof}[Proof of Corollary \ref{cor}]  Using again \eqref{eq:defpsi}, \eqref{strongerror1} and \eqref{eq:max-lessthan-max}, we get 
\begin{align}
&\mathbb{W}_p\Big(\mathcal{L}(\widetilde{X}^{1, N,h}), \mathcal{L}(Y^{n,h})\Big)\leq \Big\Vert\sup_{s\in[0,T]} \big|\widetilde{X}_{s}^{n, N,h}-Y_s^{n,h}\big|\Big\Vert_p \leq  C_{d,p,L,T}\Big\Vert \max_{0\leq \ell \leq M} \mathcal{W}_p\big(\widetilde{\mu}^h_{t_\ell}, \nu_{t_\ell}^{N,h}\big)\Big\Vert_p.\nonumber 
\end{align}
The result follows by combining Proposition \ref{prop:cvg_cont_Euler}, Proposition \ref{propXtilde}, Theorem \ref{thm:particlemethod}, together with the following inequality
\[\mathbb{W}_p\big(\mathcal{L}(\widetilde{X}^{1, N,h}), \mathcal{L}(X)\big)\leq \mathbb{W}_p\big(\mathcal{L}(\widetilde{X}^{1, N,h}), \mathcal{L}(Y^{n,h})\big)+\mathbb{W}_p\big(\mathcal{L}(X), \mathcal{L}(Y^{n,h})\big), \hfill\]
where the discretization parameter $M$ is chosen as in Theorem~\ref{thm:particlemethod}.
\end{proof}

\bigskip

\noindent\LARGE \textbf{Appendix}

\normalsize

\begin{appendix}

\section{Propagation of chaos for the particle system}
\label{sec:chaos}


In this section, we prove Theorem \ref{thm:global_thm_chaos}. Our proof of propagation of chaos relies on the use of a synchronous coupling with i.i.d. particles sharing the marginal distributions of the solution to \eqref{eq:path-dependent_McKean}. This is of course reminiscent of the celebrated approach developed in dimension 1 by Sznitman \cite{Sznitman_1991}, although our proof is more in the line of the recent exposition of Lacker \cite{lacker2018mean}.

Let us first note that the particle system \eqref{eq:particle_system_intro} is well-defined by Theorem \ref{thm:well-posedness}, as this can be written as a path-dependent diffusion of dimension $Nd$, which of course is a particular case of application of Theorem \ref{thm:well-posedness}.
Let $(Y^i_t)_{t \ge 0}$, $1 \le i \le N$, be $N$ processes solving
\begin{align}
	\label{eq:system_Y}
	Y^i_t &= X^{i,N}_0 + \int_0^t b \big(s, Y^i_{\cdot \wedge s}, \mu_{\cdot \wedge s} \big) \, \d s + \int_0^t \sigma \big(s, Y^i_{\cdot \wedge s},  \mu_{\cdot \wedge s} \big) \, \d B^{i}_s, \qquad t \in [0,T],
\end{align}
where $(\mu_s)_{s \in [0,T]}$ in the coefficient functions are the marginal distributions of the unique solution $X$ of \eqref{eq:path-dependent_McKean} given by Theorem \ref{thm:well-posedness} and where $B^i = (B^i_t)_{t \in [0,T]}$, $1 \le i \le N$ are the same i.i.d. standard Brownian motions considered in the particle system \eqref{eq:particle_system_intro}. Recall that $X^{i,N}_0$, $1 \le i \le N$ are i.i.d. copies of $X_0$. It follows from the uniqueness in Theorem \ref{thm:well-posedness} that $Y^i = (Y^i_t)_{t \in [0,T]}$ are i.i.d. copies of $X$. 

With the help of the following lemma, we define, for all $\omega \in \Omega$,
\begin{equation}
	\label{eq:def_nu_omega}
	\nu^N(\omega) := \frac1{N} \sum_{i=1}^N \delta_{Y^i(\omega)}, 
\end{equation}
which is a random measure valued in $\calP_p(\mathcal{C}([0,T], \R^d))$ for all $p \ge 1$. 

\begin{lem}
	\label{lem:marginals}
	Let $\alpha^i = (\alpha^i_t)_{t \in [0,T]}$, $1 \le i \le N$ be elements of $\mathcal{C}([0,T], \R^d)$. Then 
	\begin{enumerate}
		\item [(a)] the empirical measure $\nu^{N,\alpha} := \frac1{N} \sum_{i=1}^N \delta_{\alpha^i} \in \calP_p(\mathcal{C}([0,T], \R^d))$ for all $p \ge 1$; 
		\item [(b)] let $\iota: \calP_p(\mathcal{C}([0,T], \R^d)) \to \mathcal{C}([0,T], \calP_p(\R^d))$ be the application defined in Lemma \ref{injectionmeasure}. Then
		\[ \iota \big( \nu^{N, \alpha} \big) = \Big( \frac1{N} \sum_{i=1}^N \delta_{\alpha^i_t} \Big)_{t \in [0,T]}. \]
	\end{enumerate}
\end{lem}

\begin{proof}
	$(a)$ For all $p \ge 1$,
	\begin{align*}
		\int_{\mathcal{C}([0,T], \R^d)} \|x\|_{\sup}^p \, \nu^{N,\alpha}(\d x) = \int_{\mathcal{C}([0,T], \R^d)} \|x\|_{\sup}^p \, \Big( \frac1{N} \sum_{i=1}^N \delta_{\alpha^i}\Big)(\d x) = \frac1{N} \sum_{i=1}^N \|\alpha^i\|_{\sup}^p < \infty.
	\end{align*}
	\medskip 
	$(b)$  Recall the definition of the coordinate map $\pi_t$ from \eqref{eq:pi_coordinate}. We only need to prove that for a fixed $t \in [0,T]$, 
	\begin{align}
		\label{eq:claim_lemma_chaos}
		\nu^{N,\alpha} \circ \pi_t^{-1} = \frac1{N} \sum_{i=1}^N \delta_{\alpha^i_t}.
	\end{align} 
	Obviously both sides are probability measures on $(\R^d, \mathcal{B}(\R^d))$. Let $B \in \mathcal{B}(\R^d)$, we have
	\begin{align*}
		\nu^{N,\alpha} \circ \pi_t^{-1}(B) 
		&= \Big(\frac{1}{N} \sum_{i=1}^N \delta_{\alpha^i} \Big) \Big( \big\{ \beta \in \mathcal{C}([0,T], \R^d): \pi_t(\beta) \in B \big\} \Big) \\
		&= \frac1{N} \sum_{i=1}^N \delta_{\alpha^i}  \Big( \big\{ \beta \in \mathcal{C}([0,T], \R^d): \beta_t \in B \big\} \Big),
	\end{align*}
	where we used that $\pi_t(\beta) = \beta_t$. Notice in addition that
	\[ 
	\delta_{\alpha^i}  \Big( \big\{ \beta \in \mathcal{C}([0,T], \R^d): \beta_t \in B\big\} \Big) = 
	\left\{ 
	\begin{array}{l} 
		1 \quad \hbox{ if } \alpha^i_t \in B, \\
		0 \quad \hbox{ otherwise}. 
	\end{array}
	\right.\]
	On the other hand, $(\frac1{N} \sum_{i=1}^N \delta_{\alpha^i_t} ) (B) = \frac1{N} \sum_{i=1}^N \delta_{\alpha^i_t}(B)$
	where 
	\[ \delta_{\alpha^i_t}(B) = 
	\left\{
	\begin{array}{l}
		1 \quad \hbox{ if } \alpha^i_t \in B, \\
		0 \quad \hbox{ otherwise}. 
	\end{array} 
	\right.\]
	It follows that for all $B \in \mathcal{B}(\R^d)$, $(\nu^{N,\alpha} \circ \pi_t^{-1}) (B) = ( \frac1{N} \sum_{i=1}^N \delta_{\alpha^i_t})(B)$
	and finally that \eqref{eq:claim_lemma_chaos} holds. 
\end{proof}

With the processes $Y^i$, $1 \le i \le N$ from \eqref{eq:system_Y} at hand, we introduce a family of random distributions $(\nu^N_t)_{t \in [0,T]}$ defined by
\begin{align}
	\label{eq:def_nu} \forall\,  \omega \in \Omega, \quad  t \in [0,T], \qquad \nu^N_t(\omega) := \frac1{N} \sum_{i=1}^N \delta_{Y^i_t(\omega)}. 
\end{align}
 Lemma \ref{injectionmeasure}  guarantees that for every $\omega$, $(\nu^N_t(\omega))_{t \in [0,T]} \in \mathcal{C}([0,T], \calP_p(\R^d))$ since
\[ (\nu^N_t(\omega))_{t \in [0,T]} = \Big( \frac{1}{N} \sum_{i=1}^N \delta_{Y^i(\omega)_t} \Big)_{t \in [0,T]}. \] Moreover, by Lemma \ref{lem:marginals}, for every $\omega \in \Omega$, 
\[ \Big( \nu^N_t(\omega) \Big)_{t \in [0,T]} = \iota \Big( \nu^N(\omega) \Big) \]
so that $(\nu^N_t(\omega))_{t \in [0,T]}$ can be identified with the marginal distributions of $\nu^N(\omega)$.

\begin{proof} [Proof of Theorem \ref{thm:global_thm_chaos}] $(a)$ Fix $N > 1$ and $i \in \{1, \dots, N\}$. Let $Y^i$, $1 \le i \le N$ be the solutions to  \eqref{eq:system_Y}  and $\nu^N$ the associate empirical measure defined by \eqref{eq:def_nu_omega}, as detailed above.  
	We have, for all $s \in [0,T]$,
	\begin{align*}
		X^{i,N}_s - Y^i_s &= \int_0^s \Big[ b \big(u, X^{i,N}_{\cdot \wedge u}, \mu^N_{\cdot \wedge u} \big) - b \big( u, Y^i_{\cdot \wedge u}, \mu_{\cdot \wedge u} \big) \Big] \d u  \\
		&\qquad + \int_0^s \Big[ \sigma\big(u, X^{i,N}_{\cdot \wedge u}, \mu^N_{\cdot \wedge u} \big) - \sigma(u, Y^i_{\cdot \wedge u}, \mu_{\cdot \wedge u} \big) \Big] \d B^{i}_u. 
	\end{align*}
	We set, for all $t \in [0,T]$ and using the exchangeability
	\begin{align*}
		\bar f(t) := \sup_{1 \le i \le N} \left\Vert \sup_{s \in [0,t]} \Big| X^{i,N}_s - Y^i_s \Big| \right\Vert_p = \left\Vert \sup_{s \in [0,t]} \Big| X^{1,N}_s - Y^1_s \Big| \right\Vert_p.
	\end{align*}
	  By Lemmas~\ref{gemin} and \ref{lem:BDG_for_sigma}, for all $i \in \llbracket N \rrbracket^*$,
	\begin{align}
		\label{eq:temp_prop_chaos_1}
		\bar f (t) &= \left\Vert \sup_{s \in [0,t]} \Big| X^{i,N}_s - Y^i_s \Big| \right\Vert_p \nonumber \\
		&\le \int_0^t \Big\| b\big(s, X^{i,N}_{\cdot \wedge s}, \mu^N_{\cdot \wedge s} \big) - b\big(s, Y^i_{\cdot \wedge s}, \mu_{\cdot \wedge s}\big) \Big\|_p \d s \nonumber\\
		&\qquad + C_{d,p}^{BDG} \Big[ \int_0^t \Big\| \, \vertiii{\sigma \big(s, X^{i,N}_{\cdot \wedge s}, \mu^N_{\cdot \wedge s} \big) - \sigma \big(s, Y^i_{\cdot \wedge s}, \mu_{\cdot \wedge s}\big) } \, \Big\|_p^2 \d s \Big]^{\tfrac12}.
	\end{align}
	Recall that the marginal distributions $(\mu^N_t)_{t \in [0,T]}$ are themselves random, so that for all $u \in [0,T]$,
	\[ \Big\|  \sup_{v\in[0,T]}\mathcal{W}_p \big( (\mu^N_{v \wedge u})_{v \in [0,T]}, (\mu_{v \wedge u})_{v \in [0,T]} \big) \Big\|_p = \Big\| \sup_{v \in [0,u]} \mathcal{W}_p(\mu^N_v, \mu_v)  \Big\|_p. \] 
	Arguing as for the derivation of \eqref{ineq1} (in fact, computations are easier here since there is no use of the interpolator), and using the definition of $\bar f$, we get 
	\begin{align}
		\label{eq:control_b_prop_chaos}
		\int_0^t \Big\| b\big(s, X^{i,N}_{\cdot \wedge s}, \mu^N_{\cdot \wedge s} \big) &- b\big(s, Y^i_{\cdot \wedge s}, \mu_{\cdot \wedge s}\big) \Big\|_p \d s \nonumber \\
		& \le L \int_0^t \bar f(s) \, \d s + L \int_0^t \Big\| \sup_{v \in [0,s]}  \mathcal{W}_p \big( \mu^N_v, \mu_v \big) \Big\|_p \d s.  
	\end{align}
Adapting the derivation of \eqref{ineq2}, we also find
	\begin{align}
		\label{eq:prop_chaos_sigma}
		C_{d,p}^{BDG} & \Big[ \int_0^t  \left \Vert \vertiii{\sigma \big(s, X^{i,N}_{\cdot \wedge s}, \mu^N_{\cdot \wedge s} \big) - \sigma \big(s, Y^i_{\cdot \wedge s}, \mu_{\cdot \wedge s}\big) } \right \Vert_p^2 \d s \Big]^{\frac{1}{2}}  \nonumber \\
		&= \sqrt{2} C_{d,p}^{BDG} L \Big\{  \Big[ \int_0^t \bar f(s)^2 \d s \Big]^{\frac12} + \Big[ \int_0^t \Big\|\sup_{v \in [0,s]} \mathcal{W}_p\big(\mu^N_v, \mu_v \big) \Big\|_p^2 \d s \Big]^{\frac12} \Big\}.
	\end{align}
	By using the triangle inequality, for all $s \in [0,T]$, we get first 
	\begin{align}
		\label{eq:prop_chaos_temp_1}
		\sup_{v \in [0,s]} \mathcal{W}_p^p \big( \mu^N_v, \mu_v \big) 
		&\le 2^{p-1} \Big( \sup_{v \in [0,s]} \mathcal{W}_p^p \big(\mu^N_v, \nu^N_v\big) + \sup_{v \in [0,s]} \mathcal{W}_p^p \big( \nu^N_v, \mu_v\big) \Big).
	\end{align}
	In addition, the empirical measure defined for all $t \in [0,T]$ by $\frac1{N} \sum_{i=1}^N \delta_{(X^{i,N}_t, Y^i_t)}$ is a random coupling of the empirical measures $\mu^N_t$ and $\nu^N_t$. Thus, for all $v \in [0,T]$, 
	\begin{align*}
		\mathcal{W}_p^p \big(\mu^N_v, \nu^N_v\big) &\le \int_{\R^d \times \R^d } |x - y|^p \frac1{N} \sum_{i=1}^N \delta_{(X^{i,N}_v, Y^i_v)} (\d x, \d y) = \frac1{N} \sum_{i=1}^N \Big| X^{i,N}_v - Y^i_v \Big|^p.
	\end{align*}
	Taking the supremum over $[0,s]$ and the expectation, noticing that \[ \sup_{v \in [0,s]} \sum_{i=1}^N |X^{i,N}_v - Y^i_v|^p \le \sum_{i=1}^N \sup_{v \in [0,s]} |X^{i,N}_v - Y^i_v|^p \]
	almost surely, we get 
	\begin{align*}
		\mathbb{E}\Big[ \sup_{v \in [0,s]} \mathcal{W}_p^p \big(\mu^N_v, \nu^N_v \big) \Big] &\le \frac{1}{N} \sum_{i=1}^N \mathbb{E}\Big[ \sup_{v \in [0,s]} \big| X^{i,N}_v - Y^i_v \big|^p \Big] = \bar f(s)^p.
	\end{align*}
	Taking the expectation in \eqref{eq:prop_chaos_temp_1} and using this last inequality, using also $p \ge 2$, we get 
	\begin{align}
		\label{eq:prop_chaos_Wasserstein}
		\Big\|\sup_{v \in [0,s]} \mathcal{W}_p \big( \mu^N_v, \mu_v \big) \Big\|_p \le 2 \Big(  \bar f(s) + \Big\| \sup_{v \in [0,s]} \mathcal{W}_p \big( \nu^N_v, \mu_v \big) \Big\|_p \Big). 
	\end{align}
	Bringing together \eqref{eq:temp_prop_chaos_1}, \eqref{eq:control_b_prop_chaos}, \eqref{eq:prop_chaos_sigma} and \eqref{eq:prop_chaos_Wasserstein} we deduce
	\begin{align*}
		\bar f(t) &\le 3L \int_0^t \bar f(s) \, \d s + 2 L \int_0^t \Big\| \sup_{v \in [0,s]} \mathcal{W}_p\big( \nu^N_v, \mu_v \big) \Big\|_p \d s \\
		& \qquad + \big( \sqrt{2} + 4 \big)  C_{d,p}^{BDG} L \Big[ \int_0^t \bar f(s)^2 \, \d s \Big]^{\tfrac12} \\
		&\qquad + 4 C_{d,p}^{BDG} L \Big[ \int_0^t  \Big\|\sup_{v \in [0,s]} \mathcal{W}_p \big(\nu^N_v, \mu_v\big)\Big\|_p^2 \d s \Big]^{\tfrac12}, 
	\end{align*}
	and Lemma \ref{Gronwall} then yields
	\begin{align}
		\label{eq:control_bar_f}
		\bar f(t) &\le 4L e^{\kappa_0 t} \Big\{ \int_0^t \Big\|\sup_{v \in [0,s]} \mathcal{W}_p(\nu^N_v, \mu_v)\Big\|_p \d s   +  2 C_{d,p}^{BDG} \Big[ \int_0^t \Big\| \sup_{v \in [0,s]} \mathcal{W}_p \big( \nu^N_v, \mu_v \big)\Big\|_p^2 \d s \Big]^{\frac12} \Big\} 
	\end{align}
	with $\kappa_0 := 6L + ((\sqrt{2} + 4) C_{d,p}^{BDG} L )^2 > 0$. Injecting this result into \eqref{eq:prop_chaos_Wasserstein} and using that for all $s \in [0,T]$,  $\sup_{v \in [0,s]} \mathcal{W}_p(\nu^N_v, \mu_v) \le \sup_{v \in [0,T]} \mathcal{W}_p(\nu^N_v, \mu_v)$ almost surely, we get
	\begin{align}
		\label{eq:temp_chaos_3}
		\Big\| \sup_{v \in [0,T]} \mathcal{W}_p \big( \mu^N_v, \mu_v \big) \Big\|_p &\le 2 \Big(1+ 4Le^{\kappa_0 T} (T + 2 C_{d,p}^{BDG} \sqrt{T}) \Big)\, \Big\| \sup_{v \in [0,T]} \mathcal{W}_p \big( \nu^N_v, \mu_v \big) \Big\|_p. 
	\end{align}
	By Lemma \ref{lem:marginals}, $(\nu^N_v)_{v \in [0,T]}$ can be identified with the marginal distributions of $\nu^N$. Moreover, by Lemma \ref{injectionmeasure}, the map $\iota$ is $1$-Lipschitz continuous. Thus, 
	\begin{align*}
		\Big\| \sup_{v \in [0,T]} \mathcal{W}_p\big(\nu^N_v, \mu_v\big) \Big\|_p \le \Big\| \mathbb{W}_p(\nu^N, \mu) \Big\|_p.
	\end{align*} 
	Combining this inequality with \eqref{eq:temp_chaos_3} concludes the proof of \eqref{eq:thm_chaos_1}. The limit is obtained by applying the convergence of $\|\mathcal{W}_p(\nu^N, \mu)\|_p$ with $\nu^N$ being an empirical measure of i.i.d. processes with distribution $\mu$ on the separable metric space of infinite dimension $\mathcal{C}([0,T], \R^d)$, see for instance \cite[Theorem 6.6]{Parthasarathy_1967} for convergence in probability and the corollary \cite[Corollary 2.14]{lacker2018mean} for our setting.
	
	To prove \eqref{eq:thm_chaos_2}, we simply note that for all $k \in \{1,\dots, N\}$
	\begin{align*}
		\mathbb{E} \Big[ \sup_{1 \le i \le k} \sup_{t \in [0,T]} \big| X^{i,N}_t - Y^i_t\big|^p \Big] &\le \sum_{i=1}^k \mathbb{E}\Big[ \sup_{s \in [0,T]} \big|X^{i,N}_s - Y^i_s\big|^p \Big] \le k \bar f(t)^p \le C_{p,d,T,L} k \Big\| \mathbb{W}_p(\nu^N, \mu) \Big\|_p^p,
	\end{align*} 
	where we applied \eqref{eq:control_bar_f} to obtain the last inequality, with a constant $C_{p,d,T,L} > 0$. We conclude by using again \cite[Corollary 2.14]{lacker2018mean}. 

\smallskip

\noindent $(b)$ We now turn to the proof of \eqref{eq:thm_chaos_3}. By \eqref{eq:temp_chaos_3}, it is sufficient to show that under Assumption \ref{assump:norm_X0}  with $\mathfrak{p} \ge d+2p+1$, for all $\delta > 0$, the quantity  $\big\|\sup_{v \in [0,T]} \mathcal{W}_p \big( \nu^N_v, \mu_v \big) \big\|_p$ admits an upper bound of the form $CN^{-\frac{1}{d+4p} + \delta}$ for some constant $C>0$ depending also on $\delta$. To this end, we invoke Theorem \ref{thm:sup-W2-in-MR1329874}. It remains to verify that the process $(X_t)_{t\in[0,T]}$, solution to the path-dependent McKean–Vlasov equation \eqref{eq:path-dependent_McKean}, satisfies the conditions (a)-(e) required by Theorem \ref{thm:sup-W2-in-MR1329874}. This verification follows the same reasoning as in the proof of Proposition \ref{prop:another-cvg-rate}, and we omit details for conciseness. 
\end{proof}



\section{Proofs of Section \ref{sec:Applications} and Section \ref{sec:numericalcvg}}\label{appB}

\subsection{Proofs of Section \ref{sec:Applications}}\label{proofsubsec:appli}

We provide in this appendix the proofs of the results from Section \ref{sec:Applications}.



\begin{proof}[Proof of Proposition \ref{prop:linear_interaction}]
    \textbf{Step 1.}
    We prove that $(X_t)_{t \in [0,T]}$, solution of~\eqref{eq:linear_interaction}, also solves the time-dependent Ornstein-Uhlenbeck equation
    \begin{align}
        \label{eq:time-dep-OU}
        \d X_t = 2 t (m - X_t) \, \d t + \d B_t \quad \text{with}\quad X_0\sim \mathcal{N}(m,1). 
    \end{align}
The solution of \eqref{eq:linear_interaction}, given that $b$ satisfies Assumption \ref{assump:lipschitz}, can be obtained by a fixed-point argument, see \cite[Proof of Theorem 1.1]{Bernou_2023_v1} and \cite{Djete_2022}. Let us thus define $X^{(0)}_t = Z$ for all $t$ in $[0,T]$, with $Z \sim \mathcal{N}(m,1)$, write $(\mu^{(0)}_t)_{t \in [0,T]}$ for the corresponding marginal distributions and define, for $k \ge 0$, the process $(X^{(k+1)}_t)_{t \in [0,T]}$ as the solution of
	\begin{align*}
    \left\{ 
    \begin{array}{ll}
		\d X^{(k+1)}_t = 2 \int_0^t \big[ \int_{\R^d} \big(x - X^{(k)}_t \big) \mu^{(k)}_s(\d x) \big] \d s\, \d t + \d B_t, \qquad t \in [0,T], \\ 
        X^{(k+1)}_t \sim \mu^{(k+1)}_t, \qquad t \in [0,T], \\
        X^{(k+1)}_0 \sim \mathcal{N}(m,1).
        \end{array}
        \right.
	\end{align*}
By induction, for all $k \ge 0$, $t \in [0,T]$,  
\begin{align}
	\label{eq:exp_OU}
	\mathbb{E}\Big[X_t^{(k)}\Big] = m. 
	\end{align} 
Indeed, it is obviously true for $k = 0$, and by It\^o's formula, 
\begin{align*}
	\mathbb{E}\big[ X^{(k+1)}_t \big] &= \mathbb{E}\big[ X^{(k+1)}_0 \big] + 2\int_0^t \Big[ \int_0^s \mathbb{E} \big[ X^{(k)}_u \big] \d u - s\mathbb{E} \big[ X_s^{(k)} \big] \Big] \d s = m.
\end{align*} 
By the fixed-point argument in $L^p(\mathbb{P})$ for some $p \geq 2$ (say $p = 3$), $X^{(k)} \to X$ in $L^p(\mathbb{P})$, where $X$ solves
\begin{align*}
	\d X_t = 2 \Big\{ \int_0^t \mathbb{E}\big[X_s\big] \d s - t\, X_t \Big\} \d t + \d B_t, \qquad X_0 \sim \mathcal{N}(m, 1), 
\end{align*}
and by \eqref{eq:exp_OU} and the fact that the convergence in $L^3(\mathbb{P})$ implies the convergence in $L^1(\mathbb{P})$, $\mathbb{E}[X_t] = m$ for all $t \in [0,T]$. Hence, $(X_t)_{t \in [0,T]}$ solves \eqref{eq:time-dep-OU}.

\medskip 

\noindent\textbf{Step 2.} \textit{Conclusion.} 
Using that the solution of the differential equation
\begin{align*}
    y'(t) = 2t (m - y(t)), \qquad y(0) = x
\end{align*}
is given by $y(t) = m - m e^{-t^2} + x e^{-t^2}$, we find, for instance using \cite[Section 3.3]{Knaeble_2011} that a mild solution to \eqref{eq:time-dep-OU} is given by
\begin{align*}
    X_t = (X_0 - \mu) e^{-t^2} + \mu + \int_0^t e^{-(t^2 - r^2)} \, \d B_r. 
\end{align*}
From this explicit form, the BDG inequality (see Lemma \ref{BDGin}) implies that $X_t$ belongs to $L^p(\mathbb{P})$ for all $p \ge 2$ and the time continuity is immediate. Thus, the mild solution is also a strong solution to \eqref{eq:time-dep-OU} and thus a strong solution to \eqref{eq:linear_interaction}. The uniqueness then follows from Theorem \ref{thm:well-posedness}. 
\end{proof}
\dd 


\begin{proof}[Proof of Proposition \ref{prop:ex1}]
We rewrite the system \eqref{eq:Jansen_Rit_mean_field} in the form of a path-dependent McKean-Vlasov stochastic differential equation, as in \eqref{eq:path-dependent_McKean}. For any $\bar V = (\bar V_1, \bar V_2, \bar V_3)$, the system can be written as
\begin{align*}
\left\{
\begin{array}{l}
\mathrm{d} \bar V(t) = b\big(t,\bar V(\cdot \wedge t), \bar \mu_{\cdot \wedge t} \big) \, \mathrm{d}t + \sigma(t) \, \mathrm{d} \bar W_t, \qquad t \in [\Delta, T], \\
\bar V(t) = \bar V(0), \qquad t \in [0,\Delta],
\end{array}
\right.
\end{align*}
where $\bar{\mu}_t$ denotes the law of $\bar V(t)$, and $(\bar W_t)_{t \ge 0}$ is a standard Brownian motion in $\RR^3$. The diffusion coefficient is given by $\sigma(t) = \mathrm{diag}(f_1(t), f_2(t), f_3(t))$, and the drift $b = (b_1, b_2, b_3)$ is defined component-wise as follows:
\begin{flalign} 
&\forall \,(t, (\alpha_t)_{t\in[0,T]}, (\mu_t)_{t\in[0,T]})\in [0,T]\times \mathcal{C}\big([0,T], \R^3\big)\times \mathcal{C}\big([0,T], \calP_p(\R^3)\big), \;j\in\{1,2,3\}, &\nonumber \\
& \qquad b_j(t, (\alpha_t)_{t\in[0,T]}, (\mu_t)_{t\in[0,T]}) := - \frac{\alpha_t^j}{\tau_j} + \sum_{k=1}^3 D_{j,k} \Big( 1 + \varepsilon \int_0^t  \varphi(\alpha_s^j) \, \d s \Big) \int_{\R^3} S(y) \mu_{t-\Delta}^k(\d y)+ I_j(t).& \nonumber
\end{flalign} 
Here, for each $j \in \{1,2,3\}$ and $t \in [0,T]$, we define $\alpha_t^j := \pi^j(\alpha_t)$ and $\mu_t^j := \mu_t \circ (\pi^j)^{-1}$, where $\pi^j: \RR^3 \to \RR$ denotes the canonical projection onto the $j$-th coordinate, that is, 
$\pi^j(x_1, x_2, x_3) = x_j$. 

Note that, for each $j \in \{1,2,3\}$, the function $b_j(t, (\alpha_t)_{t \in [0,T]}, (\mu_t)_{t \in [0,T]})$ is Lipschitz continuous in the first argument $t$, since the external input functions $I_j$ are assumed to be Lipschitz continuous by Assumption~\ref{assump:old-assum-2.1}. Moreover, $b_j$ is Lipschitz continuous with respect to the path $(\alpha_t)_{t \in [0,T]}$, owing to the Lipschitz continuity of the function $\varphi$ and the boundedness of the function $S$ defined in \eqref{eq:def_S}. Finally, $b_j$ is Lipschitz continuous with respect to $(\mu_t)_{t \in [0,T]}$ under the metric $\sup_{t \in [0,T]} \mathcal{W}_1(\cdot, \cdot)$, and hence also with respect to $\sup_{t \in [0,T]} \mathcal{W}_p(\cdot, \cdot)$ for any $p \geq 1$, since $S$ defined in \eqref{eq:def_S}  is Lipschitz continuous and $\varphi$ is assumed to be bounded (see also Assumption~\ref{assump:old-assum-2.1}).
\end{proof}

\subsection{Proofs of Subsection \ref{subsec:prop_im}}\label{proofsubsec:prop_im}

\begin{proof}[Proof of Lemma \ref{combcovprop}] 
	
	Let $X, Y$ be such that $P_X = \mu$, $P_Y = \nu$ and consider another random variable $U$ having uniform distribution on $[0,1]$, independent of $(X, Y)$. One can easily check that for all $\lambda \in [0,1]$, the random variable $\mathbbm{1}_{\{U \le \lambda\}} X + \mathbbm{1}_{\{U > \lambda \} } Y$ follows the distribution $\tau(\lambda).$
	
\noindent $(a)$ Let $\lambda_1, \lambda_2\in[0,1]$. We assume without loss of generality that $\lambda_1<\lambda_2$. We have 
\begin{align}
&\mathcal{W}_{p}^{\,p}\big( \tau(\lambda_1), \tau(\lambda_2)\big)\leq \EE \left[ \;\Big| \mathbbm{1}_{\{U\leq \lambda_1\}}X + \mathbbm{1}_{\{U> \lambda_1\}}Y - \mathbbm{1}_{\{U\leq \lambda_2\}}X -  \mathbbm{1}_{\{U> \lambda_2\}}Y\Big|^{p}\;\right]\nonumber\\
&\quad = \EE \left[ \Big|  - \mathbbm{1}_{\{\lambda_1<U\leq \lambda_2\}}X + \mathbbm{1}_{\{\lambda_1<U\leq \lambda_2\}}Y\Big|^{p}\right] = \EE \left[ \mathbbm{1}_{\{\lambda_1<U\leq \lambda_2\}}\big|X -Y\big|^{p}\right]\nonumber\\
&\quad = (\lambda_2-\lambda_1)\EE [|X-Y|^{p}].\nonumber
\end{align}
Taking the infimum over $(X,Y) \in \Pi(\mu, \nu)$, we find $\mathcal{W}_{p}\big( \tau(\lambda_1), \tau(\lambda_2)\big)\leq(\lambda_2-\lambda_1)^{\frac{1}{p}}\mathcal{W}_p(\mu, \nu),$
where $\mathcal{W}_p(\mu, \nu)$ is finite since $\mu, \nu \in \mathcal{P}_p(\R^d)$. This concludes the proof of (a).
	
\smallskip 
	
\noindent $(b)$ For every fixed $\lambda \in[0,1]$, 
\begin{align}
&\mathcal{W}_p^{\,p}\big( \tau(\lambda), \delta_0\big)=\EE \Big[ \big| X\mathbbm{1}_{\{U\leq \lambda\}}+Y\mathbbm{1}_{\{U> \lambda\}}\big|^{p}\Big]\nonumber\\
&\quad =\EE \Big[ \big| X\mathbbm{1}_{\{U\leq \lambda\}}+Y\mathbbm{1}_{\{U> \lambda\}}\big|^{p}\mathbbm{1}_{\{U\leq \lambda\}}\Big]+\EE \Big[ \big| X\mathbbm{1}_{\{U\leq \lambda\}}+Y\mathbbm{1}_{\{U> \lambda\}}\big|^{p}\mathbbm{1}_{\{U> \lambda\}}\Big]\nonumber\\
&\quad =\EE \Big[ \big| X\big|^{p}\mathbbm{1}_{\{U\leq \lambda\}}\Big]+\EE \Big[ \big| Y\big|^{p}\mathbbm{1}_{\{U> \lambda\}}\Big]=\lambda \EE \big[ | X |^{p}\big] + (1-\lambda)\EE \big[ | Y |^{p} \big]\nonumber\\
&\quad \le \lambda \mathcal{W}_p^{p}(\mu, \delta_0)+ (1-\lambda)\mathcal{W}_p^{p}(\nu, \delta_0) \leq \mathcal{W}_p^{p}(\mu, \delta_0) \vee \mathcal{W}_p^{p}(\nu, \delta_0). \nonumber
\end{align}
Then we can conclude since the previous inequality is true for every $\lambda \in[0,1]$.
\end{proof}

\begin{proof}[Proof of Lemma \ref{interpolatorprop}]$(a)$ First, it is obvious that $\sup_{0\leq k\leq m}|x_k|\leq \big\Vert i_m (x_{0:m}) \big\Vert_{\sup}$ by the definition of $i_m$. For every $k\in\{0, ..., m-1\}$, for every $t\in[t_k, t_{k+1}]$, we have
\[\big|i_{m}(x_{0:m})_{t}\big|\leq |x_k|\vee |x_{k+1}| \leq \sup_{0\leq k\leq m}|x_k|\]
and for every $t\in[t_m, T]$, we have $\big|i_{m}(x_{0:m})_{t}\big|=x_{m}\leq \sup_{0\leq k\leq m}|x_k|$. Then we can conclude $\sup_{0\leq k\leq m}|x_k|= \big\Vert i_m (x_{0:m}) \big\Vert_{\sup}$.
	
\smallskip
\noindent $(b)$ First, it is obvious that   $\sup_{t\in[0,T]}\mathcal{W}_p\big( i_m(\mu_{0:m})_t, \delta_0\big)\geq\sup_{0\leq k\leq m}\mathcal{W}_p(\mu_k, \delta_0)$ by the definition of $i_m$.
For every $k\in\{0, ..., m-1\}$, Lemma \ref{combcovprop}-$(b)$ implies that 
\begin{align}
&\sup_{t\in[t_k, t_{k+1}]}\mathcal{W}_p\big(i_{m}(\mu_{0:m})_{t}, \delta_0\big)\leq \mathcal{W}_{p}(\mu_k, \delta_0)\vee\mathcal{W}_{p}(\mu_{k+1}, \delta_0) \leq \sup_{0\leq k\leq m}\mathcal{W}_p (\mu_k, \delta_0)\nonumber
\end{align} 
and $\sup_{t\in[t_m, T]}\mathcal{W}_p\big(i_{m}(\mu_{0:m})_{t}, \delta_0\big)= \mathcal{W}_{p}(\mu_m, \delta_0)\leq \sup_{0\leq k\leq m}\mathcal{W}_p (\mu_k, \delta_0).$
Then we can conclude that  $\sup_{t\in[0,T]}\mathcal{W}_p\big( i_m(\mu_{0:m})_t, \delta_0\big)=\sup_{0\leq k\leq m}\mathcal{W}_p(\mu_k, \delta_0)$.
\end{proof}

\begin{proof}[Proof of Lemma \ref{lem:inter}] We only need to prove $(a)$ and $(b)$, from which $(c)$ and $(d)$ can be directly obtained through Definition \ref{definterpolator}. 

\noindent $(a)$ For every $\lambda \in[0,1]$, 
\[|x_\lambda-y_\lambda|\leq\lambda|x_1-y_1|+(1-\lambda)|x_2-y_2|\leq \max\big(|x_1-y_1|, |x_2-y_2|\big). \]

\noindent $(b)$ Let $X_1\sim \mu_1, X_2\sim\mu_2, Y_1\sim \nu_1, Y_2\sim\nu_2$. Let $U\sim \mathcal{U}([0,1])$ independent of $(X_1, X_2, Y_1, Y_2)$. Then 
$\mathbbm{1}_{\{U\leq \lambda\}}X_1+\mathbbm{1}_{\{U>\lambda\}}X_2\sim \lambda \mu_1 +(1-\lambda)\mu_2$ and $\mathbbm{1}_{\{U\leq \lambda\}}Y_1+\mathbbm{1}_{\{U>\lambda\}}Y_2\sim \lambda \nu_1 +(1-\lambda)\nu_2.$
Hence, 
\begin{align}\label{ineq:wpinter}
\mathcal{W}^p_p(\mu_\lambda, \nu_\lambda)&\leq \EE \big[\big(\mathbbm{1}_{\{U\leq \lambda\}}(X_1- Y_1)+\mathbbm{1}_{\{U>\lambda\}}(X_2-Y_2)\big)^p\big]\nonumber\\
&= \EE \big[\big(\mathbbm{1}_{\{U\leq \lambda\}}(X_1- Y_1)^p\big]+\EE \big[\mathbbm{1}_{\{U>\lambda\}}(X_2-Y_2)\big)^p\big]\nonumber\\
&=\PP(U\leq \lambda)\EE \big[(X_1- Y_1)^p\big]+\PP(U> \lambda)\EE \big[(X_2-Y_2)\big)^p\big]\quad \text{(as $U\independent(X_1, X_2, Y_1, Y_2)$)}\nonumber\\
&=\lambda \EE \big[(X_1- Y_1)^p\big]+(1-\lambda)\EE \big[(X_2-Y_2)^p\big]. 
\end{align}
The inequality \eqref{ineq:wpinter} is true for every couplings $(X_1, Y_1)$ and $(X_2, Y_2)$. Taking the infimum over all the couplings of $\mu_1$ and $\nu_1$ (that is, on $\Pi(\mu_1,\nu_1)$ from \eqref{eq:def_Wasserstein}) and on $\Pi(\mu_2,\nu_2)$, the inequality \eqref{ineq:wpinter} gives
\[ \mathcal{W}^p_p(\mu_\lambda, \nu_\lambda)\leq \lambda \mathcal{W}^p_p(\mu_1, \nu_1) + (1-\lambda)\mathcal{W}^p_p(\mu_2, \nu_2). \]
We conclude by using 
\begin{align}
\mathcal{W}^p_p(\mu_\lambda, \nu_\lambda)&\leq \lambda \mathcal{W}^p_p(\mu_1, \nu_1) + (1-\lambda)\mathcal{W}^p_p(\mu_2, \nu_2)\leq \max \big(\mathcal{W}^p_p(\mu_1, \nu_1) , \mathcal{W}^p_p(\mu_2, \nu_2)\big). \hfill \qedhere\nonumber
\end{align} 
\end{proof}

\begin{proof}[Proof of Lemma \ref{lineargrowth}]
	Let $\delta_{0, [0,T]}\in\CPP$ be such that $ \delta_{0, [0,T]}(t)=\delta_0$ for all $t \in [0,T]$ and let $\textbf{0}\in\CRD$ be such that for all  $t\in[0,T],\; \textbf{0}(t)=0$. Then 
\begin{align}
&\Big|b\big(t, \alpha, (\mu_{t})_{t\in[0,T]}\big)\Big|-\Big|b\big(t, \textbf{0}, \delta_{0, [0,T]}\big)\Big|\leq \Big|b\big(t, \alpha, (\mu_{t})_{t\in[0,T]}\big)-b\big(t, \textbf{0}, \delta_{0, [0,T]}\big)\Big|\nonumber\\
&\qquad \leq L \Big(\Vert \alpha\Vert_{\sup}+\sup_{t\in[0,T]}\mathcal{W}_{p}(\mu_t, \delta_0)\Big).\nonumber
\end{align}
Consequently,  
\begin{align}
&\Big|b\big(t, \alpha, (\mu_{t})_{t\in[0,T]}\big)\Big|  \leq \Big(\sup_{t\in[0,T]}\big|b\big(t, \textbf{0}, \delta_{0, [0,T]}\big) \big|\vee L\Big) \Big(\Vert \alpha\Vert_{\sup}+\sup_{t\in[0,T]}\mathcal{W}_{p}(\mu_t, \delta_0) + 1 \Big).\nonumber
\end{align}
Similarly, we have 
\begin{align}
&\vertiii{\sigma\big(t, \alpha, (\mu_{t})_{t\in[0,T]}\big)}  \leq \Big(\sup_{t\in[0,T]}\vertiii{\sigma\big(t, \textbf{0}, \delta_{0, [0,T]}\big) }\vee L\Big) \Big(\Vert \alpha\Vert_{\sup}+\sup_{t\in[0,T]}\mathcal{W}_{p}(\mu_t, \delta_0) + 1\Big)\nonumber
\end{align}
so that one can take $\displaystyle C_{b, \sigma, L, T}\coloneqq \sup_{t\in[0,T]}\big|b\big(t, \textbf{0}, \delta_{0, [0,T]}\big) \big|\;\vee\; \sup_{t\in[0,T]}\vertiii{\sigma\big(t, \textbf{0}, \delta_{0, [0,T]}\big) }\;\vee\; L$ to conclude. 
\end{proof}

\begin{proof}[Proof of Lemma \ref{lem:BDG_for_sigma}]
	Notice first that it follows from Lemma \ref{BDGin} that $\int_{0}^{\cdot}H_{s}\d B_{s}$ is a $d$-dimensional local martingale satisfying
	\begin{equation}\label{BDGinequality}
		\left\Vert\sup_{s\in[0, t]}\left|\int_{0}^{s}H_{u} \, \d B_{u}\right|\right\Vert_{p}\leq C_{d,p}^{BDG}\left\Vert\sqrt{\int_{0}^{t}\vertiii{H_u}^{2}\d u}\right\Vert_{p}.
	\end{equation}
	Applying this, and using that when $U \ge 0$, $\big\| \sqrt{U} \big\|_p = \big\|U \big\|_{\tfrac{p}{2}}^{\tfrac12}$, we obtain 
	\begin{align*}
		\Big\| \sup_{s \in [0,t]} \Big| \int_0^s H_u \, \d B_u \Big| \Big\|_p &\le C_{d,p}^{BDG}\left\Vert \int_{0}^{t}\vertiii{H_u}^{2} \d u \right\Vert_{\tfrac{p}{2}}^{\tfrac12} \le C_{d,p}^{BDG} \Big[ \int_0^t \left\Vert \vertiii{H_u}^2 \right \Vert_{\tfrac{p}{2}} \d u  \Big]^{\tfrac12}
	\end{align*}
	where we used Minkowski's inequality (recall that $p \ge 2$) to obtain the last inequality. The proof follows by noticing that $\| |U|^2 \|_{\tfrac{p}{2}} = \|U\|_p^2$. 
\end{proof}

\begin{proof}[Sketch of proof of Theorem~\ref{thm:sup-W2-in-MR1329874}]
Let $ P \ge 1$ and set $t_k = k/P$ for all $k \in \{0,\ldots, P\}$. For all $0 \le k \le P-1$, set $Z_k := \sup_{t_k \le t \le t_{k+1}} (\calW^p_p(\mu^n_t, \mu^n_{t_k}) \wedge \calW^p_p(\mu^n_t, \mu^n_{t_{k+1}}) )$. 
Then
\begin{align}
\label{eq:HK_init_decompo}
    \sup_{0 \le t \le 1} \calW^{p}_p(\mu^n_t, \mu_t) \le C_p \Big( \max_k Z_k + \max_k  \calW^p_p(\mu^n_{t_k}, \mu_{t_k}) + \max_k \sup_{t_k \le t \le t_{k+1}} \calW^p_p(\mu_{t_k}, \mu_t) \Big).
\end{align}
We treat the three terms separately in the next steps.

\medskip 

\textbf{Step 1.} By continuity and definition of $\calW_p$, we immediately deduce from the assumptions, for all $k \ge 1$, $t \in [t_k, t_{k+1}]$, using conditions (d)-(e) and a straightforward interpolation
\begin{align}
\label{eq:HK_control_1}
    \calW^p_p(\mu_{t_k}, \mu_t) \le \EE[|\xi_{t_k} - \xi_t|^p] \le C |t_k - t| \le C/P.  
\end{align}

\medskip 

\textbf{Step 2.} We claim that for all $\theta \in (0,1)$,
\begin{align}
    \label{eq:HK_control_2}
    \EE \big[ \max_k \calW^p_p(\mu^n_{t_k}, \mu_{t_k}) \big] \le C_{d,p} \Big(\theta^p + \theta^{-d/2} \sqrt{\frac{P}{n}} \Big)
\end{align}
for some constant $C_{d,p} > 0$. 
To see this, we first introduce, for any $\theta \in(0,1)$, convolution with a random vector $Y_\theta \sim \mathcal{N}(0,\theta^2 I_d)$, whose density distribution is denoted by $\phi_\theta$. For any $\mu \in \calP(\R^d)$, let $\mu^\theta = \mathcal{L}(Y_\theta) \ast \mu$ be the convolution, whose density function is given by $\phi_\theta \ast \mu$. For any $\mu \in \calP(\R^d)$, considering the coupling $(X,X+Y)$ with $X\sim \mu$ and $Y_{\theta} \sim \mathcal{N}(0,\theta^2 I_d)$, we have
\begin{align*}
    \calW^p_p(\mu^\theta, \mu) \le \EE[|Y_{\theta}|^p] \le C_{d,p} \theta^p,
\end{align*}
by properties of Gaussian random variables. Allowing the value of $C_{d,p} > 0$ to change from line to line (but only depending on $d,p$), we get
\begin{align}
\label{eq:HK_step2_temp}
    \EE \big[ \max_k \calW^p_p(\mu^n_{t_k}, \mu_{t_k}) \big] \le C_{d,p}\Big(\theta^p + \EE \big[ \max_k \calW^p_p(\mu^{n,\theta}_{t_k}, \mu_{t_k}^\theta) \big] \Big),
\end{align}
where $\mu^{n,\theta}_{t_k} = \mathcal{L}(Y_{\theta}) \ast \mu^n_{t_k}$ for all $1 \le k \le P$. It remains to control the last term on the right-hand-side of~\eqref{eq:HK_step2_temp}. We split the proof into two further substeps.

\medskip 

\textbf{Step 2.1.} Proof that for all $\mu, \nu \in \calP_p(\R^d)$ such that $\mu$ (resp. $\nu$) admits a density $f$ (resp. $g$) with respect to the Lebesgue measure on $\R^d$, 
\begin{align}
\label{eq:HK_control_2_temp_2}
    \calW^p_p(\mu, \nu) \le C_p \int_{\R^d} |x|^p |f(x) - g(x)| \, \d x. 
\end{align}
To see this, consider the maximum coupling $R_{\mu, \nu}$ so that for any $\psi$ Lipschitz bounded in $\R^d \times \R^d$,
\begin{align*}
    \int_{\R^d} \psi(x,y) R_{\mu,\nu}(\d x, \d y) &= \tfrac1{1-A} \int_{\R^d \times \R^d} \psi(x,y) \big(f(x) - f \wedge g(x) \big) \big(g(y) - f \wedge g(y)\big) \ddr x \ddr y \\
    &\qquad + \int_{\R^d} \psi(x,x) (f \wedge g)(x) \, \ddr x
\end{align*}
with $A = \int_{\R^d} (f \wedge g)(x) \ddr x$. 
Choosing $\psi(x,y) = |x-y|^p$ we immediately get 
\begin{align*}
    \calW^p_p(\mu,\nu) \le C_p \int_{\R^d} |x|^p |f(x) - f \wedge g(x)| \, \d x
\end{align*}
from which~\eqref{eq:HK_control_2_temp_2} follows. 

\medskip 

\textbf{Step 2.2.} Proof that for all $\theta \in (0,1)$,
\begin{align}
    \label{eq:HK_control_2_temp_3}
    \EE \Big[ \max_k \calW^p_p(\mu^{n,\theta}_{t_k}, \mu^{\theta}_{t_k}) \Big] \le C_p \theta^{-d/2} \sqrt{\frac{P}{n}}. 
\end{align}
We recall Pollard's lemma (see e.g.~\cite[Lemma 2.3]{MR1329874}): for any random vectors $V_1, \ldots, V_P$, 
\begin{align}
\label{eq:Pollard}
    \EE \Big[ \max_{1 \le k \le P} |V_k| \Big] \le \sqrt{P} \sqrt{\max_k \EE[ |V_k|^2]}.
\end{align}
Applying this with $V_k = \calW_p^p(\mu^{n,\theta}_{t_k}, \mu^\theta_{t_k})$, $1 \le k \le P$, we get
\begin{align}
\label{eq:Pollard_applied}
    \EE \Big[ \max_k \calW^p_p(\mu^{n,\theta}_{t_k}, \mu^{\theta}_{t_k}) \Big] \le \sqrt{P} \sqrt{\max_k \EE \calW_p^{2p}(\mu^{n,\theta}_{t_k}, \mu^\theta_{t_k})}.
\end{align}
By Cauchy-Schwartz inequality,
\begin{align*}
    \int_{\R^d} |x|^p |f - g|(x) \, \d x \le C_{d,p} \Big( \int_{\R^d} (|x|^{d + 2p + 1} + 1) |f-g|^2(x) \, \d x \Big)^{\frac12}, 
\end{align*}
so that, injecting~\eqref{eq:HK_control_2_temp_2} with $\mu = \mu^{n,\theta}_{t_k}$ and $\nu = \mu^\theta_{t_k}$, denoting their respective density by $g^{n,\theta}_{t_k}$ and $g^\theta_{t_k}$, we get 
\begin{align*}
    \EE\big[ \calW_p^{2p}(\mu^{n,\theta}_{t_k}, \mu^\theta_{t_k}) \big] \le C_{d,p} \int_{\R^d} (|x|^{d+2p+1}+1) \EE \big[ |g^{n,\theta}_{t_k}(x) - g^{\theta}_{t_k}(x)|^2 \big] \, \d x.
\end{align*}
Noticing that $\EE[g^{n,\theta}_{t_k}(x)] = g^\theta_{t_k}(x)$ by definition of $\mu^n$, we get 
\begin{align*}
    \EE \big[ |g^{n,\theta}_{t_k}(x) - g^{\theta}_{t_k}(x)|^2 \big] = \frac1{n} \mathrm{Var}\big[ \phi_{\theta}(x-\xi_{t_k}) \big]
\end{align*}
where $\xi_{t_k} \sim \mu_{t_k}$, yielding 
\begin{align}
\label{eq:HK_control_2_temp_4}
    \EE\big[ \calW_p^{2p}(\mu^{n,\theta}_{t_k}, \mu^\theta_{t_k}) \big] \le \frac{C_{d,p}}{n}  \int_{(\R^d)^2} (|x|^{d+2p+1}+1) \phi_\theta^2(x-y) \mu_{t_k}(\d y) \, \, \d x.
\end{align}
Using $\phi_{\theta}^2(x) = (4\pi)^{-d/2} \theta^{-d} \phi_{\theta/\sqrt{2}}(x)$, a straightforward computation shows 
\begin{align*}
    \int_{(\R^d)^2} (|x|^{d+2p+1}+1) \phi_\theta^2(x-y) \mu_{t_k}(\d y) \big] \, \d x \le C_{d,p} \, \theta^{-d}. 
\end{align*}
Combining this with~\eqref{eq:HK_control_2_temp_4} and injecting it inside~\eqref{eq:Pollard_applied} concludes the proof of~\eqref{eq:HK_control_2_temp_3}.

\medskip

\textbf{Conclusion of Step 2.} Gathering~\eqref{eq:HK_control_2_temp_2} and~\eqref{eq:HK_control_2_temp_3} immediately entails~\eqref{eq:HK_control_2}. 

\medskip

\textbf{Step 3.} We claim that for all $\epsilon \in( 0,1)$ small enough so that $(1+\epsilon)p \le q$,
\begin{align}
\label{eq:HK_control_3}
    \EE[\max_k Z_k] \le C_{d,p,q} P^{\frac{\varepsilon-1}{2(1+\varepsilon)}} = C_{d,p,q} P^{\frac12 - \frac1{1+\epsilon}}.
\end{align}
Here we simply sketch the argument, referring to~\cite[p. 269]{MR1329874} for details, as the proof from the case $p = 2$ can be easily adapted.  
We first write for all $s, t \in[0,1]$, 
\[ Y^n_{s,t} = \frac1{n} \sum_{i=1}^n |\xi_s^i\dd - \xi_t^i|^p. \] 
Letting $t \in (t_1,t_2)$, one shows first that for all $r > 1$ such that $rp \le q$, we have by convexity
\begin{align}\label{ineq:markov1}
    \EE[(Y^n_{t,t_1})^r (Y^n_{t,t_2})^r] \le \EE \left[ \Big(\frac1{n} \sum_{i=1}^n |\xi_t^i- \xi_{t_1}^i|^{rp} \Big) \Big(\frac1{n} \sum_{i=1}^n |\xi_t^i-\xi_{t_2}^i|^{rp} \Big)\right] \le C |t_2 - t_1|^2
\end{align}
using conditions (b), (c) and a direct interpolation, mimicking the proof in~\cite{MR1329874}.
Going through dyadic decomposition and using the continuity, one bounds the supremum by the sum on small interval, and using 
\[ \big\{ \min(Y^n_{t,t_1}, Y^n_{t,t_2}) \ge \lambda \big\} \subset \big\{ (Y^n_{t,t_1})^{r} (Y^n_{t,t_2})^r \ge \lambda^{2r} \big\} \] 
we get for all $r > 1$ so that $rp \le q$ 
\begin{align}\label{ineq:markov2}
    \mathbb{P} \Big( \sup_{t_1 \le t \le t_2} \min \big( Y^n_{t,t_1}, Y^n_{t,t_2} \big) \ge \lambda \Big) \le C |t_2 - t_1|^2 \lambda^{-2r}.
\end{align}
Using then, for all $\lambda_* > 0$,
\begin{align*}
    \EE[Z^2] \le \lambda_*^2 + \int_{\lambda_*}^\infty \lambda \mathbb{P}(Z \ge \lambda) \, \d \lambda,
\end{align*}
we find, applying this to $\Xi= \sup_{t_1 \le t \le t_2} \min(Y_{t,t_1}^n, Y_{t,t_2}^n)$, 
\begin{align*}
    \EE \Big[ \sup_{t_1 \le t \le t_2} \min(Y_{t,t_1}^n, Y_{t,t_2}^n)^2 \Big] \le \lambda_*^2 + C_r \lambda_*^{2-2r} |t_2 - t_1|^2
\end{align*}
where $C_r=\frac{C}{r-1}$. Remark that $C_r\rightarrow +\infty$ when $r\rightarrow 1^+$.  Therefore, 
\dd
the choice $\lambda_* = |t_2 - t_1|^{\frac1{r}}$ and $r = 1 + \epsilon$ with $\epsilon$ small enough yields
\begin{align*}
  \max_k \EE \big[ Z_k^2 \big] \le C_{d,p,\epsilon} P^{-\frac{2}{1+\epsilon}}. 
\end{align*}

The conclusion~\eqref{eq:HK_control_3} follows then from Pollard's lemma~\eqref{eq:Pollard}.

\medskip 

\textbf{Conclusion.} Gathering the results of Steps 1-3, we get for all $\epsilon \ll 1$ such that $(1+\epsilon)p \le q$, using $P^{-1} \le P^{\tfrac12 - \tfrac1{1+\epsilon}}$ for all $P \ge 1$,
\begin{align*}
   \EE\Big[ \sup_{0 \le t \le 1} \calW_p^p(\mu^n_t,\mu_t) \Big] \le C_{d,p,\epsilon} \Big( P^{\frac12 - \frac{1}{1+\epsilon}} + \theta^p + \theta^{-d/2} \sqrt{\frac{P}{n}} \Big).
\end{align*}
Choosing $P = n^{\frac{2p(1+\epsilon)}{d(1-\epsilon) + 4p}}$ and $\theta = n^{-\frac{1-\epsilon}{d(1-\epsilon) + 4p}}$  yields
\begin{align*}
    \EE\Big[ \sup_{0 \le t \le 1} \calW_p^p(\mu^n_t,\mu_t) \Big] \le C_{d,p,\epsilon} \, n^{-\frac{p(1 - \epsilon)}{d(1-\epsilon) + 4p}}.
\end{align*}
Finally, for all $\delta > 0$, choosing $\epsilon < \min(\frac{1}{2},\frac{\delta (d+4p)}{p})$, we have 
\[ \frac{p(1-\epsilon)}{d(1-\epsilon)+4p} \ge \frac{p(1-\epsilon)}{d+4p} \ge \frac{p}{d+4p} - \delta \]
so that
\[ n^{-\frac{p(1 - \epsilon)}{d(1-\epsilon) + 4p}} \le n^{-\frac{p}{d+4p} + \delta}, \]
entailing the result with a multiplicative constant depending on $d,p,\delta$. 
\end{proof}

\end{appendix}

\noindent \textbf{Acknowledgment.} A.B. thanks Corentin Bernou for fruitful discussions regarding the neuroscience behind the model in Section \ref{subsec:neural}. 

\noindent \textbf{Funding.}  A.B. acknowledges funding from the European Union's Horizon 2020 research and innovation programme under the Marie Skłodowska-Curie Grant Agreement No 101034324 and by AAP ``Accueil EC'' of Université Claude Bernard Lyon 1. Y. L. acknowledges partial funding from Institut National des Sciences Math\'ematiques et de leurs Interactions (INSMI) through the PEPS JCJC 2023 program and financial support from the CNRS through the MITI interdisciplinary programs.

\bibliographystyle{plain}
\bibliography{Path_depen_version_merge}       

\end{document}